\documentclass{amsart}
\usepackage{amsmath, amssymb, graphicx, amsthm, pst-all}

\numberwithin{equation}{section}

\theoremstyle{plain}
\newtheorem{prop}{Proposition}[section]
\newtheorem{algo}[prop]{Algorithm}

\newtheorem{coro}[prop]{Corollary}
\newtheorem{fact}[prop]{Fact}
\newtheorem{lemm}[prop]{Lemma}
\newtheorem{ques}[prop]{Question}
\newtheorem{thrm}[prop]{Theorem}
\newtheorem*{thrm0}{Theorem 1}
\newtheorem*{claim}{{\it Claim}}

\theoremstyle{definition}
\newtheorem{defi}[prop]{Definition}
\newtheorem{nota}[prop]{Notation}
\newtheorem{rema}[prop]{Remark}
\newtheorem{exam}[prop]{Example}

\psset{unit=1mm}
\psset{dotsep=1.5pt}
\psset{dash=1.5pt 1.5pt}
\psset{linewidth=0.5pt}
\psset{arrowsize=3.5pt}
\psset{doublesep=1pt}
\psset{coilwidth=1.2mm}
\psset{coilheight=2}
\psset{coilaspect=20}
\psset{coilarm=2}
\newpsstyle{exist}{linestyle=dashed}
\newpsstyle{double}{linewidth=0.5pt,doubleline=true}

\newcounter{ITEM}
\newcommand\ITEM[1]{\setcounter{ITEM}{#1}\leavevmode\hbox{\rm(\roman{ITEM})}}

\renewcommand\aa{a}

\newcommand\bb{b}
\newcommand\BB{B}
\newcommand\BP[1]{B_{#1}^{\scriptscriptstyle+}}
\newcommand\BDD[1]{B_{#1}^{\oplus}}
\newcommand\BS{\mathrm{BS}}

\newcommand\cc{c}
\newcommand\CC{C}
\newcommand\cl[1]{[#1]}
\newcommand\clp[1]{[#1]^{\scriptscriptstyle+}}
\newcommand\CCt{\widetilde{C}}
\newcommand\comp{{\scriptscriptstyle\circ}\,}

\newcommand\Deltat{\widetilde{\VR(2.2,0)\smash\Delta}}

\newcommand\disj{\mathbin{\,\amalg\,}}
\renewcommand\div{\prec}
\newcommand\dive{\mathrel{\preccurlyeq}}

\newcommand\ee{e}
\newcommand\EE{E}

\newcommand\eqpR{\equiv_\RR^{\scriptscriptstyle+}}
\newcommand\EqpR[1]{\equiv_\RR^{+(#1)}}
\newcommand\ew{\varepsilon}

\newcommand\ff{f}

\let\ge=\geqslant
\renewcommand\gg{g}
\newcommand\GG{G}
\newcommand\Gr[2]{\langle#1\,\vert\,\nobreak#2\rangle}

\newcommand\hh{h}

\newcommand\ie{{\it i.e.}}
\newcommand\ii{i}
\newcommand\II{I}
\newcommand\inv{^{-1}}

\newcommand\jj{j}

\newcommand\kk{k}

\let\le=\leqslant
\newcommand\letter[1]{\underline{#1}}
\newcommand\Lg[1]{\vert#1\vert}
\newcommand\linfty{{}^\infty\hspace{-0.3ex}}
\newcommand\LO[1]{\mathrm{LO}(#1)}

\newcommand\mm{m}
\newcommand\MM{M}
\newcommand\MMt{\widetilde{M}}
\newcommand\Mon[2]{\langle#1\,\vert\,\nobreak#2\rangle^{\!\scriptscriptstyle+}}
\newcommand\mult{\succ}

\newcommand\multeR{\mathrel{\widetilde{%
\raisebox{5pt}{\hspace{1pt}}\hspace{-1.5pt}\smash\succcurlyeq}}}

\newcommand\nn{n}
\newcommand\NN{N}
\newcommand\NNNN{\mathbb{N}}
\newcommand\NNt{\widetilde{N}}

\newcommand\parity[1]{(-1)^{#1}}
\newcommand\pdots{\hspace{0.2ex}{\cdot}{\cdot}{\cdot}\hspace{0.2ex}}
\newcommand\pp{p}
\newcommand\PP{P}
\newcommand\PPo{P_{\!\!{\scriptscriptstyle<}}}
\newcommand\Pres[2]{(#1\:;#2)}

\newcommand\qq{q}

\newcommand\resp{{\it resp.\ }}

\newcommand\rev{\curvearrowright}
\newcommand\revL{\mathrel{\hspace{2ex}\widetilde{\hspace{2ex}\hspace{-4ex}}{\curvearrowright}}}
\newcommand\revLR{\revL_{\!\RR}}
\newcommand\revLRc{\mathrel{\revL_{\!\RRasmall}}}
\newcommand\Rev[1]{\mathrel{\curvearrowright^{\!(#1)}}}
\newcommand\revR{\mathrel{\curvearrowright_{\!\RR}}}
\newcommand\revRh{\mathrel{\curvearrowright_{\!\RRh}}}
\newcommand\revRopp{\mathrel{\curvearrowright_{\!\widetilde{R}}}}
\newcommand\RevR[1]{\mathrel{\curvearrowright_{\!\RR}^{\!\!(#1)}}}
\newcommand\RevRh[1]{\mathrel{\curvearrowright_{\!\RRh}^{\!\!(#1)}}}
\newcommand\rr{r}
\newcommand\RR{R}
\newcommand\RRh{\widehat\RR}
\newcommand\RRa{\VR(3,0)\smash{\raisebox{3.5ex}{\rotatebox{180}{\hbox{$\widehat{\hspace{2ex}}$}}}\hspace{-2ex}\RR}}
\newcommand\RRasmall{\hspace{0.4ex}\VR(2.3,0)\smash{\raisebox{2.55ex}{\rotatebox{180}{\hbox{$\scriptstyle\widehat{\hspace{1ex}}$}}}\hspace{-1.35ex}{\scriptstyle}\RR}}
\newcommand\RRt{\widetilde\RR}

\newcommand\sig[1]{\sigma_{\!#1}^{\relax}}
\newcommand\siginv[1]{\sigma_{\!#1}^{-1}}
\newcommand\sigg[2]{\sigma_{\!#1}^{#2}}
\renewcommand\ss{s}
\renewcommand{\SS}{S}
\newcommand{\SSh}{\widehat\SS}

\renewcommand\tt{t}
\newcommand\TT{T}
\newcommand\tta{\mathtt{a}}
\newcommand\ttb{\mathtt{b}}
\newcommand\ttc{\mathtt{c}}
\newcommand\ttd{\mathtt{d}}
\newcommand\tte{\mathtt{e}}
\newcommand\ttx{\mathtt{x}}
\newcommand\tty{\mathtt{y}}
\newcommand\ttz{\mathtt{z}}

\newcommand\uu{u}

\def\VR(#1,#2){\vrule width0pt height#1mm depth#2mm}
\newcommand\vv{v}

\newcommand\wdots{, ...\hspace{0.2ex},}
\newcommand\wit{\lambda}
\newcommand\ww{w}
\newcommand\WW{W}

\newcommand\xx{x}

\newcommand\yy{y}

\newcommand\ZZZZ{\mathbb{Z}}

\begin{document}

\author{Patrick DEHORNOY}

\address{Laboratoire de Math\'ematiques Nicolas Oresme, UMR 6139 CNRS, Universit\'e de Caen, 14032 Caen, France}
\email{dehornoy@math.unicaen.fr}
\urladdr{//www.math.unicaen.fr/\!\hbox{$\sim$}dehornoy}

\title{Monoids of $O$-type, subword reversing, and ordered groups}



\begin{abstract}
We describe a simple scheme for constructing finitely generated monoids in which left-divisibility is a linear ordering and for practically investigating these monoids. The approach is based on subword reversing, a general method of combinatorial group theory, and connected with Garside theory, here in a non-Noetherian context. As an application we describe several families of ordered groups whose space of left-invariant orderings has an isolated point, including torus knot groups and some of their amalgamated products.
\end{abstract}

\footnote{Work partially supported by the ANR grant ANR-08-BLAN-0269-02}


\maketitle


A group~$\GG$ is left-orderable if there exists a linear ordering on~$\GG$ that is left-invariant, \ie, $\gg < \gg'$ implies $\hh\gg < \hh\gg'$ for every~$\hh$ in~$\GG$. Viewing an ordering on~$\GG$ as a subset of~$\GG \times \GG$, one equips the family~$\LO\GG$ of all left-invariant orderings of~$\GG$ with a topology induced by the product topology of~$\mathfrak{P}(\GG \times \GG)$. Then $\LO\GG$ is a compact space and, in many cases, in particular when $\GG$ is a countable non-abelian free group, $\LO\GG$ has no isolated points and it is a Cantor set~\cite{Sik, DDHPV}. By contrast, apart from the cases when $\LO\GG$ is finite and therefore discrete, as is the case for the Klein bottle group and, more generally, for the Tararin groups~\cite{Tar, KKM}, not so many examples are known when $\LO\GG$ contains isolated points. By the results of~\cite{DuD}, this happens when $\GG$ is an Artin braid group (see also \cite{Nav}), and, by those of~\cite{Nav2, Ito}, when $\GG$ is a torus knot group, \ie, a group of presentation $\langle \ttx, \tty \mid \ttx^\mm = \tty^\nn \rangle$ with $\mm, \nn \ge 2$. These results, as well as the further results of~\cite{Ito2}, use non-elementary techniques.

The aim of this paper is to observe that a number of ordered groups with similar properties, including the above ones, can be constructed easily using a monoid approach. A necessary and sufficient condition for a submonoid~$\MM$ of a group~$\GG$ to be, when $1$ is removed, the positive cone of a left-invariant ordering on~$\GG$ is that $\MM$ is what will be called \emph{of $O$-type}, namely it is cancellative, has no nontrivial invertible element, and its left-and right-divisibility relations (see Definition~\ref{D:Div}) are linear orderings. Moreover, the involved ordering is isolated in the corresponding space~$\LO\GG$ whenever $\MM$ is finitely generated. We are thus naturally led to the question of recognizing which (finite) presentations define monoids of $O$-type. 

Here we focus on presentations of a certain syntactical type called triangular. Although no complete decidability result can probably be expected, the situation is that, in practice, many cases can be successfully addressed, actually all cases in the samples we tried. The main tool we use here is subword reversing~\cite{Dfa, Dff, Dgp, Dia}, a general method of combinatorial group theory that is especially suitable for investigating divisibility in a presented monoid and provides efficient algorithms that make experiments easy. Both in the positive case (when the defined monoid is of $O$-type) and in the negative one (when it is not), the approach leads to sufficient $\Sigma_1^1$-conditions, \ie, provides effective procedures returning a result when the conditions are met but possibly running forever otherwise. At a technical level, the main new observation is that subword reversing can be useful even in a context where the traditional Noetherianity assumptions fail.

The outcome is the construction of families of finitely generated monoids of $O$-type, hence of ordered groups with isolated points in the space of left-orderings, together with algorithmic tools for analysing these structures. There is a close connection with Garside theory~\cite{Garside} as most of the mentioned examples admit a Garside element. The scheme can be summarized as follows (see Theorem~\ref{T:Main} for a more general version): 

\begin{thrm0}
A sufficient condition for a group~$\GG$ to be orderable is that 
\begin{quote}
$\GG$ admits a (finite or infinite) presentation\\
$(*)$\hfill $\Pres{\tta_1, \tta_2, ... }{\tta_1 = \tta_2 \,\ww_2\, \tta_2, \tta_2 = \tta_3 \,\ww_3\, \tta_3, ...}$\hfill\null\\
where $\ww_2, \ww_3, ...$ are words in $\{{\tta_1, \tta_2, ... }\}$ (no~$\tta_\ii\inv$) and there exists in the monoid presented by~$(*)$ an element~$\Delta$ such that there exist~$\gg_1, \gg'_1, \gg_2, \gg'_2$,... satisfying $\tta_1 \gg_1 = \Delta = \gg'_1 \tta_1$ and $\tta_\ii \Delta = \Delta \gg_\ii$ and $\Delta \tta_\ii = \gg'_\ii \Delta$ for $\ii \ge 2$.
\end{quote}
Then, the subsemigroup of~$\GG$ generated by~$\tta_1, \tta_2$,... is the positive cone of a left-invariant
ordering on~$\GG$. If $(*)$ is finite, this ordering is isolated in the space~$\LO\GG$. If $(*)$ is finite or recursive, the word problem of~$\GG$ and the decision problem of the ordering are decidable.
\end{thrm0}

Among others, the approach applies to the above mentioned torus knot groups, providing a short construction of an isolated ordering, and in particular to the group~$B_3$ of $3$-strand braids, providing one more proof of its orderability. More examples are listed in Table~\ref{T:Recap} and in Section~\ref{S:Fam} below.

\begin{table}[h]
\smaller
\begin{tabular}{ccl}
\hline
\hline

\VR(4,0)
1:&
$\Gr{\ttx,\tty}{\ttx^{\pp+1}{=}\tty^{\qq+1}}$
&$\Pres{\tta, \ttb}{\tta = \ttb(\tta^\pp\ttb)^\qq}$\\ 

&&\hspace{5mm}$\Delta = \tta^{\pp+1}$ central (Proposition~\ref{P:Fam1});\\

\VR(3.5,0)
2:
&$\Gr{\ttx, \tty, \ttz}{\ttx^{\pp+1} = \tty^{\qq+1}, \tty^{\rr + 1} = \ttz^{\ss+1}}$
&$\Pres{\tta, \ttb, \ttc}{\tta = \ttb(\tta^\pp\ttb)^\qq, 
\ttb = \ttc((\tta^\pp\ttb)^{\rr}\tta^\pp\ttc)^{\ss}}$\\

&&\hspace{5mm}$\Delta = \tta^{(\pp+1)(\rr+1)}$ central (Proposition~\ref{P:Fam4});\\

\VR(3.5,0)
3:&$\Gr{\ttx, \tty}{\ttx^{\pp+1} = (\tty (\ttx^{\rr-\pp} \tty)^\ss)^{\qq+1}}$
&$\Pres{\tta, \ttb}{\tta = \ttb(\tta^\rr \ttb)^\ss (\tta^\pp \ttb (\tta^\rr\ttb)^\ss)^\qq}$ with $\rr \ge \pp$\\

&&\hspace{5mm}$\delta = \tta$ dominating (Proposition~\ref{P:Fam5});\\

&&\hspace{5mm}is also $\Gr{\tta, \ttb, \ttc}{\tta = \ttb(\tta^\pp \ttb)^\qq, \ttb = \ttc (\tta^\rr\ttc)^\ss}$;\\

\VR(3.5,0)
4:&$\Gr{\ttx, \tty}{\ttx^{\rr+1} = (\tty\ttx^2 \tty)^{\qq+1}}$
&$\Pres{\tta, \ttb, \ttc}{\tta = \ttb\tta^2(\ttb^2\tta^2)^\qq \ttc, \ttb = \ttc (\ttb \tta^2)^\rr \ttb\tta}$\\

\VR(0,3)
&\rlap{with $\rr = 0$ or $\rr = 1$}\hspace{1.2cm}
&\hspace{5mm}$\Delta = (\ttb\tta^2)^{2\qq + \rr + 3}$ central (Proposition~\ref{P:Fam6}).\\

\hline
\hline
\end{tabular}
\caption{\sf\smaller \VR(4,0)Some groups eligible for the current approach, hence ordered with an isolated point in the space of orderings: on the right, a presentation eligible for Theorem~1 or its extensions and the involved distinguished element~$\Delta$.}
\label{T:Recap}
\end{table} 

The paper is organized as follows. In Section~\ref{S:Otype}, we introduce the notion of a monoid of $O$-type and describe its connection with ordered groups. In Section~\ref{S:Triangular}, we define triangular presentations, raise the central question, namely recognizing when a (right)-triangular presentation defines a monoid of (right)-$O$-type (and therefore leads to an ordered group), and state without proof the main technical result (``Main Lemma''), which reduces the central question to the existence of common right-multiples. Section~\ref{S:Rev} contains a brief introduction to subword reversing, with observations about the particular form it takes in the context of right-triangular presentations. In Section~\ref{S:Complete}, we establish that every right-triangular presentation is what we call complete for right-reversing, and deduce a proof of the Main Lemma. Next, we investigate in Section~\ref{S:Mult} the notions of a dominating and a quasi-central element in a monoid and, putting things together, we obtain the expected sufficient conditions for a presentation to define a monoid of $O$-type. The proof of Theorem~1 is then completed in Section~\ref{S:Groups}. Then, we report in Section~\ref{S:Exp} about some computer investigations and, in Section~\ref{S:Fam}, we describe some examples, in particular those mentioned in Table~\ref{T:Recap}. Finally, in Section~\ref{S:Limits}, we establish some negative results about the existence of triangular presentations, and gather some open questions in~Section~\ref{S:Ques}.

The author thanks A.\,Navas, L.\,Paris, C.\,Rivas, and D.\,Rolfsen for discussions about the subject of the paper.


\section{Monoids of $O$-type}
\label{S:Otype}

If $\GG$ is an orderable group and $<$ is a left-invariant ordering of~$\GG$, the positive cone~$\PPo$ of~$<$, \ie, the set $\{\gg \in \GG \mid \gg > 1\}$, is a subsemigroup of~$\GG$ satisfying $\GG = \PPo \disj \PPo{}\inv \disj \{1\}$. Conversely, if $\PP$ is a subsemigroup of~$\GG$ satisfying $\GG =\nobreak \PP \disj\nobreak \PP\inv \disj\nobreak \{1\}$, then the relation $\gg\inv \hh \in \PP$ defines a left-invariant ordering on~$\GG$ and $\PP$ is the associated positive cone. 

In the sequel, the notions of divisors and multiples will play a central role. It is convenient to consider them in the context of monoids, \ie, semigroups with a unit.

\begin{defi}
\label{D:Div} 
Assume that $\MM$ is a monoid. For $\gg, \hh$ in~$\MM$, we say that $\gg$ is a \emph{left-divisor} of~$\hh$, or, equivalently, $\hh$ is a \emph{right-multiple} of~$\gg$, denoted $\gg \dive \hh$, if there exists~$\hh'$ in~$\MM$ satisfying $\gg \hh' = \hh$. Symmetrically, we say that $\gg$ is a \emph{right-divisor} of~$\hh$, or, equivalently, $\hh$ is a \emph{left-multiple} of~$\gg$, denoted $\hh \multeR \gg$, if there exists~$\hh'$ in~$\MM$ satisfying $\hh = \hh' \gg$. 
\end{defi}
 
For every monoid~$\MM$, left- and right-divisibility are partial preorders on~$\MM$, and they are partial orders whenever $1$ is the only invertible element of~$\MM$. Note that the right-divisibility relation of a monoid~$\MM$ is the left-divisibility relation of the opposite monoid~$\MMt$, \ie, the monoid with the same domain equipped with the operation defined by $\gg \mathbin{\widetilde\cdot} \hh = \hh\gg$.

It is easy to translate the existence of an invariant ordering in a group into the language of monoids and divisibility. We recall that a monoid is called left-cancellative (\resp right-cancellative), for all~$\gg, \gg', \hh$ in the monoid, $\hh\gg = \hh\gg'$ (\resp $\gg\hh = \gg' \hh$) implies~$\gg = \gg'$. A monoid is cancellative if it is both left- and right-cancellative. The monoids we shall investigate are as follows.  

\begin{defi}
A monoid~$\MM$ is said to be \emph{of right-$O$-type} (\resp \emph{left-$O$-type}) if $\MM$ is left-cancellative (\resp right-cancellative), $1$~is the only invertible element in~$\MM$, and, for all $\gg, \hh$ in~$\MM$, at least one of $\gg \dive \hh$, $\hh \dive \gg$ (\resp $\gg \multeR \hh$, $\hh \multeR \gg$)  holds. A monoid is \emph{of $O$-type} if it is both of right- and left-$O$-type.
\end{defi}

In other words, a monoid~$\MM$ is of right-$O$-type if it is left-cancellative and left-divisibility is a linear ordering on~$\MM$, and it is of $O$-type if it is cancellative and left- and right-divisibility are linear orderings on~$\MM$. The letter $O$ stands for ``order''; it may seem strange that the notion connected with left-divisibility is called ``right-$O$-type'', but this option is natural when one thinks in terms of multiples and it is more coherent with the forthcoming terminology. The connection with ordered groups is easy.

\begin{lemm}
\label{L:Cone}
For $\GG$ a group and $\MM$ a submonoid of~$\GG$, the following are equivalent:

\ITEM1 The group~$\GG$ admits a left-invariant ordering whose positive cone is $\MM {\setminus} \{1\}$;

\ITEM2 The monoid~$\MM$ is of $O$-type.
\end{lemm}

\begin{proof}
Assume \ITEM1. Put $\PP = \MM {\setminus} \{1\}$. First, by assumption, $\MM$ is included in a group, hence it must be cancellative. Next, assume that $\gg$ is an invertible element of~$\MM$, \ie, there exists~$\hh$ in~$\MM$ satisfying $\gg \hh = 1$. If $\gg$ belongs to~$\PP$, then so does~$\hh$ and, therefore, $\gg$ belongs to~$\PP \cap \PP\inv$, contradicting the assumption that $\PP$ is a positive cone. So $1$ must be the only invertible element of~$\MM$. Now, let $\gg, \hh$ be distinct elements of~$\MM$. Then one of $\gg\inv \hh$, $\hh\inv \gg$ belongs to~$\PP$, hence to~$\MM$: in the first case, $\gg \dive \hh$ holds, in the second, $\hh \dive \gg$. Symmetrically, one of $\gg\hh\inv$, $\hh\gg\inv$ belongs to~$\PP$, hence to~$\MM$, now implying $\gg \multeR \hh$ or $\hh \multeR \gg$. So any two elements of~$\MM$ are comparable with respect to~$\dive$ and~$\multeR$. Hence $\MM$ is of $O$-type, and \ITEM1 implies~\ITEM2.

Conversely, assume that $\MM$ is of $O$-type. Put $\PP = \MM {\setminus} \{1\}$ again. Then $\PP$ is a subsemigroup of~$\GG$. The assumption that $1$ is the only invertible element in~$\MM$ implies $\PP \cap \PP\inv = \emptyset$. Next, the assumption that any two elements of~$\MM$ are comparable with respect to~$\dive$ implies {\it a fortiori} that any two of its elements admit a common right-multiple. By Ore's theorem~\cite{ClP}, this implies that $\GG$ is a group of right-fractions for~$\MM$, \ie, every element of~$\GG$ admits an expression of the form~$\gg \hh\inv$ with~$\gg, \hh$ in~$\MM$. Now, let $\ff$ be an element of~$\GG$. As said above, there exist~$\gg, \hh$ in~$\MM$ satisfying $\ff = \gg \hh\inv$. By assumption, at least one of $\gg \multeR \hh$, $\hh \multeR \gg$ holds in~$\MM$. This means that at least one of $\ff \in \MM$, $\ff \in \MM\inv$ holds. Therefore, we have $\GG = \MM \cup \MM\inv$, which is also $\GG = \PP \cup \PP\inv \cup \{1\}$. So $\PP$ is a positive cone on~$\GG$, and \ITEM2 implies~\ITEM1.
\end{proof}

It will be convenient to restate the orderability criterion of Lemma~\ref{L:Cone} in terms of presentations. A group presentation~$\Pres\SS\RR$ is called \emph{positive} \cite{Dgp} if $\RR$ a family of relations of the form $\uu = \vv$, where $\uu, \vv$ are nonempty words in the alphabet~$\SS$ (no empty word, and no letter~$\ss\inv$). Every positive presentation~$\Pres\SS\RR$ gives rise to two structures, namely a monoid, here denoted~$\Mon\SS\RR$, and a group, denoted~$\Gr\SS\RR$. Note that a monoid admits a positive presentation if and only if $1$ is the only invertible element. Also remember that, in general, the monoid~$\Mon\SS\RR$ need not embed in the group~$\Gr\SS\RR$.  

\begin{prop}
\label{P:Recipe}
A necessary and sufficient condition for a group~$\GG$ to be orderable is that 
\begin{equation*}
\mbox{$\GG$ admits a positive presentation~$\Pres\SS\RR$ such that the monoid~$\Mon\SS\RR$ is of $O$-type.} 
\end{equation*}
In this case, the subsemigroup of~$\GG$ generated by~$\SS$ is the positive cone of a left-invariant ordering on~$\GG$. If $\SS$ is finite, this ordering is an isolated point in the space~$\LO\GG$.  
\end{prop}

\begin{proof}
Assume that $\GG$ is an orderable group. Let $\PP$ be the positive cone of a left-invariant ordering on~$\GG$, and let $\MM = \PP \cup \{1\}$. By the implication \ITEM1\,$\Rightarrow$\,\ITEM2 of Lemma~\ref{L:Cone}, the monoid~$\MM$ is of $O$-type. As $1$ is the only invertible element in~$\MM$, the latter admits a positive presentation~$\Pres\SS\RR$. As $\GG = \MM \cup \MM\inv$ holds, $\GG$ is a group of right-fractions for~$\MM$. By standard arguments, this implies that $\Pres\SS\RR$ is also a presentation of~$\GG$. 

Conversely, assume that $\GG$ admits a positive presentation~$\Pres\SS\RR$ such that the monoid $\Mon\SS\RR$ is of $O$-type. Let $\MM$ be the submonoid of~$\GG$ generated by~$\SS$, and let $\PP = \MM{\setminus}\{1\}$. As observed in the proof of Lemma~\ref{L:Cone}, Ore's theorem implies that $\Mon\SS\RR$ embeds in a group of fractions, and the latter admits the presentation~$\Pres\SS\RR$, hence is isomorphic to~$\GG$. Hence, the identity mapping on~$\SS$ induces an embedding~$\iota$ of~$\Mon\SS\RR$ into~$\GG$. Therefore, the image of~$\iota$, which is the submonoid of~$\GG$ generated by~$\SS$, hence is~$\MM$, admits the presentation~$\Pres\SS\RR$. So the assumption implies that $\MM$ is of $O$-type. Then, by the implication \ITEM2\,$\Rightarrow$\,\ITEM1 of Lemma~\ref{L:Cone}, $\PP$ is the positive cone of a left-invariant ordering on~$\GG$. 

As for the last point, the definition of the topology on the space~$\LO\GG$~\cite{Sik} implies that, if the positive cone of a left-ordering on the group~$\GG$ is generated, as a semigroup, by a finite set~$\SS$, then the ordering is an isolated point in the space~$\LO\GG$ because this ordering is the only one in which $\SS$ is positive and the set of all such orderings is open.
\end{proof}


\section{Triangular presentations}
\label{S:Triangular}

We are thus led to looking for monoids of $O$-type and, more specifically, for recognizing which presentations define monoids of $O$-type. Owing to the symmetry of the definition, we shall mainly focus on recognizing monoids of right-$O$-type and then use the criteria for the opposite presentation. Now, if a monoid~$\MM$ is of right-$O$-type and it is generated by some subset~$\SS$, then, for all $\ss, \ss'$ in~$\SS$, the elements~$\ss$ and~$\ss'$ are comparable with respect to~$\dive$, \ie, $\ss' = \ss \gg$ holds for some~$\gg$, or {\it vice versa}. In other words, some relation of the particular form $\ss' = \ss \ww$ must be satisfied in~$\MM$. We shall consider presentations in which all relations have this form (see Section~\ref{S:Limits} for a discussion about the relevance of this approach).

\begin{defi}
A semigroup relation $\uu = \vv$ is called \emph{triangular} if either $\uu$ or $\vv$ consists of a single letter.
\end{defi}

So, a triangular relation has the generic form $\ss' = \ss \ww$, where $\ss, \ss'$ belong to the reference alphabet. For instance, $\tta = \ttb\tta\ttb$ and $\ttb = \ttc^2\ttb\tta$ are typical triangular relations in the alphabet $\{\tta, \ttb, \ttc\}$. The problem we shall address now is

\begin{ques}
\label{Q:Connection}
Assume that $\Pres\SS\RR$ is a presentation consisting of triangular relations. Is the associated monoid necessarily of right-$O$-type?
\end{ques}

The following counter-example shows that a uniform positive answer is impossible.

\begin{exam}
\label{X:Abelian}
The presentation $\Pres{\tta, \ttb, \ttc}{\ttc = \tta \ttb, \ttc = \ttb\tta}$ consists of two triangular relations. The associated monoid~$\MM$ is a rank~$2$ free Abelian monoid based on~$\tta$ and~$\ttb$, and neither of~$\tta, \ttb$ is a right-multiple of the other. So $\MM$ is not of right-$O$-type.
\end{exam}

Clearly, the problem in Example~\ref{X:Abelian} is the existence of several relations $\ttc = ...$ simultaneously. We are thus led to restricting to particular families of triangular relations. If $\SS$ is a nonempty set, we denote by~$\SS^*$ the free monoid of all words in the alphabet~$\SS$. We use $\ew$ for the empty word.

\begin{defi}
\label{D:Triang}
A positive presentation $\Pres\SS\RR$ is called \emph{right-triangular} if there exist $\SS' \subseteq \SS$ and maps~$\NN$ (``next'') $: \SS' \to \SS$ and~$\CC$ (``complement'') $: \SS' \to \SS^* {\setminus}\{\ew\}$  such that $\NN$ is injective with no fixpoint and $\RR$ consists of the  relations $\NN(\ss) = \ss \CC(\ss)$ for~$\ss$ in~$\SS'$. We write $\CC^\ii(\ss)$ for $\CC(\ss) \CC(\NN(\ss)) \pdots \CC(\NN^{\ii-1}(\ss))$ when $\NN^\ii(\ss)$ is defined, and~$\RRh$ for $\RR \cup \{\NN^\ii(\ss) =\nobreak \ss \CC^\ii(\ss) \mid \ii \ge 2\}$. A \emph{left-triangular} presentation is defined symmetrically by relations $\NNt(\ss) = \widetilde\CC(\ss) \ss$. A presentation is \emph{triangular} if it is both right- and left-triangular.
\end{defi}

\begin{exam}
\label{X:Triang}
Assume $\SS = \{\tta, \ttb, \ttc\}$ and $\RR = \{\tta = \ttb\tta\ttc, \ttb = \ttc\ttb\tta\}$. Then $\Pres\SS\RR$ is a right-triangular presentation. The associated maps~$\NN$ and~$\CC$ are given by $\NN(\ttb) = \tta$, $\CC(\ttb) = \tta\ttc$, $\NN(\ttc) = \ttb$, $\CC(\ttc) = \ttb\tta$, and we have $\RRh = \RR \cup \{\tta = \ttc\ttb\tta^2\ttc\}$. The presentation~$\Pres\SS\RR$ is also left-triangular, with~$\NNt$ and~$\CCt$ defined by $\NNt(\tta) = \ttb$, $\CCt(\tta) = \ttc\ttb$, $\NNt(\ttc) = \tta$, $\CCt(\ttc) = \ttb\tta$.
\end{exam}

If $\Pres\SS\RR$ is a right-triangular presentation, the family~$\RRh$ is a sort of transitive closure of~$\RR$, and the presentations~$\Pres\SS\RR$ and $\Pres\SS\RRh$ define the same monoid and the same group. Triangular presentations can be described in terms of the left- and right-graphs~\cite{Adj, Rem}. The \emph{left-graph} (\resp \emph{right-graph}) of~$\Pres\SS\RR$ is the unoriented graph with vertex set~$\SS$ such that $\{\ss, \ss'\}$ is an edge if and only if there exists a relation $\ss... = \ss'\!...$ (\resp $...\ss = ... \ss'$) in~$\RR$. Then a presentation~$\Pres\SS\RR$ is right-triangular if it consists of triangular relations and, in addition, the left-graph of~$\Pres\SS\RR$ is a union of discrete chains. In practice, we shall be mostly interested in the case when there is only one (countable) chain, in which case there exists a (finite or infinite) subset~$\II$ of~$\ZZZZ$ such that $\SS$ is $\{\tta_\ii \mid \ii\in \II\}$ and $\RR$ consists of one relation $\tta_{\ii-1} = \tta_\ii \ww_\ii$ for each~$\ii$ in~$\II$ that is not minimal. 

Our main technical result will be a criterion for recognizing which right-triangular presentations give rise to a monoid of right-$O$-type.

\begin{prop}[Main Lemma]
\label{P:MainLemma}
Assume that $\Pres\SS\RR$ is a right-triangular presentation. Then the following are equivalent:

\ITEM1 The monoid~$\Mon\SS\RR$ is of right-$O$-type;

\ITEM2 Any two elements of~$\Mon\SS\RR$ admit a common right-multiple.
\end{prop}

The proof of the Main Lemma will be completed at the end of Section~\ref{S:Complete} below. Merging the result with Proposition~\ref{P:Recipe} immediately provides the following sufficient condition of orderability.

\begin{coro}
\label{C:Criterion}
A sufficient condition for a group~$\GG$ to be orderable is that 
\begin{quote}
$\GG$ admits a triangular presentation~$\Pres\SS\RR$ such that any two elements of the monoid~$\Mon\SS\RR$ have a common right-multiple and a common left-multiple. 
\end{quote}
In this case, the subsemigroup of~$\GG$ generated by~$\SS$ is the positive cone of a left-invariant ordering on~$\GG$. If $\SS$ is finite, this ordering is an isolated point in the space~$\LO\GG$.  
\end{coro}

\begin{proof}
By the Main Lemma, the monoid~$\Mon\SS\RR$, which admits a right-triangular presentation is of right-$O$-type, and so is the opposite monoid. Hence $\Mon\SS\RR$ is also of left-$O$-type, and therefore it is of $O$-type. Then Proposition~\ref{P:Recipe} implies that $\GG$ is orderable, with the expected explicit ordering.
\end{proof}


\section{Subword reversing}
\label{S:Rev}

We shall prove the Main Lemma by using subword reversing. In essence, subword reversing is a strategy for constructing van Kampen diagrams in a context of monoids, \ie, equivalently, for finding derivations between words, and we shall see that it is especially relevant for investigating triangular presentations (due to the special form of triangular presentations, it might well be that alternative arguments using rewrite systems or other approaches also exist, but this is not clear).

The description given below is sketchy, as we only mention the definition and the needed technical results. We refer to~\cite{Dgp, Dia} for additional motivation and explanation.

As is usual with presented groups, if $\SS$ is an alphabet, we introduce a formal copy~$\SS\inv$ of~$\SS$ consisting of one letter~$\ss\inv$ for each letter of~$\SS$. The letters of~$\SS$ are then called \emph{positive}, whereas those of~$\SS\inv$ are called \emph{negative}. Accordingly, a word in the alphabet~$\SS \cup \SS\inv$ will be called a \emph{signed $\SS$-word}, whereas a word in the alphabet~$\SS$ is called an {$\SS$-word}, or a \emph{positive} $\SS$-word if we wish to insist that there is no negative letter. If $\ww$ is a signed $\SS$-word, $\ww\inv$ denotes the word obtained from~$\ww$ by exchanging~$\ss$ and~$\ss\inv$ everywhere and reversing the order of the letters. A word of the form~$\uu\inv \vv$ with $\uu, \vv$ positive is called \emph{negative--positive}. 

\begin{defi}
\label{D:Reversing}
Assume that $\Pres\SS\RR$ a positive presentation and $\ww, \ww'$ are signed $\SS$-words. We say that $\ww$ is \emph{right-$\RR$-reversible to~$\ww'$ in one step}, denoted $\ww \RevR1 \ww'$, if either there exist $\ss, \ss'$ in~$\SS$, a relation $\ss \vv' = \ss' \vv$ of~$\RR$, and signed words~$\ww_1, \ww_2$ satisfying
\begin{equation}
\label{E:Rev1}
\ww = \ww_1 \, \ss\inv \ss' \, \ww_2
\qquad\mbox{and}\qquad
\ww'= \ww_1 \, \vv' \vv\inv \, \ww_2,
\end{equation}
or there exist $\ss$~in~$\SS$ and signed $\SS$-words~$\ww_1, \ww_2$ satisfying
\begin{equation}
\label{E:Rev2}
\ww = \ww_1 \, \ss\inv \ss \, \ww_2
\qquad\mbox{and}\qquad
\ww' = \ww_1 \, \ww_2.
\end{equation}
We say that $\ww$ is \emph{right-$\RR$-reversible to~$\ww'$ in $\nn$~steps}, denoted $\ww \RevR\nn \ww'$, if there exist $\ww_0 \wdots \ww_\nn$ satisfying $\ww_0 = \ww$, $\ww_\nn = \ww'$ and $\ww_\ii \RevR1 \ww_{\ii+1}$ for each~$\ii$. 
We write $\ww \revR \ww'$ if $\ww \RevR\nn \ww'$ holds for some~$\nn$.
\end{defi}

Note that \eqref{E:Rev2} becomes an instance of~\eqref{E:Rev1} if, for every~$\ss$ in~$\SS$, the trivial relation $\ss = \ss$ is considered to belong to~$\RR$. Right-reversing consists in replacing a negative--positive length two subword with a posit\-ive--negative word, hence somehow reversing the signs, whence the terminology. We shall often write ``reversing'' for ``right-reversing'' (except at the end of Section~\ref{S:Mult} where left-reversing, the symmetric counterpart of right-reversing, occurs). 

\begin{exam}
\label{X:Rev}
Assume $\SS = \{\tta, \ttb, \ttc\}$ and $\RR = \{\tta = \ttb\tta\ttb, \ttb = \ttc\ttb\ttc, \tta = \ttc\ttb\ttc\tta\ttb\}$. Starting for instance with $\ww =\nobreak \tta\inv \ttc\inv \tta$, we find
$$\ww = \tta\inv \underline{\ttc\inv \tta} 
\ \RevR1\  \underline{\tta\inv \ttb} \ttc \tta \ttb
\ \RevR1\  \ttb\inv \tta\inv \ttc \tta \ttb,$$
where, at each step, the reversed subword is underlined. Observe that the word obtained after two reversing steps is $\ttb\inv \ww\inv \ttb$, so that $\ww \RevR{4\nn} \ttb^{-2\nn} \ww \ttb^{2\nn}$ holds for every~$\nn$.
\end{exam}

It is useful to associate with every sequence of reversing steps a rectangular grid diagram that illustrates it (see~\cite{Dia} for full details). Assume that $\ww_0, \ww_1, ...$ is an $\RR$-reversing sequence, \ie, $\ww_\ii \RevR1 \ww_{\ii+1}$ holds for every~$\ii$. The diagram is analogous to a van Kampen diagram, and it is constructed inductively. First we associate with~$\ww_0$ a path shaped like an ascending staircase by reading~$\ww_0$ from left to right and iteratively appending a horizontal right-oriented edge labeled~$\ss$ for each letter~$\ss$, and a vertical down-oriented edge labeled~$\ss$ for each letter~$\ss\inv$. Then, assume that the diagram for~$\ww_0 \wdots \ww_\ii$ has been constructed, and $\ww_{\ii+1}$ is obtained from~$\ww_\ii$ by reversing some subword~$\ss\inv \ss'$ into~$\vv' \vv\inv$. Inductively, the subword~$\ss\inv \ss'$ corresponds to an open pattern 
\VR(8,6)\begin{picture}(15,0)(-3,5)
\pcline{->}(0.5,10)(9.5,10)
\taput{$\ss'$}
\pcline{->}(0,9.5)(0,0.5)
\tlput{$\ss$}
\end{picture}
in the diagram, and we complete it by appending new arrows forming the closed pattern
\VR(4,8)\begin{picture}(17,0)(-3,4)
\pcline{->}(0.5,10)(9.5,10)
\taput{$\ss'$}
\pcline{->}(0,9.5)(0,0.5)
\tlput{$\ss$}
\pcline{->}(0.5,0)(9.5,0)
\tbput{$\vv'$}
\pcline{->}(10,9.5)(10,0.5)
\trput{$\vv$}
\put(3,4){$\revR$}
\end{picture}. If the length~$\ell$ of~$\vv$ is more than one, the arrow labeled~$\vv$ consists of $\ell$~concatenated arrows. If $\vv$ is empty, we append a equality sign, as in 
\VR(9,8)\begin{picture}(15,0)(-3,4)
\pcline{->}(0.5,10)(9.5,10)
\taput{$\ss'$}
\pcline{->}(0,9.5)(0,0.5)
\tlput{$\ss$}
\pcline{->}(0.5,0)(9.5,0)
\tbput{$\vv'$}
\pcline[style=double](10,9)(10,1)
\put(3,4){$\revR$}
\end{picture}. 
It then follows from the inductive definition that all words $\ww_\ii$ can be read in the diagram by following the paths that connect the bottom-left corner to the top-right corner, see Figure~\ref{F:Rev}.

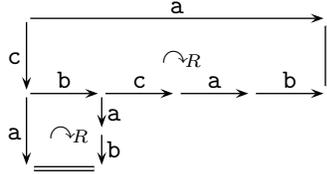
\begin{figure} [htb]
\begin{picture}(40,22)(0,2)
\pcline{->}(0.5,20)(39.5,20)
\taput{$\tta$}
\pcline{->}(0.5,10)(9.5,10)
\taput{$\ttb$}
\pcline{->}(10.5,10)(19.5,10)
\taput{$\ttc$}
\pcline{->}(20.5,10)(29.5,10)
\taput{$\tta$}
\pcline{->}(30.5,10)(39.5,10)
\taput{$\ttb$}
\pcline[style=double](1,0)(9,0)
\put(18,14){$\revR$}

\pcline{->}(0,19.5)(0,10.5)
\tlput{$\ttc$}
\pcline{->}(0,9.5)(0,0.5)
\tlput{$\tta$}
\pcline{->}(10,9.5)(10,5.5)
\trput{$\tta$}
\pcline{->}(10,4.5)(10,0.5)
\trput{$\ttb$}
\pcline[style=double](40,19)(40,11)
\put(3,4){$\revR$}
\end{picture}
\caption{\sf\smaller Reversing diagram associated with the reversing sequence of Example~\ref{X:Rev}: starting from the signed word~$\tta\inv \ttc\inv \tta$, which corresponds to the left and top arrows, we successively reverse $\ttc\inv \tta$ into $\ttb \ttc \tta \ttb$, and $\tta\inv \ttb$ into $\ttb\inv \tta\inv$, thus obtaining the final word $\ttb\inv \tta\inv \ttc \tta \ttb$.}
\label{F:Rev}
\end{figure}

If $\Pres\SS\RR$ is a positive presentation, and $\uu, \vv$ are $\SS$-words, applying iterated subword reversing to the signed word~$\uu\inv \vv$ may lead to three different behaviours:

- either the process continues for ever (as in the case of Example~\ref{X:Rev}),

- or one gets stuck with a factor~$\ss\inv \ss'$ such that $\RR$ contains no relation $\ss... = \ss'\!...$,

- or the process leads in finitely many steps to a positive--negative word~$\vv' \uu'{}\inv$ where $\uu'$ and $\vv'$ are $\SS$-words (no letter~$\ss\inv$): then the sequence cannot be extended since the last word contains no subword of the form~$\ss\inv \ss'$; this case corresponds to a reversing diagram of the form
\VR(9,8)\begin{picture}(18,0)(-3,4)
\pcline{->}(0.5,10)(11.5,10)
\taput{$\vv$}
\pcline{->}(0,9.5)(0,0.5)
\tlput{$\uu$}
\pcline{->}(0.5,0)(11.5,0)
\tbput{$\vv'$}
\pcline{->}(12,9.5)(12,0.5)
\trput{$\uu'$}
\put(4,4){$\revR$}
\end{picture}, and we shall then say that the reversing of~$\uu\inv \vv$ is \emph{terminating}.

We shall use without proof two (elementary) results about reversing. The first one connects $\RR$-reversing with $\RR$-equivalence and it expresses that a reversing diagram projects to a van Kampen diagram when the vertices connected with equality signs are identified. 

\begin{nota}
For $\Pres\SS\RR$ a positive presentation, we denote by~$\eqpR$ the smallest congruence on~$\SS^*$ that includes~$\RR$, so that $\Mon\SS\RR$ is $\SS^*\!/{\eqpR}$. For~$\ww$ an $\SS$-word, we denote by~$\clp\ww$ the $\eqpR$-class of~$\ww$, \ie, the element of the monoid~$\Mon\SS\RR$ represented by~$\ww$.
\end{nota}
 
\begin{lemm}\cite[Proposition 1.9]{Dgp}
\label{L:Equiv}
Assume that $\Pres\SS\RR$ is a positive presentation, and $\uu, \vv, \uu', \vv'$ are $\SS$-words satisfying $\uu\inv \vv \revR \vv' \uu'{}\inv$. Then $\uu \vv' \eqpR \vv \uu'$ holds. In particular, $\uu\inv \vv \revR \ew$ implies $\uu \eqpR \vv$.
\end{lemm}

In other words, the existence of a reversing diagram
\VR(8,8)\begin{picture}(19,0)(-3,4)
\pcline{->}(0.5,10)(11.5,10)
\taput{$\vv$}
\pcline{->}(0,9.5)(0,0.5)
\tlput{$\uu$}
\pcline{->}(0.5,0)(11.5,0)
\tbput{$\vv'$}
\pcline{->}(12,9.5)(12,0.5)
\trput{$\uu'$}
\put(4,4){$\revR$}
\end{picture}
implies $\uu \vv'\eqpR\nobreak \vv \uu'$, which also reads $\clp\uu \clp{\vv'} = \clp\vv \clp{\uu'}$: the element represented by~$\uu\vv'$ and~$\vv\uu'$ is a common right-multiple of~$\clp\uu$ and~$\clp\vv$ in the associated monoid. Thus subword reversing can be seen as a tool for constructing common right-multiples in presented monoids. 

The second basic result says that, when the reversing of a compound word~$\uu\inv \vv_1 \vv_2$ terminates, the reversing steps involving~$\vv_1$ and~$\vv_2$ can be separated.

\begin{lemm}\cite[Lemma~1.8]{Dgp}
\label{L:Split}
Assume that $\Pres\SS\RR$ is a positive presentation and $\uu, \vv, \uu', \vv'$ are $\SS$-words satisfying $\uu\inv \vv \RevR\nn \vv' \uu'{}\inv$. Then, for every decomposition $\vv = \vv_1 \vv_2$, there exist an $\SS$-word~$\uu_0$ and decompositions $\vv' = \vv'_1 \vv'_2$ and $\nn = \nn_1 + \nn_2$ such that  $\uu\inv \vv_1 \RevR{\nn_1} \vv'_1 \uu_0\inv$ and $\uu_0\inv \vv_2 \RevR{\nn_2} \vv'_2 \uu'{}\inv$ hold.
\end{lemm}

In other words, every diagram 
\VR(9,8)\begin{picture}(31,0)(-3,4)
\pcline{->}(0.5,10)(11.5,10)
\taput{$\vv_1$}
\pcline{->}(12.5,10)(23.5,10)
\taput{$\vv_2$}
\pcline{->}(0,9.5)(0,0.5)
\tlput{$\uu$}
\pcline{->}(0.5,0)(23.5,0)
\tbput{$\vv'$}
\pcline{->}(24,9.5)(24,0.5)
\trput{$\uu'$}
\put(10,4){$\RevR\nn$}
\end{picture}
splits into 
\VR(2,9)\begin{picture}(38,0)(-3,4)
\pcline{->}(0.5,10)(14.5,10)
\taput{$\vv_1$}
\pcline{->}(15.5,10)(29.5,10)
\taput{$\vv_2$}
\pcline{->}(0,9.5)(0,0.5)
\tlput{$\uu$}
\pcline{->}(0.5,0)(14.5,0)
\tbput{$\vv'_1$}
\pcline{->}(15.5,0)(29.5,0)
\tbput{$\vv'_2$}
\pcline{->}(15,9.5)(15,0.5)
\trput{$\uu_0$}
\pcline{->}(30,9.5)(30,0.5)
\trput{$\uu'$}
\put(5,4){$\RevR{\nn_1}$}
\put(21,4){$\RevR{\nn_2}$}
\end{picture}.

Here comes the first specific observation about reversing with triangular relations.

\begin{lemm}
\label{L:Empty}
If $\Pres\SS\RR$ is a positive presentation consisting of triangular relations, and $\uu, \vv, \uu', \vv'$ are $\SS$-words satisfying $\uu\inv \vv \revR \vv' \uu'{}\inv$, then at least one of~$\uu', \vv'$ is empty.
\end{lemm}

\begin{proof}
We use induction on the number of reversing steps, \ie, the number~$\nn$ such that $\uu\inv \vv \RevR\nn \vv' \uu'{}\inv$ holds. For $\nn = 0$, the only possibility is that $\uu, \vv, \uu', \vv'$ all are empty and the result is trivial. For $\nn = 1$, the only possibility is that $\uu$ or~$\vv$ consists of one letter, say for instance $\uu = \ss \in \SS$. Write $\vv = \ss' \ww$ with $\ss'$ in~$\SS$. If the (unique) reversing step is of the type $\ss\inv \ss \rev \ew$, we obtain $\uu' = \ew$ (and $\vv' = \ww$). Otherwise, the reversing step is either of the type $\ss\inv \ss' \rev \ww'$ with $\ss \ww' = \ss'$ a relation of~$\RR$, or of the type $\ss\inv \ss' \rev \ww'{}\inv$ with $\ss' \ww' = \ss$ a relation of~$\RR$. In the first case, we obtain $\uu' = \ew$ (and $\vv' = \ww' \ww$); in the second case, the final word~$\ww'{}\inv \ww$ is positive--negative only if $\ww$ is empty, and, in this case, we have $\vv' = \ew$ (and $\uu' = \ww'$). The argument is similar if $\vv$, instead of~$\uu$, has length one.

Assume now $\nn \ge 2$. Then at least one of the words~$\uu, \vv$ has length two or more. Assume that $\vv$ does, and write it as $\vv_1 \vv_2$ with $\vv_1, \vv_2$ nonempty. By Lemma~\ref{L:Split}, the assumption that $\uu\inv \vv_1 \vv_2$ reverses to~$\vv' \uu'{}\inv$ in $\nn$~steps implies the existence of $\SS$-words~$\vv'_1, \vv'_2, \uu_0$ and of numbers~$\nn_1, \nn_2$ satisfying
$$\vv' = \vv'_1 \vv'_2, \quad \nn = \nn_1 + \nn_2, \quad \uu\inv \vv_1 \ \RevR{\nn_1}\ \vv'_1 \uu_0\inv, \quad \mbox{and} \quad \uu_0\inv \vv_2 \ \RevR{\nn_2}\ \vv'_1 \uu'{}\inv.$$
Two cases are possible. Assume first $\nn_1 = \nn$, whence $\nn_2 = 0$. As, by assumption, $\vv_2$ is nonempty, the hypothesis that $\uu_0\inv \vv_2$ is a positive--negative word implies that $\uu_0$ is empty, and so is~$\uu'$. Assume now $\nn_1 < \nn$. The value $\nn_1 = 0$ is impossible as it would imply that $\uu$ or~$\vv_1$ is empty, contrary to the assumption. Hence we also have $\nn_2 < \nn$. Now assume that $\uu'$ is nonempty. Then, as we have $\nn_2 < \nn$, the induction hypothesis implies that $\vv'_2$ is empty. Next, $\uu'$ can be nonempty only if $\uu_0$ is nonempty. Then, as we have $\nn_1 < \nn$, the induction hypothesis implies that $\vv'_1$ is empty as well, and we conclude that $\vv'$, which is~$\vv'_1 \vv'_2$, is empty. See Figure~\ref{F:Empty}.
\end{proof}

\begin{figure}[htb]
\begin{picture}(100,10)
\psline{->}(0.5,10)(9.5,10)
\psline{->}(10.5,10)(19.5,10)
\psline{->}(0.5,0)(9.5, 0)
\psline{->}(10.5,0)(19.5,0)
\psline{->}(0,9.5)(0,0.5)
\psline[style=double](10,9)(10,1)
\psline[style=double](20,9)(20,1)

\psline{->}(40.5,10)(49.5,10)
\psline{->}(50.5,10)(59.5,10)
\psline[style=double](41,0)(49, 0)
\psline{->}(50.5,0)(59.5,0)
\psline{->}(40,9.5)(40,0.5)
\psline{->}(50,9.5)(50,0.5)
\psline[style=double](60,9)(60,1)

\psline{->}(80.5,10)(89.5,10)
\psline{->}(90.5,10)(99.5,10)
\psline[style=double](81,0)(89, 0)
\psline[style=double](91,0)(99,0)
\psline{->}(80,9.5)(80,0.5)
\psline{->}(90,9.5)(90,0.5)
\psline{->}(100,9.5)(100,0.5)
\end{picture}
\caption{\sf\smaller The three possible ways of concatenating two reversing diagrams in which one of the output words is empty: in each case, one of the final output words has to be empty.}
\label{F:Empty}
\end{figure}
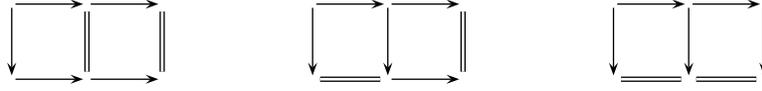

So, in the context of triangular relations, when reversing is terminating, it shows not only that the elements of the monoid represented by the initial words admit a common right-multiple, but also that these elements are comparable with respect to left-divisibility. Indeed, by Lemma~\ref{L:Equiv}, if we have $\uu\inv \vv \revR \vv'$ with~$\vv'$ a (positive) $\SS$-word, we deduce $\clp\uu \clp{\vv'} = \clp\vv$, whence $\clp\uu \dive \clp\vv$ in the monoid~$\Mon\SS\RR$ and, symmetrically, if we have $\uu\inv \vv \revR \uu'{}\inv$ with~$\uu'$ an $\SS$-word, we deduce $\clp\vv \clp{\uu'} = \clp\uu$, whence $\clp\vv \dive \clp\uu$.


\section{Completeness of subword reversing}
\label{S:Complete}

Owing to Lemma~\ref{L:Empty}, if $\Pres\SS\RR$ is a positive presentation that consists of triangular relations, in order to prove that any two elements of the monoid~$\Mon\SS\RR$ are comparable with respect to left-divisibility, it is enough to show that, for all $\SS$-words~$\uu, \vv$, there exists at least one reversing sequence from~$\uu\inv \vv$ that is terminating, \ie, is finite and finishes with a positive--negative word.

A natural situation in which reversing is guaranteed to be terminating is the case when, for every pair of letters~$\ss, \ss'$, there exists at least one relation $\ss... = \ss'\!...$ in~$\RR$ (so that one never gets stuck) and all relations~$\uu = \vv$ of~$\RR$ involve words~$\uu, \vv$ of length at most two, so that reversing does not increase the length of words. More generally, termination is guaranteed when one can identify a set of $\SS$-words~$\SSh$ including~$\SS$ so that, for all~$\uu, \vv$ in~$\SSh$, there exist~$\uu', \vv'$ in~$\SSh$ satisfying $\uu\inv \vv \revR \vv' \uu'{}\inv$ (which amounts to meet the above conditions with respect to the extended alphabet~$\SSh$). However, except in a few trivial examples, this approach fails when applied to presentations with triangular relations: usually, the closure~$\SSh$ of~$\SS$ under reversing is infinite and difficult to work with. Therefore, we must use a more subtle approach in two steps, namely  showing that 

\ITEM1 if two elements represented by words~$\uu, \vv$ admit a common right-multiple, then the reversing of~$\uu\inv \vv$ terminates, and 

\ITEM2 any two elements of the considered monoids admit a common right-multiple.\\
When this is done, Lemma~\ref{L:Empty} can be applied and one is close to concluding that the considered monoid is of right-$O$-type. In this section, we address point~\ITEM1. Here comes the second, more important observation of the paper, namely that \ITEM1 is \emph{always} true for a right-triangular presentation. Technically, the proof relies on what is known as the completeness condition.

If $\Pres\SS\RR$ is any positive presentation, then, by Lemma~\ref{L:Equiv}, $\uu\inv \vv \revR \ew$, \ie, the existence of a diagram
\VR(9,6)\begin{picture}(17,0)(-3,4)
\pcline{->}(0.5,10)(11.5,10)
\taput{$\vv$}
\pcline{->}(0,9.5)(0,0.5)
\tlput{$\uu$}
\pcline[style=double](1,0)(11,0)
\pcline[style=double](12,9)(12,1)
\put(4,4){$\revR$}
\end{picture},
implies $\uu \eqpR \vv$. We consider now the converse implication.

\begin{defi}[\bf complete]
\label{D:Complete}
 A positive presentation $\Pres\SS\RR$ is called \emph{complete for right-reversing} if, for all $\SS$-words~$\uu, \vv$,
\begin{equation}
\label{E:Complete}
\uu \eqpR \vv
\mbox{\qquad implies \qquad}
\uu\inv \vv \revR \ew. 
\end{equation}
\end{defi}

As the converse of~\eqref{E:Complete} is always true, if $\Pres\SS\RR$ is complete, \eqref{E:Complete} is an equivalence.

\begin{rema}
If there exist two letters~$\ss, \ss'$ of~$\SS$ such that, in the presentation~$\RR$, there is more than one relation of the type~$\ss... = \ss'\!...$ (including the case when there exists a relation~$\ss... = \ss...$ different from the implicit trivial relation $\ss = \ss$), then $\RR$-reversing need not be a deterministic process and, starting from some words~$\uu, \vv$, there may exist several pairs $\uu', \vv'$ satisfying $\uu\inv \vv \revR \vv' \uu'{}\inv$. According to our definitions, the condition $\uu\inv \vv \revR \ew$ involved in~\eqref{E:Complete} means that there exists \emph{at least one way} of obtaining the empty word starting from~$\uu\inv \vv$. However, this type of non-determinism never occurs with a right-triangular presentation~$\Pres\SS\RR$ or its completion $(\SS, \RRh)$: by definition, $\RRh$ contains at most one relation $\ss... = \ss'\!...$ for each pair of letters, so $\RR$- and $\RRh$-reversings are deterministic.
\end{rema} 

The intuition behind completeness is that, when a presentation is complete for right-reversing, the {\it a priori} complicated relation~$\eqpR$ can be replaced with the more simple relation~$\revR$. As explained in~\cite{Dia}, this makes recognizing some properties of the associated monoid and group easy. In our current context, in order to address point~\ITEM1 above, we are interested in connecting the existence of common multiples and termination of reversing. When the completeness condition is satisfied, this is easy.

\begin{lemm}
\label{L:Common}
Assume that $\Pres\SS\RR$ is a positive presentation that is complete for right-reversing. Then, for all~$\gg, \hh$ in~$\Mon\SS\RR$, the following are equivalent:

\ITEM1 The elements $\gg$ and $\hh$ admit a common right-multiple; 

\ITEM2 For some $\SS$-words~$\uu, \vv$ representing~$\gg$ and~$\hh$, the reversing of~$\uu\inv \vv$ is terminating; 

\ITEM3 For all $\SS$-words~$\uu, \vv$ representing~$\gg$ and~$\hh$, the reversing of~$\uu\inv \vv$ is terminating. 
\end{lemm}

\begin{proof}
Assume that $\gg$ and~$\hh$ admit a common right-multiple~$\ff$. By definition, there exist $\gg', \hh'$ satisfying $\ff = \gg \hh' = \hh \gg'$. Let $\uu, \vv, \uu', \vv'$ be arbitrary $\SS$-words representing~$\gg, \hh, \gg'$, and~$\hh'$. Then we have $\uu \vv' \eqpR \vv \uu'$, whence $(\uu \vv')\inv (\vv \uu') \revR\nobreak \ew$, \ie, $\vv'{}\inv \uu\inv \vv \uu' \revR \ew$ since $\Pres\SS\RR$ is complete for right-reversing. Applying Lemma~\ref{L:Split} twice, we split the reversing diagram of~$(\uu \vv')\inv (\vv \uu')$ into four diagrams:
$$\begin{picture}(30,23)
\pcline{->}(0.5,20)(14.5,20)
\taput{$\vv$}
\pcline{->}(15.5,20)(29.5,20)
\taput{$\uu'$}
\pcline{->}(0.5,10)(14.5,10)
\pcline{->}(15.5,10)(29.5,10)
\pcline[style=double](1,0)(14,0)
\pcline[style=double](16,0)(29,0)
\pcline{->}(0,19.5)(0,10.5)
\tlput{$\uu$}
\pcline{->}(0,9.5)(0,0.5)
\tlput{$\vv'$}
\pcline{->}(15,19.5)(15,10.5)
\pcline{->}(15,9.5)(15,0.5)
\pcline[style=double](30,19)(30,11)
\pcline[style=double](30,9)(30,1)
\put(5,4){$\revR$}
\put(21,4){$\revR$}
\put(5,14){$\revR$}
\put(21,14){$\revR$}
\put(32,0){.}
\end{picture}$$
Each of the four diagrams above necessarily corresponds to a terminating reversing and, in particular, the reversing of~$\uu\inv \vv$ must terminate. So \ITEM1 implies~\ITEM3. 

On the other hand, it is obvious that \ITEM3 implies~\ITEM2. Finally, by Lemma~\ref{L:Equiv}, any relation $\uu\inv \vv \revR \vv' \uu'{}\inv$ implies that the element of~$\Mon\SS\RR$ both represented by~$\uu\vv'$ and~$\vv\uu'$ is a common right-multiple of the elements represented by~$\uu$ and~$\vv$, so \ITEM2 implies~\ITEM1.
\end{proof}

Owing to Lemmas~\ref{L:Empty} and~\ref{L:Common}, we are led to wondering whether a presentation consisting of triangular relations is necessarily complete for right-reversing. A priori, the question may seem hopeless as the only method known so far for establishing that a presentation~$\Pres\SS\RR$ is complete for right-reversing \cite{Dff, Dgp} consists in establishing a certain combinatorial condition (the ``cube condition'') using an induction that is possible only when the associated monoid~$\MM$ satisfies some Noetherianity condition, namely that, for every~$\gg$ in~$\MM$, 
\begin{equation}
\label{E:Noeth}
\mbox{there is no infinite sequence $\gg_0, \gg_1, ...$ satisfying $\gg_0 \prec \gg_1 \prec \gg_2 \prec \pdots \dive \gg$,}
\end{equation} 
where $\gg \prec \hh$ means $\gg \dive \hh$ with $\gg \not= \hh$. Now, \eqref{E:Noeth} turns out to fail whenever $\RR$ contains a relation of the form $\ss = ... \ss ...$, hence in most of the cases we are interested in. However, right-triangular presentations turn out to be eligible for an alternative completeness argument. 

\begin{prop}
\label{P:Complete}
For every right-triangular presentation $\Pres\SS\RR$, the associated presentation $(\SS, \RRh)$ is complete for right-reversing.
\end{prop}

The proof will be split into several steps. Until the end of the proof, we assume that $\Pres\SS\RR$ is a fixed right-triangular presentation, with associated functions~$\NN$ and~$\CC$. We recall that this means that $\RR$ consists of the relations $\NN(\ss) = \ss \CC(\ss)$ with~$\ss$~in~$\SS$. We recall also that $\CC^\ii(\ss)$ stands for $\CC(\ss) \CC(\NN(\ss)) \pdots \CC(\NN^{\ii-1}(\ss))$ whenever $\NN^\ii(\ss)$ is defined. 

By definition, if $\uu, \vv$ are $\SS$-words, then $\uu \eqpR \vv$ holds, \ie, $\uu$ and~$\vv$ represent the same element in~$\Mon\SS\RR$, if and only if there exists an \emph{$\RR$-derivation} from~$\uu$ to~$\vv$, \ie, a sequence $\ww_0 = \uu, \ww_1 \wdots \ww_\nn = \vv$ such that each word~$\ww_\kk$ is obtained from~$\ww_{\kk-1}$ by applying exactly one relation of~$\RR$. We write $\uu \EqpR\nn \vv$ when there exists a length~$\nn$ derivation from~$\uu$ to~$\vv$.

\begin{defi}
Assume that $\ww$ is a nonempty $\SS$-word. We denote by~$\II(\ww)$ the initial letter of~$\ww$, and by~$\TT(\ww)$ (like ``tail'') the subword satisfying $\ww =\nobreak \II(\ww) \TT(\ww)$. We say that a letter~$\ss$ of~$\SS$ \emph{underlies}~$\ww$ if $\II(\ww) = \NN^\ii(\ss)$ holds for some $\ii \ge 0$. In this case, we put $\EE_\ss(\ww) = \ss \CC^\ii(\ss) \TT(\ww)$; otherwise, we put $\EE_\ss(\ww) = \ww$.
\end{defi}

A straightforward induction on~$\ii$ gives $\NN^\ii(\ss) \eqpR \ss \CC^\ii(\ss)$ whenever defined, and we deduce $\ww = \NN^\ii(\ss) \TT(\ww) \eqpR \ss \CC^\ii(\ss) \TT(\ww) = \EE_\ss(\ww)$ whenever $\II(\ww) = \NN^\ii(\ss)$ holds.

We begin with a direct consequence of the definition of a right-triangular presentation.
 
\begin{lemm}
\label{L:Underlying}
Assume that $(\ww_0 \wdots \ww_\nn)$ is a sequence of $\SS$-words such that, for every~$\kk$, some letter of~$\SS$ underlies~$\ww_\kk$ and~$\ww_{\kk+1}$. Then some letter underlies all of~$\ww_0 \wdots \ww_\nn$.
\end{lemm}

\begin{proof}
We use induction on~$\nn$. The result is obvious for $\nn \le 1$. Assume $\nn \ge 2$. By induction hypothesis, there exists~$\ss$ underlying~$\ww_0$ and~$\ww_1$, and $\ss'$ underlying $\ww_1 \wdots \ww_\nn$. So there exist~$\ii, \jj$ such that $\II(\ww_1)$ is both $\NN^\ii(\ss)$ and~$\NN^\jj(\ss')$. Assume first $\ii \le \jj$. The injectivity of the map~$\NN$ implies $\ss = \NN^{\jj-\ii}(\ss')$. Hence $\ss'$ underlies~$\ww_0$ as well, and, therefore, $\ss'$ underlies all of~$\ww_0 \wdots \ww_\nn$. The argument is similar in the case $\ii \ge \jj$, with now $\ss' = \NN^{\ii-\jj}(\ss')$ and $\ss$ underlying all of~$\ww_0 \wdots \ww_\nn$.
\end{proof}

\begin{lemm}
\label{L:Derivation1}
Assume that $\uu, \vv$ are nonempty $\SS$-words satisfying $\uu \EqpR1 \vv$. Then there exists~$\ss$ underlying~$\uu$ and~$\vv$ and, for every~$\ss$ in~$\SS$, we have $\EE_\ss(\uu) \EqpR{\le1} \EE_\ss(\vv)$. Moreover, exactly one of the following holds:

\ITEM1 we have $\II(\uu) = \II(\vv)$ and $\TT(\uu) \EqpR1 \TT(\vv)$; 

\ITEM2 we have $\II(\uu) \not= \II(\vv)$ and $\EE_\ss(\uu) =\nobreak \EE_\ss(\vv)$ for all~$\ss$ underlying $\uu$ and~$\vv$. 
\end{lemm}

\begin{proof}
The assumption that $\uu \EqpR1 \vv$ holds means that there exist a number~$\pp \ge 1$ and a relation of~$\RR$ such that $\vv$ is obtained from~$\uu$ by applying that relation to its subword starting at position~$\pp$. Assume first $\pp \ge 2$. In this case, the initial letter is not changed, \ie, we have $\II(\uu) = \II(\vv)$, whereas $\TT(\vv)$ is obtained from~$\TT(\uu)$ by applying a relation of~$\RR$ at position~$\pp-1$, and $\TT(\uu) \EqpR1 \TT(\vv)$ holds. Next, underlying a word~$\ww$ depends on the initial letter of~$\ww$  only, hence, as $\uu$ and~$\vv$ have the same initial letter, the letters underlying~$\uu$ and $\vv$ coincide. Finally, let $\ss$ belong to~$\SS$. Assume first that $\II(\uu)$, which is also~$\II(\vv)$, is~$\NN^\ii(\ss)$. Then, by definition, we have $\EE_\ss(\uu) = \ss\CC^\ii(\ss) \TT(\uu)$ and $\EE_\ss(\vv) = \ss\CC^\ii(\ss) \TT(\vv)$, so that $\TT(\uu) \EqpR1 \TT(\vv)$ implies $\EE_\ss(\uu) \EqpR1 \EE_\ss(\vv)$. Otherwise, we have $\EE_\ss(\uu) = \uu$ and $\EE_\ss(\vv) = \vv$, whence $\EE_\ss(\uu) \EqpR1 \EE_\ss(\vv)$ again. Hence $\EE_\ss(\uu) \EqpR1 \EE_\ss(\vv)$ holds for every~$\ss$ in this case.

Assume now $\pp = 1$. This means that there exists~$\ss$ and~$\ww$ satisfying $\uu = \ss\CC(\ss) \ww$ and $\vv = \NN(\ss) \ww$, or {\it vice versa}. In this case, we have $\II(\uu) = \ss \not= \NN(\ss) = \II(\vv)$ and, by definition, $\ss$ underlies both~$\uu$ and~$\vv$. Now, assume that $\ss'$ is any element of~$\SS$ that underlies~$\uu$. This means that we have $\II(\uu) = \ss = \NN^\ii(\ss')$ for some~$\ii \ge 0$, and we then have $\II(\vv) = \NN(\ss) = \NN^{\ii+1}(\ss')$, so $\ss'$ underlies~$\vv$ as well. Then we find 
$$\EE_{\ss'}(\uu) = \ss' \CC^\ii(\ss') \TT(\uu) = \ss' \CC^\ii(\ss') \CC(\ss) \ww = \ss' \CC^{\ii+1}(\ss') \ww = \EE_{\ss'}(\vv).$$
On the other hand, assume that $\ss'$ is an element of~$\SS$ that does not underlie~$\uu$. Then we have $\EE_{\ss'}(\uu) = \uu$. If $\ss'$ underlies~$\vv$, owing to the fact that $\II(\vv)$ is~$\NN(\II(\uu))$ and $\NN$ is injective, the only possibility is $\ss' = \NN(\ss)$ and, in this case, we have $\EE_{\ss'}(\vv) = \vv$. If $\ss'$ does not underlie~$\vv$, by definition we have $\EE_{\ss'}(\vv) = \vv$ as well. So, in every such case, we find $\EE_{\ss'}(\uu) = \uu \EqpR1 \vv = \EE_{\ss'}(\vv)$. 

The case $\uu = \NN(\ss) \ww$, $\vv= \ss\CC(\ss) \ww$ is of course similar. Then the proof is complete since the relation $\EE_{\ss}(\uu) \EqpR{\le1} \EE_{\ss}(\vv)$ has been established for every~$\ss$ in every case. 
\end{proof}

\begin{lemm}
\label{L:Derivation2}
Assume that $\uu, \vv$ are nonempty $\SS$-words satisfying $\uu \EqpR\nn \vv$. Then at least one of the following holds:

- we have $\II(\uu) = \II(\vv)$ and $\TT(\uu) \EqpR\nn \TT(\vv)$; 

- there exists~$\ss$ underlying $\uu$ and~$\vv$ and satisfying $\EE_\ss(\uu) \EqpR{<\nn} \EE_\ss(\vv)$. 
\end{lemm}

\begin{proof}
Let $(\ww_0 \wdots \ww_\nn)$ be an $\RR$-derivation from~$\uu$ to~$\vv$. Two cases are possible. Assume first that the initial letter never changes in the considered derivation, \ie, $\II(\ww_\kk) = \II(\uu)$ holds for every~$\kk$. Then all one step derivations $(\ww_\kk, \ww_{\kk+1})$ correspond to case~\ITEM1 in Lemma~\ref{L:Derivation1}. The latter implies $\II(\ww_\kk) = \II(\ww_{\kk+1})$ and $\TT(\ww_\kk) \EqpR1 \TT(\ww_{\kk+1})$ for every~$\kk$, whence $\II(\uu) = \II(\vv)$ and $\TT(\uu) \EqpR\nn \TT(\vv)$.

Assume now that the initial letter changes at least once in $(\ww_0 \wdots \ww_\nn)$, say $\II(\ww_\ii) \not=\nobreak \II(\ww_{\ii+1})$. First, Lemma~\ref{L:Derivation1} together with Lemma~\ref{L:Underlying} implies the existence of~$\ss$ in~$\SS$ that underlies~$\ww_\kk$ for every~$\kk$. Next, each one step derivation $(\ww_\ii, \ww_{\ii+1})$ corresponds to case~\ITEM2 in Lemma~\ref{L:Derivation1}. So, as $\ss$ underlies~$\ww_\ii$ and~$\ww_{\ii+1}$, we have $\EE_\ss(\ww_\ii) = \EE_\ss(\ww_{\ii+1})$. On the other hand, by Lemma~\ref{L:Derivation1} again, we have $\EE_\ss(\ww_\kk) \EqpR{\le1} \EE_\ss(\ww_{\kk+1})$ for $\kk \not= \ii$, so, summing up, we obtain $\EE_\ss(\uu) \EqpR{<\nn} \EE_\ss(\vv)$ for this particular choice of~$\ss$.
\end{proof}

We can now complete the argument establishing that the presentation $(\SS, \RRh)$ is complete for right-reversing. We denote by~$\Lg\ww$ the length (number of letters) of a word~$\ww$.

\begin{proof}[Proof of Proposition~\ref{P:Complete}]
We show using induction on~$\nn \ge 0$ and, for a given value of~$\nn$, on $\max(\Lg\uu, \Lg\vv)$, that $\uu \EqpR\nn \vv$ implies $\uu\inv \vv \revRh \ew$. 

Assume first $\nn = 0$. Then the assumption implies $\uu = \vv$, in which case $\uu\inv \vv$ reverses to the empty word by $\Lg\uu$ successive deletions of subwords~$\ss\inv \ss$.

Assume now $\nn \ge 1$. Then $\uu$ and~$\vv$ must be nonempty. Assume first that $\II(\uu) =\nobreak \II(\vv)$ and $\TT(\uu) \EqpR\nn \TT(\vv)$ hold.
By definition, we have 
$$\max(\Lg{\TT(\uu)}, \Lg{\TT(\vv)}) =\nobreak \max(\Lg\uu, \Lg\vv) -1,$$
so the induction hypothesis implies $\TT(\uu)\inv \TT(\vv) \revRh \ew$. On the other hand, as $\II(\uu)$ and $\II(\vv)$ are equal, we have
$$\uu\inv \vv = \TT(\uu)\inv \II(\uu)\inv \II(\vv) \TT(\vv) \RevRh1 \TT(\uu)\inv \TT(\vv).$$
By transitivity of reversing, we deduce $\uu\inv \vv \revRh \ew$.

Assume now that $\II(\uu) =\nobreak \II(\vv)$ and $\TT(\uu) \EqpR\nn \TT(\vv)$ do not hold. Then, by Lemma~\ref{L:Derivation2}, there must exist~$\ss$ such that $\ss$ underlies~$\uu$ and~$\vv$ and $\EE_\ss(\uu) \EqpR{\nn'} \EE_\ss(\vv)$ holds for some~$\nn' <\nobreak \nn$. Then the induction hypothesis implies $\EE_\ss(\uu)\inv \EE_\ss(\vv) \revRh\nobreak \ew$. Write $\ss_\kk$ for~$\NN^\kk(\ss)$. As $\ss$ underlies~$\uu$ and~$\vv$, there exist~$\ii, \jj$ satisfying $\II(\uu) = \ss_\ii$ and $\II(\vv) = \ss_\jj$. Assume for instance $\ii \le \jj$. By definition, we have $\EE_\ss(\uu) = \ss \CC^\ii(\ss) \TT(\uu)$ and $\EE_\ss(\vv) =\nobreak \ss \CC^\jj(\ss) \TT(\vv) = \ss \CC^\ii(\ss) \CC^{\jj-\ii}(\ss_\ii) \TT(\vv)$, so that, if $\ell$ is the length of the word~$\ss \CC^\ii(\ss)$, the first $\ell$~steps in any reversing sequence starting from $\EE_\ss(\uu)\inv \EE_\ss(\vv)$ must be
\begin{align*}
\EE_\ss(\uu)\inv \EE_\ss(\vv) 
= \TT(\uu)\inv \CC^\ii(\ss)\inv \ss\inv \ss \CC^\ii(\ss) 
&\CC^{\jj-\ii}(\ss_\ii) \TT(\vv) \\
&\RevRh{\ell}\  \TT(\uu)\inv \CC^{\jj-\ii}(\ss_\ii) \TT(\vv).
\end{align*}
It follows that the relation $\EE_\ss(\uu)\inv \EE_\ss(\vv) \revRh\nobreak \ew$ deduced above from the induction hypothesis implies 
\begin{equation}
\label{E:Derivation2}
\TT(\uu)\inv \CC^{\jj-\ii}(\ss_\ii) \TT(\vv) \ \revRh\  \ew.
\end{equation}
Now, let us consider the $\RRh$-reversing of~$\uu\inv \vv$, \ie, of $\TT(\uu)\inv \ss_\ii\inv \ss_\jj \TT(\vv)$. By definition, the only relation of~$\RRh$ of the form $\ss_\ii ... = \ss_\jj ...$ is $\ss_\ii \CC^{\jj-\ii}(\ss_\ii) = \ss_\jj$, so the first step in the reversing must be $\TT(\uu)\inv \ss_\ii\inv \ss_\jj \TT(\vv) \revRh \TT(\uu)\inv \CC^{\jj-\ii}(\ss_\ii) \TT(\vv)$. Concatenating this with~\eqref{E:Derivation2}, we deduce $\uu\inv \vv \revRh \ew$ again, which completes the induction.
\end{proof}

With the above completeness at hand, we are now ready for assembling pieces and establishing the Main Lemma (Proposition~\ref{P:MainLemma}).

\begin{proof}[Proof of the Main Lemma]
Put $\MM = \Mon\SS\RR$. If $\MM$ is of right-$O$-type, any two elements of~$\MM$ are comparable with respect to left-divisibility, hence they certainly admit a common right-multiple, namely the larger of them. So \ITEM1 trivially implies~\ITEM2.

Conversely, assume that any two elements of~$\MM$ admit a common right-multiple. First, as $\MM$ admits a right-triangular presentation, hence a positive presentation, $1$ is the only invertible element in~$\MM$.

Next, $\MM$ must be left-cancellative. Indeed, the point is to prove that, if $\ss$ belongs to~$\SS$ and $\uu, \vv$ are $\SS$-words satisfying $\ss\uu \eqpR \ss\vv$, then $\uu \eqpR \vv$ holds. By Proposition~\ref{P:Complete}, the presentation~$(\SS, \RRh)$ is complete for right-reversing. Hence $\ss\uu \eqpR \ss\vv$ implies $(\ss\uu)\inv (\ss\vv) \revRh\nobreak \ew$, \ie, $\uu\inv \ss\inv \ss \vv \revRh \ew$. Now, the first step in any reversing sequence from $\uu\inv \ss\inv \ss \vv$ is $\uu\inv \ss\inv \ss \vv \revRh \uu\inv \vv$, so the assumption implies $\uu\inv \vv \revRh \ew$, whence $\uu \eqpR \vv$.

Finally, let $\gg, \hh$ be two elements of~$\MM$. Hence, by Lemma~\ref{L:Common}, which is relevant as $(\SS, \RRh)$ is complete for right-reversing, there exist $\SS$-words~$\uu, \vv$ representing~$\gg$ and~$\hh$ and such that the $\RRh$-reversing of~$\uu\inv \vv$ is terminating, \ie, there exist $\SS$-words~$\uu', \vv'$ satisfying $\uu\inv \vv \revRh \vv' \uu'{}\inv$. By construction, the family~$\RRh$ consists of triangular relations so, by Lemma~\ref{L:Empty}, at least one of the words~$\uu', \vv'$ is empty. This means that at least one of $\gg \dive \hh$ or $\hh \dive \gg$ holds in~$\MM$, \ie, $\gg$ and~$\hh$ are comparable with respect to left-divisibility. So $\MM$ is a monoid of right-$O$-type, and \ITEM2 implies~\ITEM1.
\end{proof}

To conclude this section, we observe in view of future examples that the triangular presentations defining monoids admitting common right-multiples must be of some simple type.

\begin{lemm}
\label{L:Enum}
Assume that $\Pres\SS\RR$ is a right-triangular presentation defining a monoid in which any two elements admit a common right-multiple. Then there exists a (finite or infinite) interval~$\II$ of~$\ZZZZ$ such that $\SS$ is $\{\tta_\ii \mid \ii \in \II\}$ and $\RR$ consists of one relation $\tta_{\ii-1} =\nobreak \tta_\ii \CC(\tta_\ii)$ for each non-minimal~$\ii$ in~$\II$.
\end{lemm}

\begin{proof}
Let $\ss, \ss'$ belong to~$\SS$. As $\ss$ and $\ss'$ admit a common right-multiple, Lemma~\ref{L:Common} implies that the $\RRh$-reversing of~$\ss\inv \ss'$ terminates, which in turn requires that $\RRh$ contains at least one relation of the form $\ss ... = \ss' \!...$\ . Hence the left-graph of~$\Pres\SS\RR$ consists of a unique chain, which, by definition of a right-triangular presentation, means that $\Pres\SS\RR$ has the form stated in the lemma.
\end{proof}


\section{Existence of common multiples}
\label{S:Mult}

Owing to the Main Lemma (Proposition~\ref{P:MainLemma}), the point for establishing that a monoid specified by a right-triangular presentation is of right-$O$-type is to prove that any two elements admit a common right-multiple. We now establish several sufficient conditions in this direction.


\subsection*{Top words and ceiling}

The first criterion deals with the existence of what will be called top words. We recall that, for $\ww$ a word, $\Lg\ww$ denotes the length of~$\ww$.

\begin{defi}
\label{D:Ceiling}
Assume that $\MM$ is a monoid generated by a set~$\SS$. An $\SS$-word~$\ww$ is called a \emph{right-top} word in~$\MM$ if $\gg \dive \clp\ww$ holds for every $\gg$ in~$\SS^{\Lg\ww}$. A \emph{right-$\SS$-ceiling} for~$\MM$ is a left-infinite $\SS$-word $... \ss_2 \ss_1$ such that, for every~$\nn$, the word $\ss_\nn \pdots \ss_1$ is a right-top $\SS$-word.
\end{defi}

The criterion we obtain is as follows.

\begin{prop}
\label{P:CeilingO}
Assume that $\MM$ is a monoid that admits a triangular presentation based on a finite set~$\SS$. Then the following are equivalent:

\ITEM1 The monoid~$\MM$ is of right-$O$-type;

\ITEM2 The monoid~$\MM$ admits a right-$\SS$-ceiling;

\ITEM3 There exist $\ss_1, \ss_2, ...$ in~$\SS$ satisfying
\begin{equation}
\label{E:Ceiling}
\forall\nn\ \forall\ss{\in}\SS\, (\ss \, \ss_{\nn-1} \pdots \ss_1 \dive \ss_\nn \ss_{\nn-1}\pdots \ss_1);
\end{equation}

\ITEM4 The monoid~$\MM$ admits right-top $\SS$-words of unbounded length.
\end{prop}

We naturally say that a presentation is \emph{based on~$\SS$} if it has the form~$\Pres\SS\RR$ for some~$\RR$. Before proving Proposition~\ref{P:CeilingO}, we begin with easy preliminary observations about top words and ceilings.

\begin{lemm}
\label{L:Max}
Assume that $\MM$ is a left-cancellative monoid generated by a set~$\SS$.

\ITEM1 If $\ww$ is a right-top $\SS$-word in~$\MM$, then $\gg \dive \clp\ww$ holds for every $\gg$ in~$\SS^{\le\Lg\ww}$.

\ITEM2 A final fragment of a right-top word is a right-top word.

\ITEM3 A left-infinite $\SS$-word $... \ss_2 \ss_1$ is a right-$\SS$-ceiling in~$\MM$ if and only \eqref{E:Ceiling} is satisfied.
\end{lemm}

\begin{proof}
\ITEM1 Assume that $\ww$ is a right-top word and we have $\gg \in \SS^\nn$ with $\nn \le \Lg\ww$. Let $\ss$ be an arbitrary element of~$\SS$. Then $\gg \ss^{\Lg\ww-\nn}$ belongs to~$\SS^{\Lg\ww}$, and we have $\gg \dive \gg \ss^{\Lg\ww-\nn} \dive \clp\ww$, whence $\gg \dive \clp\ww$.

\ITEM2 Assume that $\ww$ is a right-top word and $\vv$ is a final fragment of~$\ww$, say $\ww = \uu \vv$. Let $\hh$ belong to~$\SS^{\Lg\vv}$. By construction, $\clp\uu \hh$ belongs to~$\SS^{\Lg\ww}$, so $\clp\uu \hh \dive \clp\ww$ holds. As $\MM$ is left-cancellative and $\clp\ww = \clp\uu \clp\vv$ holds, we deduce $\hh \dive \clp\vv$, so $\vv$ is a right-top word.

\ITEM3 Assume that $...\ss_2 \ss_1$ is a right-$\SS$-ceiling in~$\MM$. Then, for all~$\nn$ and~$\ss$ in~$\SS$, the element $\ss\ss_{\nn-1} \pdots \ss_1$ belongs to~$\SS^\nn$ hence, by definition, $\ss\ss_{\nn-1} \pdots \ss_1 \dive \ss_\nn\ss_{\nn-1} \pdots \ss_1$ holds, whence~\eqref{E:Ceiling}. 

Conversely, assume that $\ss_1, \ss_2, ...$ satisfy~\eqref{E:Ceiling}. We prove using induction on~$\nn$ that $\gg \in \SS^\nn$ implies $\gg \dive \ss_\nn \pdots \ss_1$. For $\nn = 1$, \eqref{E:Ceiling} directly gives $\ss \dive \ss_1$ for every~$\ss$ in~$\SS$. Assume $\nn \ge 2$ and $\gg \in \SS^\nn$. Write $\gg = \ss\hh$ with $\ss \in \SS$ and $\hh \in \SS^{\nn-1}$. The induction hypothesis gives $\hh \dive \ss_{\nn-1} \pdots \ss_1$, and then \eqref{E:Ceiling} implies $\gg  = \ss\hh \dive \ss \ss_{\nn-1} \pdots \ss_1 \dive \ss_\nn \ss_{\nn-1} \pdots \ss_1$.
\end{proof}

\begin{lemm}
\label{L:CeilingUnique}
Assume that $\MM$ is a cancellative monoid with no nontrivial invertible element, and $\SS$ generates~$\MM$. Then, for every~$\ell$, there exists at most one right-top $\SS$-word of length~$\ell$, and there exists at most one right-$\SS$-ceiling in~$\MM$; when the latter exists, every right-top word is a final fragment of it.
\end{lemm}

\begin{proof}
First, if $\ww, \ww'$ are right-top $\SS$-words of the same length, then, by definition, we have $\clp\ww \dive\nobreak \clp{\ww'}$ and $\clp{\ww'} \dive \clp\ww$, whence $\clp\ww = \clp{\ww'}$ owing to the assumption that there is no nontrivial invertible element in~$\MM$. 

Now, we prove the uniqueness of a right-top word of length~$\ell$ using induction on~$\ell$. For $\ell = 0$, the empty word is the unique word of length~$\ell$. For $\ell = 1$, if $\ss$ and $\ss'$ are distinct elements of~$\SS$, then $\clp\ss = \clp{\ss'}$ is impossible by hypothesis. Assume now that $\ww, \ww'$ are right-top words of length~$\ell$ with $\ell \ge 2$. Write $\ww = \ss\vv$, $\ww' = \ss' \vv'$. By Lemma~\ref{L:Max}\ITEM2, $\vv$ and~$\vv'$ are right-top words, so, by induction hypothesis, $\vv = \vv'$ holds. By the above observation, we have $\clp\ww = \clp{\ww'}$. Then the assumption that $\MM$ is right-cancellative implies $\clp{\ss} = \clp{\ss'}$, hence $\ss = \ss'$, and, finally, $\ww = \ww'$.

Next, assume that $\WW, \WW'$ are right-$\SS$-ceilings in~$\MM$. Then, for every~$\nn$, the length~$\nn$ final fragments of~$\WW$ and~$\WW'$ are right-top words of length~$\nn$, and, therefore, by the above result, they coincide. As this holds for every~$\nn$, so do~$\WW$ and~$\WW'$. By the same argument, if $\ww$ is a length~$\nn$ right-top word and $\WW$ is a right-$\SS$-ceiling, then the length~$\nn$ final fragment of~$\WW$ is a right-top word, hence it coincides with~$\ww$. So, $\ww$ is a final fragment of~$\WW$.
\end{proof}

\begin{proof}[Proof of Proposition~\ref{P:CeilingO}]
First, \ITEM2, \ITEM3, and~\ITEM4 are equivalent. Indeed, the equivalence of~\ITEM2 and~\ITEM3 directly follows from Lemma~\ref{L:Max}\ITEM3. Next, \ITEM2 implies~\ITEM4 by definition. Conversely, the assumption that $\MM$ admits a triangular presentation implies that $\MM$ is cancellative and admits no nontrivial invertible element. If $\MM$ admits right-top $\SS$-words of unbounded length, then, by Lemma~\ref{L:CeilingUnique}\ITEM1, there exists a well defined left-infinite word~$\WW$ such that the latter all are final fragments of~$\WW$. Then $\WW$ is a right-$\SS$-ceiling, and \ITEM4 implies~\ITEM2.

Assume now that $\MM$ is of right-$O$-type. By assumption, the relation~$\dive$ is a linear ordering on~$\MM$. Hence, as $\SS$ is finite, we can inductively select elements~$\ss_1, \ss_2, ...$ in~$\SS$ such that $\ss_\nn(\ss_{\nn-1} \pdots \ss_1)$ is the upper bound of the elements $\ss\, (\ss_{\nn-1} \pdots \ss_1)$ with~$\ss$ in~$\SS$. Then \eqref{E:Ceiling} holds, and \ITEM1 implies~\ITEM3, hence~\ITEM2 and~\ITEM4.

Conversely, assume that $\MM$ admits right-top words of unbounded length. Let $\gg, \hh$ be two elements of~$\MM$. As $\SS$ generates~$\MM$, there exist~$\nn, \pp$ such that $\gg$ lies in~$\SS^\nn$ and $\hh$ lies in~$\SS^\pp$. Let $\ww$ be a right-top word of length at least equal to~$\nn$ and~$\pp$. By Lemma~\ref{L:Max}\ITEM1, $\clp\ww$ is a common right-multiple of~$\gg$ and~$\hh$. So any two elements of~$\MM$ admit a common right-multiple and, therefore, by the Main Lemma (Proposition~\ref{P:MainLemma}), $\MM$ is of right-$O$-type. So \ITEM4 implies~\ITEM1.
\end{proof}

Note that, in the above proof, the hypothesis that the monoid admits a triangular presentation is not used to establish the existence of the ceiling: in every cancellative monoid of right-$O$-type, hence in particular in every monoid of $O$-type, generated by a finite set, there exists a unique ceiling. However, for the converse direction (the one we are mainly interested in), the existence of a ceiling is sufficient to guarantee that $\dive$ is a linear ordering only when the Main Lemma is valid, hence, in practice, for monoids that admit a (right)-triangular presentation.


\subsection*{Dominating elements}

The approach of Proposition~\ref{P:CeilingO} is concrete and makes experiments easy (see Section~\ref{S:Exp}). However, it is in general not easy to prove the existence of a ceiling explicitly, and it is often more convenient to use the following notion.

\begin{defi}
\label{D:Domin}
Assume that $\MM$ is a monoid. For $\delta, \gg$ in~$\MM$, we say that $\delta$ \emph{right-dominates}~$\gg$ 
if we have
\begin{equation}
\label{E:Domin}
\forall\nn{\ge} 0\ (\,\gg\delta^{\nn} \mathbin{\dive} \delta^{\nn+1}\,)
\end{equation}
For $\SS \subseteq \MM$, we say that $\delta$ right-dominates~$\SS$ if it right-dominates every element of~$\SS$. 
\end{defi}

Note that an element always dominates oneself. The counterpart of Proposition~\ref{P:CeilingO} is now an implication, not an equivalence.

\begin{prop}
\label{P:DominO}
Assume that $\MM$ is a monoid that admits a right-triangular presentation based on a set~$\SS$, and there exists in~$\MM$ an element that right-dominates~$\SS$. Then $\MM$ is of right-$O$-type.
\end{prop}

The result will directly follow from

\begin{lemm}
\label{L:Domin}
Assume that $\MM$ is a monoid and $\delta$ is an element of~$\MM$ that right-dominates some subset~$\SS$ of~$\MM$. Then $\gg \dive \delta^\nn$ holds for every~$\nn$ and every~$\gg$ in~$\SS^\nn$, and any two elements in the submonoid of~$\MM$ generated by~$\SS$ admit a common right-multiple. In particular, if $\SS$ generates~$\MM$, any two elements of~$\MM$ admit a common right-multiple.
\end{lemm}

\begin{proof}
We prove using induction on~$\nn$ that $\gg \in \SS^\nn$ implies $\gg \dive \delta^\nn$. For $\nn = 0$, \ie, for $\gg = 1$, the property is obvious and, for $\nn = 1$, it is~\eqref{E:Domin}. Assume $\nn \ge 2$ and $\gg \in \SS^\nn$. Write $\gg = \ss \hh$ with $\ss \in \SS$ and $\hh \in \SS^{\nn-1}$. By induction hypothesis, we have $\hh \dive \delta^{\nn-1}$, whence $\gg = \ss \hh \dive \ss \delta^{\nn-1}$, whence $\gg \dive \delta^\nn$ by~\eqref{E:Domin}. 

Now, let $\MM'$ be the submonoid of~$\MM$ generated by~$\SS$, \ie, the union of $\{1\}$ and all $\SS^\nn$ for $\nn \ge 1$. Let $\gg, \hh$ belong to~$\MM' {\setminus} \{1\}$. There exist $\nn, \pp$ such that $\gg$ belongs to~$\SS^\nn$ and $\hh$ belongs to~$\SS^\pp$. Then $\delta^{\max(\nn, \pp)}$ is a common right-multiple of~$\ff$ and~$\gg$.
\end{proof}

\begin{proof}[Proof of Proposition~\ref{P:DominO}]
By hypothesis, the set~$\SS$ generates~$\MM$. So Lemma~\ref{L:Domin} implies that any two elements of~$\MM$ admit a common right-multiple, and the Main Lemma (Proposition~\ref{P:MainLemma} then implies that $\MM$ is of right-$O$-type.
\end{proof}

We add two observations. The first one says that, in order to establish that an element dominates another one, it can be enough to consider exponents lying in an arithmetic series.

\begin{lemm}
\label{L:DominBis}
For $\delta, \gg$ in a monoid~$\MM$, a necessary and sufficient condition for $\delta$ to right-dominate~$\gg$ is that there exist~$\mm \ge 1$ satisfying
\begin{equation}
\label{E:DominBis}
\forall\kk{\ge}0\ (\,\gg\delta^{\kk\mm + \mm-1} \mathbin{\dive} \delta^{\kk\mm +1}\,).
\end{equation}
\end{lemm} 

\begin{proof}
If $\delta$ right-dominates~$\gg$, then, by definition, \eqref{E:DominBis} holds with $\mm = 1$. 

Conversely, assume \eqref{E:DominBis}. Let $\nn$ be a nonnegative integer. Let $\kk$ be maximal with $\kk\mm \le \nn$. Then we have $\nn \le \kk\mm + \mm - 1$, and \eqref{E:DominBis} implies $\gg\delta^\nn \dive \gg\delta^{\kk\mm + \mm - 1} \dive \delta^{\kk\mm + 1} \dive \delta^{\nn+1}$, so $\delta$ right-dominates~$\ss$.
\end{proof}

The second observation is a connection between ceiling and dominating element.

\begin{lemm}
\label{L:CeilingDomin}
Assume that $\MM$ is a cancellative monoid that admits no nontrivial invertible element and is generated by a set~$\SS$. Then the following are equivalent:

\ITEM1 The monoid~$\MM$ admits a right-$\SS$-ceiling that is periodic with period $\ss_\ell \pdots \ss_1$;

\ITEM2 The element~$\ss_\ell \pdots \ss_1$ right-dominates~$\SS^\ell$ in~$\MM$.
\end{lemm}

\begin{proof}
Assume \ITEM1. Put $\delta = \ss_\ell \pdots \ss_1$. For all~$\nn$ and~$\gg$ in~$\SS^\ell$, the element $\gg \delta^\nn$ belongs to~$\SS^{\ell(\nn+1)}$, so, by definition, we have $\gg \delta^\nn \dive \ss_{\ell(\nn+1)} \pdots \ss_1 = (\ss_\ell \pdots \ss_1)^{\nn+1} = \delta^{\nn+1}$. So $\delta$ right-dominates~$\SS^\ell$, and \ITEM1 implies~\ITEM2.

Conversely, assume that $\ss_\ell \pdots \ss_1$ right-dominates~$\SS^\ell$, and let~$\gg$ belong to~$\SS^{\nn\ell}$. As $\gg$ belongs to~$(\SS^\ell)^\nn$, Lemma~\ref{L:Domin} implies $\gg \dive \delta^\nn$, \ie, $\gg \dive \clp{(\ss_\ell \pdots \ss_1)^\nn}$. This shows that, for every~$\nn$, the $\SS$-word $(\ss_\ell \pdots \ss_1)^\nn$ is a right-top word. By Proposition~\ref{P:CeilingO}, we deduce that there exists a right-$\SS$-ceiling in~$\MM$ and, by Lemma~\ref{L:CeilingUnique}\ITEM2, that this sequence is $\linfty(\ss_\ell \pdots \ss_1)$.
\end{proof}

An important special case of Lemma~\ref{L:CeilingDomin} is when some element~$\ss$ of~$\SS$ dominates~$\SS$, \ie, when the top generator of~$\SS$ dominates~$\SS$. By Lemma~\ref{L:CeilingDomin}, this is equivalent to the existence of a constant ceiling $\linfty\ss$. We shall see in Sections~\ref{S:Exp} and~\ref{S:Fam} that this situation often occurs. However, there are also cases when the ceiling is not constant, corresponding to a dominating element~$\delta$ that belongs to~$\SS^\ell$ for some $\ell \ge 2$. In that case, by Lemma~\ref{L:Domin}, the condition that $\delta$ dominates~$\SS$, and not~$\SS^\ell$ as in Lemma~\ref{L:CeilingDomin}, is sufficient to establish the existence of common multiples.


\subsection*{Quasi-central elements}

Checking that an element is possibly dominating requires considering unbounded exponents. We consider now a notion that is stronger, but easier to establish.

\begin{defi}
\label{D:Quasi}
An element~$\Delta$ of a monoid~$\MM$ is called \emph{right-quasi-central} 
if there exists an endomorphism~$\phi$ of~$\MM$ such that, for every~$\gg$ in~$\MM$, we have
\begin{equation}
\label{E:Quasi}
\gg \Delta = \Delta \phi(\gg). 
\end{equation}
\end{defi}

When $\phi$ is the identity, we recover the standard notion of a central element, \ie, one that commutes with every element. Once again, we obtain a sufficient condition for a monoid to be of right-$O$-type.

\begin{prop}
\label{P:QuasiO}
Assume that $\MM$ is a monoid that admits a right-triangular presentation based on a set~$\SS$, and there exists in~$\MM$ a right-quasi-central element~$\Delta$ satisfying $\forall\ss{\in}\SS\,(\ss \dive\nobreak \Delta)$. Then $\MM$ is of right-$O$-type.
\end{prop}

The proof will relie on the following connection with dominating elements.

\begin{lemm}
\label{L:QuasiDomin}
Assume that $\MM$ is a left-cancellative monoid generated by a set~$\SS$ and $\delta^\mm$ is right-quasi-central in~$\MM$. Then $\delta$ right-dominates every element~$\gg$ that satisfies $\gg \delta^{\mm-1} \dive \delta$.
\end{lemm}

\begin{proof}
Put $\Delta = \delta^\mm$, and let $\phi$ be the (necessarily unique) endomorphism of~$\MM$ witnessing that $\Delta$ is right-quasi-central. First, \eqref{E:Quasi} applied with $\gg = \Delta$ gives $\delta \Delta =\nobreak \delta^{\mm+1} = \Delta \phi(\delta)$, whence $\phi(\delta) = \delta$ since $\MM$ is left-cancellative. 

Next, we claim that $\gg \dive \hh$ implies $\phi(\gg) \dive \phi(\hh)$. Indeed, by definition, $\gg \dive \hh$ implies the existence of~$\hh'$ satisfying $\gg \hh' = \hh$, whence $\phi(\gg) \phi(\hh') = \phi(\hh)$ since $\phi$ is an endomorphism. This shows that $\phi(\gg) \dive \phi(\hh)$ is satisfied. So, in particular, and owing to the above equality, $\gg \dive \delta$ implies $\phi(\gg) \dive \delta$.

Now assume $\gg \delta^{\mm-1} \dive \delta$. Then, for every~$\kk$, we find
$$\gg \delta^{\kk\mm + \mm-1} = \gg \delta^{\mm-1} \Delta^\kk = \Delta^\kk \phi^\kk(\gg\delta^{\mm-1}) \dive \Delta^\kk \phi^\kk(\delta) = \Delta^\kk \delta = \delta^{\kk\mm + 1},
$$
and, by Lemma~\ref{L:DominBis}, we conclude that $\delta$ right-dominates~$\gg$.
\end{proof}

\begin{proof}[Proof of Proposition~\ref{P:QuasiO}]
Under the hypotheses, by Lemma~\ref{L:QuasiDomin}, the element~$\Delta$ right-dominates~$\SS$. Then Proposition~\ref{P:DominO} implies that $\MM$ is of right-$O$-type.
\end{proof}

The main interest of considering quasi-central elements here is that, due to the following characterization, establishing that an element is quasi-central is easy, involving in particular finitely many verifications only.

\begin{lemm}
\label{L:Quasi}
Assume that $\MM$ is a left-cancellative monoid generated by a set~$\SS$. Then, for every~$\Delta$ in~$\MM$, the following are equivalent:

\ITEM1 The element $\Delta$ is right-quasi-central and satisfies $\ss \dive \Delta$ for every~$\ss$ in~$\SS$,

\ITEM2 The relation $\forall\ss{\in}\SS\,(\ss \dive \Delta \dive \ss \Delta)$ holds.
\end{lemm}

\begin{proof}
Assume \ITEM1 and let $\phi$ be the witnessing endomorphism. Let $\ss$ belong to~$\SS$. By assumption, $\ss \dive \Delta$ is true. Let $\gg$ be the element satisfying $\ss \gg = \Delta$. Then \eqref{E:Quasi} implies $\ss \Delta = \Delta \phi(\ss)$, whence $\ss \Delta = \ss \gg \phi(\ss)$, and $\Delta = \gg \phi(\ss)$ since $\MM$ is left-cancellative. We deduce $\Delta = \ss\gg \dive \ss\gg \phi(\ss) = \ss\Delta$, and \ITEM1 implies~\ITEM2.

Conversely, assume \ITEM2. We shall define an endomorphism~$\phi$ witnessing that $\Delta$ is right-quasi-central in~$\MM$. First, for~$\ss$ in~$\SS$, we define~$\phi(\ss)$ to be the unique element satisfying 
\begin{equation}
\label{E:Quasi3}
\ss \Delta = \Delta \phi(\ss),
\end{equation}
which exists since, by assumption, $\Delta \dive \ss\Delta$ holds. Now, assume that $\ss_1 \wdots \ss_\nn$, $\ss'_1 \wdots\ss'_\pp$ belong to~$\SS$ and $\ss_1 \pdots \ss_\nn = \ss'_1 \pdots \ss'_\pp$ holds in~$\MM$. By applying~\eqref{E:Quasi3} repeatedly, we obtain
$$\Delta \phi(\ss_1) \pdots \phi(\ss_\nn) 
= \ss_1 \pdots \ss_\nn \Delta
= \ss'_1 \pdots \ss'_\pp \Delta
= \Delta \phi(\ss'_1) \pdots \phi(\ss'_\pp),$$
whence $\phi(\ss_1) \pdots \phi(\ss_\nn) = \phi(\ss'_1) \pdots \phi(\ss'_\pp)$ since $\MM$ is left-cancellative. It follows that, for every~$\gg$ in~$\MM {\setminus} \{1\}$, we can define~$\phi(\gg)$ to be the common value of $\phi(\ss_1) \pdots \phi(\ss_\nn)$ for all expressions of~$\gg$ as a product of elements of~$\SS$. We complete with $\phi(1) = 1$. Then, by construction, $\phi$ is an endomorphism of~$\MM$ and \eqref{E:Quasi} is satisfied for every~$\gg$ in~$\MM$. 
\end{proof}

We conclude the section with a connection between central element and ceiling.

\begin{lemm}
\label{L:CentralCeiling}
Assume that $\MM$ is a cancellative monoid of right-$O$-type, and that $\ss_\ell \pdots \ss_1$ is a right-top $\SS$-word in~$\MM$ such that $\clp{\ss_\ell \pdots \ss_1}$ is central in~$\MM$. Then we have $\ss_\ii = \ss_1$ for every~$\ii$, and $\linfty\ss_1$ is the right-$\SS$-ceiling in~$\MM$.
\end{lemm}

\begin{proof}
Let $\Delta = \ss_\ell \pdots \ss_1$. First, we have $\gg \dive \Delta$ for every~$\gg$ in~$\SS^\ell$. Then Lemma~\ref{L:QuasiDomin} implies that $\Delta$ right-dominates~$\SS^\ell$. Hence, by Lemma~\ref{L:CeilingDomin}, the right-$\SS$-ceiling is periodic with period $\ss_\ell \pdots \ss_1$. Now consider its length $\ell+1$ final fragment $\ss_1 \ss_\ell \pdots \ss_1$. Then, in $\MM$, we have $\ss_1 \ss_\ell \pdots \ss_1 = \ss_\ell \pdots \ss_1 \ss_1$, so $\ss_1 \ss_\ell \pdots \ss_1$ and $\ss_\ell \pdots \ss_1 \ss_1$ are two right-top $\SS$-words of length~$\ell + 1$. By Lemma~\ref{L:CeilingUnique}, these words must coincide, which is possible only for $\ss_1 = ... = \ss_\ell$. 
\end{proof}

Thus the only situation when a right-top word can be central is when it is a power of the top generator. Note that Lemma~\ref{L:CentralCeiling} says nothing about central elements that admit no expression as a right-top word.


\section{Back to ordered groups}
\label{S:Groups}

In Section~\ref{S:Mult}, we established several sufficient conditions for a monoid to be of right-$O$-type. Returning to the context of ordered groups, we immediately deduce orderability conditions. We recall that, if $\Pres\SS\RR$ is a positive group presentation, $\Pres\SS\RRt$ is the opposite presentation (same generators, reversed relations).

\begin{prop}
\label{P:Main1}
A sufficient condition for a group~$\GG$ to be orderable is that 
\begin{quote}
$\GG$ admits a triangular presentation~$\Pres\SS\RR$ such that the monoids~$\Mon\SS\RR$ and~$\Mon\SS\RRt$ are eligible for at least one of Propositions~\ref{P:CeilingO}, \ref{P:DominO}, or~\ref{P:QuasiO}. 
\end{quote}
In this case, the subsemigroup of~$\GG$ generated by~$\SS$ is the positive cone of a left-invariant ordering on~$\GG$. If $\SS$ is finite, this ordering is an isolated point in the space~$\LO\GG$.  
\end{prop}

\begin{proof}
Whenever the monoid~$\Mon\SS\RR$ is eligible for one of the mentioned results, it is of right-$O$-type. For a similar reason, the monoid~$\Mon\SS\RRt$ is of right-$O$-type, hence $\Mon\SS\RR$ is of left-$O$-type. Therefore, $\Mon\SS\RR$ is of $O$-type, and then we apply Corollary~\ref{C:Criterion}.
\end{proof}

Let us now address the solvability of the decision problem for the ordering involved in Proposition~\ref{P:Main1}. We shall establish the correctness of Algorithm~\ref{A:Decision} below. The latter simultaneously appeals to right-reversing as defined in Section~\ref{S:Rev} and to \emph{left-reversing}~$\revL$, the symmetric procedure that replaces $\ss'\ss\inv$ with $\vv\inv\vv'$ such that $\vv\ss' = \vv'\ss$ is a relation. The properties of left-reversing are of course symmetric to those of right-reversing: formally, using $\widetilde\ww$ for the mirror-image of~$\ww$ (same letters in reserved order), $\ww \revLR \ww'$ is equivalent to $\widetilde\ww \revRopp \widetilde\ww'$ where, as usual, $\RRt$ is the family of all relations $\widetilde\uu = \widetilde\vv$ for $\uu = \vv$ in~$\RR$. We denote by~$\RRa$ the family obtained by adding to~$\RR$ the relations~$\NNt^\ii(\ss) = \CCt^\ii(\ss) \ss$ with $\ii \ge 2$ (the left-completion of~$\RR$, symmetric to the right-completion~$\RRh$). If $\ww$ is a signed $\SS$-word, we denote by~$\cl\ww$ the element of the group~$\Gr\SS\RR$ represented by~$\ww$.

\begin{algo}[decision problem of the ordering]
\label{A:Decision}

\hfill\\
\noindent $\bullet$ {\bf Data}: A finite (or recursive) triangular presentation $\Pres\SS\RR$;

\noindent $\bullet$ {\bf Input}: A signed $\SS$-word~$\ww$;

\noindent $\bullet$ {\bf Procedure}: 

- Right-$\RRh$-reverse $\ww$ into~$\vv \uu\inv$ with $\uu, \vv$ in~$\SS^*$;

- Left-$\RRa$-reverse $\vv \uu\inv$ into~$\uu'{}\inv \vv'$ with $\uu', \vv'$ in~$\SS^*$;

\noindent $\bullet$ {\bf Output}: 

- For $\uu' \not= \ew$ and $\vv' = \ew$, return ``\;$\cl\ww <1$''; 

- For $\uu' = \vv' = \ew$, return ``\;$\cl\ww = 1$'';
 
- For $\uu' = \ew$ and $\vv' \not= \ew$, return ``\;$\cl\ww >1$''.
\end{algo}

\begin{prop}
\label{P:Main2}
Under the hypotheses of Proposition~\ref{P:Main1}, if $\Pres\SS\RR$ is finite or recursive, Algorithm~\ref{A:Decision} solves the decision problem for the ordering, and the word problem of~$\GG$.
\end{prop}

\begin{proof}
Put $\MM = \Mon\SS\RR$. First, as $\Pres\SS\RR$ is finite or, at least, recursive, the relations~$\revRh$ and~$\revLRc$ are recursive, so Algorithm~\ref{A:Decision} is indeed effective. Next, as the presentation $(\SS, \RRh)$ is complete for right-reversing and, by Proposition~\ref{P:Main1}, any two elements of the monoid~$\MM$ admit a common right-multiple, every right-$\RRh$-reversing sequence is terminating: for every signed $\SS$-word~$\ww$, there exist positive $\SS$-words~$\uu, \vv$ satisfying $\ww \revRh \uu \vv\inv$. Similarly, as $(\SS, \RRa)$ is complete for left-reversing and any two elements of~$\MM$ admit a common left-multiple, every $\RRa$-reversing sequence is terminating and, therefore, there exist positive $\SS$-words~$\uu', \vv'$ satisfying $\uu \vv\inv \revLRc \uu'{}\inv \vv'$. Hence Algorithm~\ref{A:Decision} always terminates. Moreover, by (the counterpart of) Lemma~\ref{L:Empty}, at least one of the words~$\uu', \vv'$ is empty.

By construction, $\ww \revRh \uu \vv\inv \revLRc \uu'{}\inv \vv'$ implies $\cl\ww = \cl{\uu'{}\inv \vv'}$ in~$\Gr\SS\RR$. If $\uu'$ is nonempty and $\vv'$ is empty, we deduce $\cl\ww = \cl{\uu'{}\inv} \in \MM\inv{\setminus}\{1\}$, whence $\cl\ww < 1$ for the ordering whose positive cone is $\MM\inv{\setminus}\{1\}$. If $\uu'$ and $\vv'$ are empty, we deduce $\cl\ww = \cl{\ew} = 1$. Finally, if $\uu'$ is empty and $\vv'$ is nonempty, we obtain $\cl\ww = \cl{\vv'} \in \MM{\setminus}\{1\}$, whence $\cl\ww > 1$. So Algorithm~\ref{A:Decision} decides the relation~$<$. As $<$ is a strict linear ordering, the algorithm also solves the word problem as $\cl\ww \not=1$ is equivalent to the disjunction of $\cl\ww < 1$ and $\cl\ww > 1$.
\end{proof}

\begin{rema}
Algorithm~\ref{A:Decision} does not only give a YES/NO answer: for every initial signed $\SS$-word~$\ww$, the method provides a positive $\SS$-word~$\ww'$ such that $\ww$ is equivalent either to~$\ww'$ or or to~$\ww'{}\inv$. In other words, it gives for every element~$\gg$ of the group~$\Gr\SS\RR$ an explicit positive or negative decomposition of~$\gg$ in terms of the distinguished generators of the positive cone.
\end{rema}

The proof of Theorem~1 can now be completed. We state the results in a more general form that avoids unnecessary symmetries and corresponds to the case of dominating or quasi-central elements; a similar statement involving right-ceilings and \eqref{E:Ceiling} exists of course. 

\begin{thrm}
\label{T:Main}
A sufficient condition for a group~$\GG$ to be orderable is that 
\begin{quote}
$\GG$ admits a triangular presentation~$\Pres\SS\RR$ and, in the monoid~$\Mon\SS\RR$, there exist~$\Delta$, $\Deltat$ satisfying $\forall\ss{\in}\SS\,(\ss \dive \Delta \dive \ss\Delta)$ or $\forall\ss{\in}\SS\,\, \forall\nn\, (\ss \Delta^\nn \dive \Delta^{\nn+1})$, and $\forall\ss{\in}\SS\,(\Deltat\ss \multeR \Deltat \multeR\nobreak \ss)$ or $\forall\ss{\in}\SS\,\, \forall\nn\, (\Deltat^{\nn+1} \multeR \Deltat^\nn \ss)$.
\end{quote}
In this case, the subsemigroup of~$\GG$ generated by~$\SS$ is the positive cone of a left-invariant ordering on~$\GG$. If $\SS$ is finite, this ordering is isolated in the space~$\LO\GG$. If $\Pres\SS\RR$ is finite or recursive, the word problem of~$\GG$ and the decision problem of the ordering are decidable.
\end{thrm}

\begin{proof}
Owing to Lemmas~\ref{L:QuasiDomin} and~\ref{L:Quasi}, the element~$\Delta$ right-dominates~$\SS$ in~$\Mon\SS\RR$, so the latter is eligible for Proposition~\ref{P:DominO}. Symmetrically, $\Deltat$ right-dominates~$\SS$ in~$\Mon\SS\RRt$, and $\Mon\SS\RRt$ is eligible for Proposition~\ref{P:DominO} as well. Then the hypotheses of Propositions~\ref{P:Main1} and~\ref{P:Main2} are satisfied, and the latter give the results.
\end{proof}

We add two more observations. The first one involves monoids that are of right-$O$-type but not necessarily of $O$-type. In this case, the termination of left-reversing is not guaranteed, and the monoid need not be connected with a left-invariant ordering in the group. However, the group is still a group of right-fractions for the monoid, and we can solve its word problem by appealing to right-reversing only.

\begin{algo}[word problem]
\label{A:WordPb}
\hfill\\
\noindent $\bullet$ {\bf Data}: A finite (or recursive) right-triangular presentation $\Pres\SS\RR$;

\noindent $\bullet$ {\bf Input}: A signed $\SS$-word~$\ww$;

\noindent $\bullet$ {\bf Procedure}: 

- Right-$\RRh$-reverse $\ww$ into~$\vv \uu\inv$ with $\uu, \vv$ in~$\SS^*$;

- Right-$\RRh$-reverse $\uu\inv \vv$ into~$\vv' \uu'{}\inv$ with $\uu', \vv'$ in~$\SS^*$;

\noindent $\bullet$ {\bf Output}: 

- For $\uu' = \vv' = \ew$, return ``\;$\cl\ww = 1$'';
 \enlargethispage{3mm}
 
- For $\uu' \not= \ew$ or $\vv' \not= \ew$, return ``\;$\cl\ww \not= 1$''.
\end{algo}

\begin{prop}
\label{P:DoubleRev}
Assume that $\Pres\SS\RR$ is a right-triangular presentation and there exists an element~$\Delta$ in~$\Mon\SS\RR$ satisfying $\forall\ss{\in}\SS\,(\ss \dive \Delta \dive \ss\Delta)$ or $\forall\ss{\in}\SS\,\, \forall\nn\, (\ss \Delta^\nn \dive \Delta^{\nn+1})$. Then Algorithm~\ref{A:WordPb} solves the word problem of the group~$\Gr\SS\RR$.
\end{prop}

\begin{proof}
By Proposition~\ref{P:Complete}, the presentation $(\SS, \RRh)$ is complete for right-reversing and, by Proposition~\ref{P:Main1}, any two elements of the monoid~$\Mon\SS\RR$ admit a common right-multiple, hence every $\RRh$-reversing sequence is terminating: for every signed $\SS$-word~$\ww$, there exist positive $\SS$-words~$\uu, \vv$ satisfying $\ww \revRh \uu \vv\inv$. Hence Algorithm~\ref{A:WordPb}, which consists of two concatenated reversings, always terminates. 

Then, by construction, $\ww \revRh \uu \vv\inv$ implies $\cl\ww = \cl{\uu\vv\inv}$ in~$\Gr\SS\RR$. Hence $\cl\ww = 1$ holds if and only if we have $\cl{\uu\vv\inv} = 1$, or, equivalently, $\cl\uu = \cl\vv$. By Ores's theorem, the monoid~$\Mon\SS\RR$ embeds in the group~$\Gr\SS\RR$, so the latter condition is equivalent to $\clp\uu = \clp\vv$, \ie, to~$\uu \eqpR \vv$. As $(\SS, \RRh)$ is complete for right-reversing, the latter condition is equivalent to $\uu\inv \vv \revRh \ew$, \ie, with the notation of Algorithm~\ref{A:WordPb}, to $\uu' = \vv' = \ew$. 
\end{proof}

The second observation is a connection with Garside theory~\cite{Garside}. Say that an element~$\Delta$ is \emph{left-quasi-central} in a monoid~$\MM$ if it is right-quasi-central in the opposite monoid~$\MMt$.

\begin{prop}
\label{P:Garside}
Assume that $\MM$ is a monoid of right-$O$-type and $\Delta$ is right-quasi-central (\resp simultaneously right- and left-quasi-central) in~$\MM$ and its left-divisors generate~$\MM$. Then $\Delta$ is a right-Garside element (\resp a Garside element) in~$\MM$ in the sense of~\cite[Definitions VI.1.36 and 2.29]{Garside}.
\end{prop}

\begin{proof}
By assumption, the monoid~$\MM$ is left-cancellative and the left-divisors of~$\Delta$ generate~$\MM$. As $\Delta$ is right-quasi-central, every right-divisor of~$\Delta$ is left-divides~$\Delta$ since, as noted in the proof of Lemma~\ref{L:QuasiDomin}, $\Delta = \gg' \gg$ implies $\gg \phi(\gg') = \Delta$. Finally, for every~$\gg$ in~$\MM$, the elements~$\gg$ and~$\Delta$ admit a greatest common left-divisor (left-gcd), namely the smaller of them with respect to~$\dive$. Hence, by definition, $\Delta$ is a right-Garside element in~$\MM$. 

If $\Delta$ is also left-quasi-central, then, by symmetry, the left-divisors of~$\Delta$ must be included in its right-divisors, and, therefore, the left- and right-divisors of~$\Delta$ coincide. Then $\Delta$ is a Garside element in~$\MM$.
\end{proof}

It follows that, under the hypotheses of Proposition~\ref{P:Garside}, the left-divisors of~$\Delta$ in~$\MM$ form what is called a Garside family \cite[Definition I.1.34]{Garside} and every element of~$\MM$ admits a distinguished decomposition in terms of these elements. However, as left-divisibility is a linear ordering in this case, this decomposition is rather trivial: every element is left-divisible by some maximal power of~$\Delta$, and the normal decompositions all have the simple form $(\Delta \wdots \Delta, \gg)$ with $\gg \dive \Delta$.


\section{Experimental approach}
\label{S:Exp}

Subword reversing is easily implemented, allowing for computer experiments. Thus we can consider arbitrary (right)-triangular presentations and investigate whether the associated monoids are of (right)-$O$-type. Owing to Lemma~\ref{L:Enum}, we consider presentations of form
\begin{equation}
\label{E:PresExp}
\Pres{\tta_1 \wdots \tta_\nn}{\tta_1 = \tta_2 \CC(\tta_2) \wdots \tta_{\nn-1} = \tta_\nn \CC(\tta_\nn)}
\end{equation}
where $\CC(\tta_2) \wdots \CC(\tta_\nn)$ are words in the alphabet $\{\tta_1 \wdots \tta_\nn\}$. Then the first natural step in the investigation consists in (trying to) compute (a final fragment of) the right-ceiling, if it exists. The principle is as follows.

\begin{lemm}
\label{L:CeilingExp}
Assume that $\Pres\SS\RR$ is a presentation of the form~\eqref{E:PresExp}. Starting from~$\ss_1 = \tta_1$, inductively find~$\ss_\nn$ in~$\SS$ so that, for every~$\ii$, \begin{equation}
\label{E:CeilingExp}
\mbox{$(\tta_\ii \ss_{\nn-1} \pdots \ss_1)\inv (\ss_\nn \ss_{\nn-1} \pdots \ss_1)$ is $\RRh$-reversible to a positive word}.\end{equation}
\nobreak
Then $\Mon\SS\RR$ is of right-$O$-type if and only if the construction never stops.
\end{lemm}

\begin{proof}
By Lemma~\ref{L:Empty}, and by completeness of the presentation~$\Pres\SS\RR$ for subword reversing, the only possibilities, when a negative--positive word~$\uu\inv \vv$ is reversed are

- either the reversing terminates in finitely many steps with a word that is either positive or negative, in which case $\clp\uu \dive \clp\vv$ or $\clp\vv \dive \clp\uu$ holds in~$\Mon\SS\RR$,

- or the reversing never terminates, in which case $\clp\uu$ and $\clp\vv$ have no common right-multiple in~$\Mon\SS\RR$.

\noindent So, assuming that $\ss_1 \wdots \ss_{\nn-1}$ have been found, two cases may occur in the search of~$\ss_\nn$ satisfying~\eqref{E:CeilingExp}:

- either all reversings terminate in finitely many steps, $\ss_\nn$ is found, and the process continues; then \eqref{E:Ceiling} is satisfied, $\ss_\nn \pdots \ss_1$ is a right-top $\SS$-word, and, if this continues for every~$\nn$, then, by Proposition~\ref{P:CeilingO}, $\Mon\SS\RR$ is of right-$O$-type,

- or some reversing does not terminate, and the process stops; then two elements $\ss \ss_{\nn-1} \pdots \ss_1$ and $\ss' \ss_{\nn-1} \pdots \ss_1$ admit no common right-multiple, and $\Mon\SS\RR$ cannot be of right-$O$-type. 
\end{proof}

Note that, for an alphabet~$\SS$ of size~$\mm$, determining the final $\ell + 1$ letters in the right-ceiling requires only $\ell(\nn-1)$ word reversings and, in particular, only $\ell$~reversing are required in the case of two generators.

Although easy, the previous step can lead to suspicions, and not to a proven conclusion: on the positive side, one obtains a finite fragment of the ceiling, from which it is \textit{a priori} impossible to deduce the existence of the ceiling; symmetrically, on the negative side, a long finite reversing sequence is not a proof of a non-terminating reversing. However, so far, the suspicions provided by Lemma~\ref{L:CeilingExp} never turned to be wrong: although \textit{ad hoc} examples could certainly be constructed, we know of no monoid containing a right-top word of length~20 that eventually turned to be not of $O$-type, and of no reversing sequence of length~$10,000$ that eventually turned to be terminating.

Now, the good point is that, in some cases, proven conclusions can be obtained. Let us begin with the negative case, \ie, establishing that a monoid is not of $O$-type. The following result shows that some syntactic conditions \textit{a priori} discard certain presentations.

\begin{lemm}
\label{L:Excluded}
Assume that $\Pres\SS\RR$ is a triangular presentation.

\ITEM1 If a relation of~$\RRh$ has the form $\ss = \ww$ with $\Lg\ww > 1$ and $\ww$ finishing with~$\ss$, then $\Mon\SS\RR$ is not right-cancellative and, therefore, $\Mon\SS\RR$ is not of right-$O$-type.

\ITEM2 If a relation of~$\RRh$ has the form $\ss = \ww$ with $\ww$ beginning with $(\uu\vv)^\rr \uu \ss$ with~$\rr \ge\nobreak 1$, $\uu$~nonempty, and $\vv$ such that $\vv\inv \ss$ reverses to a word beginning with~$\ss$, hence in particular if $\vv$ is empty or it can be decomposed as $\uu_1 \wdots \uu_\mm$ where $\uu_\kk\ss$ is a prefix of~$\ww$ for every~$\kk$, then the elements $\ss$ and $\clp\uu \ss$ have no common right-multiple in~$\Mon\SS\RR$ and, therefore, $\Mon\SS\RR$ is not of right-$O$-type.
\end{lemm}

\begin{proof}
\ITEM1 If $\RR$ contains a relation $\ss = \uu \ss$ with $\uu$ nonempty, $\ss = \clp\uu \ss$ holds in~$\Mon\SS\RR$, whereas $1 = \clp\uu$ fails. So $\Mon\SS\RR$ is not right-cancellative.

\ITEM2 We claim that the right-$\RRh$-reversing of $\ss\inv \uu \ss$ cannot be terminating, see Figure~\ref{F:NonTerminating}. Indeed, writing the involved relation $\ss = (\uu \vv)^\rr \uu \ss \ww_1$ with $\vv\inv \ss \revRh \ss \ww_2$, we find
\begin{align*}
\ss\inv \uu \ss 
&\ \revRh\  \ww_1\inv \ss\inv (\vv \uu)^{-(\rr-1)} \uu\inv \vv\inv \ss\\
&\ \revRh\  \ww_1\inv \ss\inv (\vv \uu)^{-(\rr-1)} \uu\inv \ss \ww_2\\
&\ \revRh\  \ww_1\inv \ss\inv (\vv \uu)^{-(\rr-1)} (\vv\uu)^{\rr-1} \vv \uu \ss \ww_1 \ww_2\\
&\ \revRh\  \ww_1\inv \ss\inv \vv \uu \ss \ww_1 \ww_2\\
&\ \revRh\  \ww_1\inv \ww_2\inv \cdot \ss\inv  \uu \ss \cdot \ww_1 \ww_2. 
\end{align*}
We deduce that $\ss\inv \uu \ss \ \revRh \ (\ww_1\inv \ww_2\inv)^\nn\cdot \ss\inv  \uu \ss \cdot (\ww_1 \ww_2)^\nn$ holds for every~$\nn$ and, therefore, it is impossible that $\ss\inv \uu \ss$ leads in finitely many steps to a positive--negative word. Then, by Lemma~\ref{L:Common}, which is relevant since, by Proposition~\ref{P:Complete}, $(\SS, \RRh)$ is complete for right-reversing, $\ss$ and $\clp\uu \ss$ admit no common right-multiple in~$\Mon\SS\RR$. 
\end{proof}

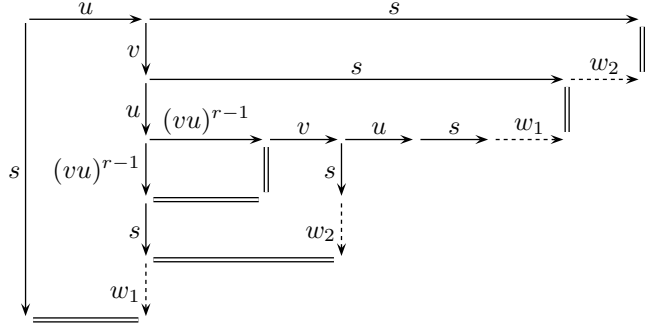
\begin{figure}[htb]
\begin{picture}(82,42)(0,0)
\pcline{->}(0,39.5)(0,0.5)
\tlput{$\ss$}
\pcline{->}(16,39.5)(16,32.5)
\tlput{$\vv$}
\pcline{->}(16,31.5)(16,24.5)
\tlput{$\uu$}
\pcline{->}(16,23.5)(16,16.5)
\tlput{$(\vv\uu)^{\rr-1}$}
\pcline{->}(16,15.5)(16,8.5)
\tlput{$\ss$}
\pcline[style=exist]{->}(16,7.5)(16,0.5)
\tlput{$\ww_1$}
\pcline{->}(0.5,40)(15.5,40)
\taput{$\uu$}
\pcline[style=double](1,0)(15,0)
\pcline{->}(16.5,24)(31.5,24)
\taput{$(\vv\uu)^{\rr-1}$}
\pcline{->}(32.5,24)(41.5,24)
\taput{$\vv$}
\pcline{->}(42.5,24)(51.5,24)
\taput{$\uu$}
\pcline{->}(52.5,24)(61.5,24)
\taput{$\ss$}
\pcline[style=exist]{->}(62.5,24)(71.5,24)
\taput{$\ww_1$}
\pcline{->}(16.5,32)(71.5,32)
\taput{$\ss$}
\pcline[style=exist]{->}(72.5,32)(81.5,32)
\taput{$\ww_2$}
\pcline{->}(16.5,40)(81.5,40)
\taput{$\ss$}
\pcline[style=double](82,39)(82,33)
\pcline[style=double](72,31)(72,25)
\pcline[style=double](32,23)(32,17)
\pcline[style=double](17,16)(31,16)
\pcline[style=double](17,8)(41,8)
\pcline{->}(42,23.5)(42,16.5)
\tlput{$\ss$}
\pcline[style=exist]{->}(42,15.5)(42,8.5)
\tlput{$\ww_2$}
\end{picture}
\caption{\sf\smaller Proof of Lemma~\ref{L:Excluded}\ITEM2: in a positive number of steps, the word~$\ss\inv \uu \ss$ reverses to a word that includes it and, therefore, the reversing cannot be terminating.}
\label{F:NonTerminating}
\end{figure}

For instance, a relation $\tta = \ttb\tta\ttb\tta\ttb^3\tta^2...$ is impossible in a right-triangular presentation for a monoid of right-$O$-type: indeed, the right-hand side of the relation can be written as $(\ttb\tta)\ttb\tta\ttb^2(\ttb\tta)\tta...$, which is eligible for Lemma~\ref{L:Excluded}\ITEM2 with $\uu = \ttb\tta$ and $\vv = \ttb\tta\ttb\cdot\ttb$, a product of two words~$\uu_1, \uu_2$ such that $\uu_\ii\tta$ is a prefix of the right-hand term of the relation.

Symmetrically, on the positive side, \ie, for establishing that a monoid is of right-$O$-type, the existence of a quasi-central element is can be checked in finite time. 

\begin{lemm}
\label{L:QuasiExp}
Assume that $\Pres\SS\RR$ is a presentation of the form~\eqref{E:PresExp}, and $\ww$ is an $\SS$-word beginning with~$\tta_1$ and such that, for $2 \le \ii \le \nn$, 
\begin{equation}
\label{E:QuasiExp}
\mbox{$\ww\inv \tta_\ii \ww$ is $\RRh$-reversible to a positive word}.
\end{equation}
Then $\clp\ww$ is right-quasi-central in~$\Mon\SS\RR$ and the latter is a monoid of right-$O$-type.
\end{lemm}

\begin{proof}
That $\clp\ww$ is right-quasi-central in~$\Mon\SS\RR$ directly follows from Lemma~\ref{L:Quasi}, since \eqref{E:QuasiExp} implies $\clp\ww \dive \tta_\ii \clp\ww$. Moreover, by construction, we have $\tta_\nn \dive \pdots \dive\nobreak \tta_1 \dive\nobreak \clp\ww$, whence $\tta_\ii \dive \clp\ww$ for every~$\ii$. Hence, by Proposition~\ref{P:QuasiO}, $\Mon\SS\RR$ is of right-$O$-type.
\end{proof}

Thus, the generic program for investigating a right-triangular presentation~$\Pres\SS\RR$ is

\noindent\ITEM1 check if $\Pres\SS\RR$ is eligible for Lemma~\ref{L:Excluded}---in which case $\Mon\SS\RR$ is not of right-$O$-type;

\noindent\ITEM2 if not, compute the right-ceiling using Lemma~\ref{L:CeilingExp} with escape conditions on the length of the reversing sequences and of the ceiling;

\noindent\ITEM3 if some reversing in step~\ITEM2 seems to be non-terminating, try to extract a cyclic reversing pattern, \ie, finding an $\SS$-word~$\uu$ satisfying $\uu \RevRh\nn ... \uu ...$ for some $\nn > 0$---in which case $\Mon\SS\RR$ is not of right-$O$-type;

\noindent\ITEM4 otherwise try to find a right-quasi-central element---in which case $\Mon\SS\RR$ is of right-$O$-type; if the ceiling seems to be periodic with period~$\ww$, then $\clp\ww$ is a natural candidate.
 
It turns out that the above program works well, at least for presentations that are short enough. In the case of two generators, almost all presentations we tried are either discarded by Lemma~\ref{L:QuasiExp} or contain a quasi-central element that is a power of the top generator.

\begin{fact}
\label{F:Exp2}
Among the 1,023 presentations $(\tta, \ttb ; \tta = \ttb\ww)$ with $\ww$ of length~$\le 9$ in $\{\tta, \ttb\}^*$, 

- 854 are eligible for Lemma~\ref{L:Excluded}, yielding a monoid not of right-$O$-type.

- 3 lead to a cyclic reversing, yielding a monoid not of right-$O$-type.

- 166 are eligible for Lemma~\ref{L:QuasiExp}, yielding a monoid of right-$O$-type (33 are of $O$-type).
\end{fact}

\begin{table}[tb]
\smaller
\begin{tabular}{lccl}
\hline
\hline
\VR(3.5,1.5)&right-$O$
&left-$O$\\
\hline

\VR(4,0)$\tta = \ttb\tta\ttb\tta\ttb \tta\ttb$
&YES 
&YES
&{\smaller$\Delta = \tta^2$ central;}\\

\VR(0,0)$\tta = \ttb\tta^2\ttb\tta\ttb \tta\ttb$
&YES 
&&{\smaller$\Delta = \tta^3$ right-quasi-central,
$\phi(\tta) = \tta$, 
$\phi(\ttb) = (\ttb\tta\ttb\tta\ttb)^3$}\\

&&NO 
&{\smaller$\tta = ... \tta (\tta\ttb)(\tta\ttb)(\tta\ttb)$: 
Lemma~$\widetilde{\ref{L:Excluded}}$ with $\uu = (\tta\ttb)$, $\vv = \ew$}\\

\VR(0,0)$\tta = \ttb\tta\ttb\tta^2\ttb \tta\ttb$
&NO 
&NO 
&{\smaller$\tta = (\ttb\tta)(\ttb\tta)\tta...$: 
Lemma~\ref{L:Excluded} with $\uu = \ttb\tta$ and $\vv = \ew$}\\

\VR(0,0)
$\tta = \ttb\tta^3\ttb\tta\ttb \tta\ttb$
&YES 
&&{\smaller$\Delta = \tta^4$ right-quasi-central,
$\phi(\tta) = \tta$, 
$\phi(\ttb) = \ttb(\tta\ttb)^8$}\\

&&NO
&{\smaller$\tta = ... \tta (\tta\ttb)(\tta\ttb)$: 
Lemma~$\widetilde{\ref{L:Excluded}}$ with $\uu = (\tta\ttb)$ and $\vv = \ew$}\\

\VR(3.5,0)$\tta = \ttb\tta\ttb^3\tta\ttb$
&YES
&YES
&{\smaller$\Delta = (\tta\ttb)^3 = (\ttb\tta)^3$ central}\\

\VR(0,0)$\tta = \ttb\tta^2\ttb\tta\ttb\tta^2\ttb$
&YES
&YES
&{\smaller$\Delta = (\tta^2\ttb)^2 = (\ttb\tta^2)^2$ right- and left-quasi-central,}\\

&&&\hspace{10mm}{\smaller
$\phi(\tta) = \tta(\ttb\tta^2\ttb)^2$, 
$\widetilde\phi(\tta) = (\ttb\tta^2\ttb)^2\tta$, 
$\phi(\ttb) = \widetilde\phi(\ttb) = \ttb$}\\

\VR(0,0)$\tta = \ttb\tta^2\ttb^3\tta^2\ttb$
&NO
&NO
&{\smaller non-terminating right-reversing: $\uu \Rev{10} \vv^{-1} \uu \vv$}\\

&&&\hspace{10mm}{\smaller for $\uu = \tta^{-2}\ttb\tta^2\ttb\tta$ and $\vv = \ttb\tta^2\ttb^3$}\\

\VR(0,0)$\tta = \ttb\tta^2\ttb\tta\ttb^2\tta^2\ttb$
&YES
&&{\smaller$\Delta = (\tta^2\ttb)^2$ right-quasi-central, 
$\phi(\tta) = (\tta\ttb^2\tta)^2\tta\ttb$, 
$\phi(\ttb) = (\ttb\tta^2\ttb^2)^2$}\\

&&NO
&{\smaller non-terminating left-reversing: $\uu\  \revL^{(26)}\, \vv \uu \vv\inv$}\\

&&&\hspace{10mm}{\smaller for $\uu = \tta^2\ttb^2\tta^2\ttb\tta\ttb^3\tta^2\ttb\tta\inv$ and $\vv = \ttb$}\\

\VR(0,0)$\tta = \ttb\tta^2\ttb^4\tta^2\ttb$
&NO
&NO
&{\smaller non-terminating right-reversing: $\uu \Rev{12} \vv^{-1} \uu \vv$}\\

\VR(0,2)
&&&\hspace{10mm}{\smaller for $\uu = \ttb^{-1}\tta^{-2}\ttb\tta^2\ttb\tta$ and $\vv = \ttb^4\tta^2\ttb\tta\ttb^4\tta^2\ttb$}\\

\hline
\hline
\end{tabular}
\caption{\sf\smaller \VR(4,0) Examples of two-generator monoids with a triangular presentation: in all cases that are not \textit{a priori} discarded by Lemma~\ref{L:Excluded} or its symmetric counterpart ``Lemma~$\widetilde{\ref{L:Excluded}}$\,'' , one can either find a (quasi)-central element or identify a non-terminating reversing, hence decide whether the associated monoid is of right- or left-$O$-type.}
\label{T:Exp2}
\end{table}

See Table~\ref{T:Exp2} for some typical examples. Results are entirely similar in the case of three generators or more. Again, most of the cases that are not discarded by Lemma~\ref{L:Excluded} or its counterpart turn out to be eligible for Lemma~\ref{L:QuasiExp}, and the exceptions can be successfully addressed directly. It is probably useless to give more details here.


\section{Some families of monoids of $O$-type}
\label{S:Fam}

We shall now describe five infinite families for which we can exhibit central, quasi-central, or dominating elements and which therefore are of right-$O$-type or of $O$-type. We begin with the simplest case, namely when some power of the top generator is central or quasi-central.

\begin{prop}
\label{P:Fam1}
For $\pp, \qq, \rr \ge 1$, let $\MM$ be the monoid defined by  
\begin{equation}
\label{E:Fam1}
\Pres{\tta, \ttb}{\tta = \ttb (\tta^\pp \ttb^\rr)^\qq}.
\end{equation}

\ITEM1 The element $\tta^{\pp+1}$ is right-quasi-central in~$\MM$ and $\MM$ is of right-$O$-type.

\ITEM2 For $\rr = 1$, the element $\tta^{\pp+1}$ is central in~$\MM$ and $\MM$ is of $O$-type. 
\end{prop}

\begin{proof}
\ITEM1 Applying~\eqref{E:Fam1}, we first find
$$\tta = \ttb (\tta^\pp\ttb^\rr)^\qq 
= \ttb \cdot \tta \cdot (\tta^{\pp-1} \ttb^\rr)(\tta^\pp \ttb^\rr)^{\qq-1},$$
whence, repeating the operation $\rr$~times and moving the brackets,
$$\tta = \ttb^\rr \cdot \tta \cdot ((\tta^{\pp-1} \ttb^\rr)(\tta^\pp \ttb^\rr)^{\qq-1})^\rr 
= \ttb^\rr (\tta^\pp\ttb^\rr)^\qq \cdot ((\tta^{\pp-1} \ttb^\rr)(\tta^\pp \ttb^\rr)^{\qq-1})^{\rr-1}.$$
Let $\Delta = \tta^{\pp+1}$. Substituting the above value of~$\tta$ at the underlined position, moving the brackets, and applying the relation once in the contracting direction, we find
\begin{align*}
\ttb \cdot \Delta
= \ttb \tta^\pp \underline{\tta}
&= \ttb \tta^\pp \cdot \ttb^\rr (\tta^\pp\ttb^\rr)^\qq \cdot ((\tta^{\pp-1} \ttb^\rr)(\tta^\pp \ttb^\rr)^{\qq-1})^{\rr-1}\\
&= \ttb (\tta^\pp \ttb^\rr)^\qq \cdot \tta^\pp \ttb^\rr ((\tta^{\pp-1} \ttb^\rr)(\tta^\pp \ttb^\rr)^{\qq-1})^{\rr-1}\\
&= \tta\cdot \tta^\pp \cdot \ttb^\rr ((\tta^{\pp-1} \ttb^\rr)(\tta^\pp \ttb^\rr)^{\qq-1})^{\rr-1}
= \Delta \cdot \ttb^\rr ((\tta^{\pp-1} \ttb^\rr)(\tta^\pp \ttb^\rr)^{\qq-1})^{\rr-1}.
\end{align*}
Now, we have $\ttb \dive \tta \dive \tta^{\pp+1} = \Delta$ in~$\MM$. Next, as $\Delta$ is a power of~$\tta$, it commutes with~$\tta$. So, by Lemma~\ref{L:Quasi}, $\Delta$ is right-quasi-central in~$\MM$, with associated endomorphism defined by
\begin{equation}
\label{E:Fam1f}
\phi(\tta) = \tta, \qquad \phi(\ttb) = \ttb^\rr ((\tta^{\pp-1} \ttb^\rr)(\tta^\pp \ttb^\rr)^{\qq-1})^{\rr-1}.
\end{equation}
By Proposition~\ref{P:QuasiO}, $\MM$ is of right-$O$-type. 

\ITEM2 Assume now $\rr = 1$. By~\ITEM1, $\MM$ is of right-$O$-type. Now \eqref{E:Fam1} reduces here to $\tta =\nobreak \ttb (\tta^\pp\ttb)^\qq$, which is symmetric. Hence, by~\ITEM1, the opposite monoid~$\MMt$ is of right-$O$-type, so $\MM$ is of left-$O$-type, and, therefore, of $O$-type. Finally, by~\eqref{E:Fam1f}, the endomorphism~$\phi$ is the identity here, so that $\Delta$ is central. 
\end{proof}

By Proposition~\ref{P:Recipe}, every monoid of $O$-type gives rise to an ordered group. The groups occurring in connection with the monoids of Proposition~\ref{P:Fam1} are the torus knot groups. 

\begin{coro}
\label{C:Fam1}
For $\pp, \qq \ge 1$, let $\GG$ be the torus knot group $\Gr{\ttx, \tty}{\ttx^{\pp+1} =\nobreak \tty^{\qq+1}}$. Then $\GG$ is orderable, the subsemigroup of~$\GG$ generated by~$\ttx$ and~$\ttx^{-\pp}\tty$ is the positive cone of a left-invariant ordering on~$\GG$ which is isolated in~$\LO\GG$. 
\end{coro}

\begin{proof}
In view of Propositions~\ref{P:Recipe} and~\ref{P:Fam1}, it is sufficient to prove that \eqref{E:Fam1} is a presentation of~$\GG$ in terms of the mentioned elements. Now, put $\aa = \ttx$ and $\bb = \ttx^{-\pp}\tty$ in~$\GG$. We find $\aa^{\pp+1} =\nobreak \bb(\aa^\pp\bb)^\qq\aa^\pp$, whence $\aa =\nobreak \bb(\aa^\pp\bb)^\qq$, so, if $\GG'$ is the group defined by the presentation~\eqref{E:Fam1}, there exists a well defined homomorphism~$\Phi$ of~$\GG'$ to~$\GG$ mapping~$\tta$ to~$\aa$ and~$\ttb$ to~$\bb$. Moreover $\Phi$ is surjective as $\aa$ and~$\bb$ generate~$\GG$. Conversely, put $\xx = \tta$ and $\yy = \tta^\pp \ttb$ in~$\GG'$. We obtain now $\xx^{\pp+1} = \tta \tta^\pp = \ttb (\tta^\pp \ttb)^\qq \ttb = (\ttb \tta^\pp)^{\qq+1} = \yy^{\qq+1}$, and there exists a surjective homomorphism~$\Phi'$ of~$\GG$ to~$\GG'$ mapping~$\ttx$ to~$\xx$ and~$\tty$ to~$\yy$. By construction, $\Phi' \comp \Phi$ is the identity, so $\Phi$ is an isomorphism, and \eqref{E:Fam1} is a presentation of~$\GG$.
\end{proof}

\begin{rema}
\label{R:Embed}
Once noted that the group $\Gr{\tta, \ttb}{\tta = \ttb(\tta^\pp\ttb)^\qq}$ is isomorphic to the group $\Gr{\ttx, \tty}{\ttx^{\pp+1} =\nobreak \tty^{\qq+1}}$, it is obvious that $\ttx^{\pp+1}$, \ie, $\tta^{\pp+1}$, is central in the group. However, this is not sufficient to deduce that $\tta^{\pp+1}$ is central in the monoid~$\Mon{\tta, \ttb}{\tta = \ttb(\tta^\pp\ttb)^\qq}$ as the latter is not {\it a priori} known to embed in the group. So, in order to apply the scheme of Theorem~\ref{T:Main}, it is crucial to make all verifications inside the monoid, \ie, without using inverses except possibly those provided by cancellativity. 
\end{rema}

For $\pp = \qq = 1$, the group~$\GG$ of Corollary~\ref{C:Fam1} is the Klein bottle group~$\Gr{\tta, \ttb}{\tta = \ttb\tta\ttb}$. For $\pp = 2$ and $\qq = 1$, the group~$\GG$, \ie, $\Gr{\tta, \ttb}{\tta = \ttb \tta^2 \ttb}$, is Artin's braid group~$B_3$. In terms of the standard Artin generators~$\sig\ii$, the elements~$\tta$ and~$\ttb$ can be realized as $\sig1\sig2$ and~$\siginv2$, and the associated ordering is the isolated ordering described by Dubrovina--Dubrovin in~\cite{DuD} (see also~\cite{Pic}). The braid group~$B_3$ is also obtained for $\pp = 1$ and $\qq = 2$, \ie, for $\Gr{\tta, \ttb} {\tta = \ttb \tta \ttb \tta \ttb}$, with $\tta$ and~$\ttb$ now realizable as $\sig1\sig2\sig1$ and~$\siginv2$. Note that, when realized as above, the associated submonoids of~$B_3$ coincide as, using $\tta$ in the case~$\pp = 2, \qq = 1$ and $\tta'$ in the case~$\pp = 1, \qq = 2$, we find $\tta = \tta' \ttb$ and $\tta' = \ttb \tta^2$. Therefore the associated (isolated) orderings of~$B_3$ coincide.

For $\rr \ge 2$, the monoid~$\MM$ of Proposition~\ref{P:Fam1} embeds in a group of right-fractions~$\GG$. However, the left counterpart of Lemma~\ref{L:Excluded} (``Lemma~$\widetilde{\ref{L:Excluded}}$'') implies that $\tta$ and~$\tta\ttb$ have no common left-multiple in~$\MM$. Hence $\MM$ is not of left-$O$-type, and the group~$\GG$ is not a group of left-fractions for~$\MM$: the right-fraction $\tta \ttb \tta\inv$ is an element of~$\GG$ that cannot be expressed as a left-fraction. As a consequence, the semigroup~$\MM{\setminus}\{1\}$ defines a \emph{partial} left-invariant ordering on~$\GG$ only: for instance, the elements~$\ttb\inv \tta\inv$ and $\tta\inv$ are not comparable as their quotient $\tta\ttb \tta\inv$ belongs neither to~$\MM$ nor to~$\MM\inv$. Note that, for $\pp = \qq = 1$, the group~$\GG$, \ie, $\Gr{\tta, \ttb}{\tta = \ttb\tta\ttb^{\rr+1}}$, is the Baumslag--Solitar group $\BS(\rr+1, -1)$, whereas the opposite group $\Gr{\tta, \ttb}{\tta = \ttb^{\rr+1}\tta\ttb}$ is $\BS(-1, \qq+1)$. Besides these examples, the case $\pp = \rr = 2$, $\qq = 1$, \ie, $\Gr{\tta, \ttb}{\tta = \ttb\tta^2\ttb^2}$, is the first non-classical example in the family. In this case, $\tta^3$ is a right-quasi-central element that is not central, and the associated endomorphism is given by $\phi(\tta) = \tta$ and $\phi(\ttb) = \ttb^2 \tta \ttb^2$.

We now describe a family involving an arbitrarily large family of generators. 

\begin{prop}
\label{P:Fam4}
Assume $\ell \ge 2$ and $\mm_2, \nn_2, \mm_3, \nn_3 \wdots \mm_\ell, \nn_\ell \ge 1$. Let $\MM$ the be monoid defined by
\begin{equation}
\label{E:Fam4}
\Pres{\tta_1 \wdots \tta_{\ell}}{\tta_1 = \tta_2 \ww_2^{\nn_2} \wdots \tta_{\ell-1} = \tta_\ell \ww_\ell^{\nn_\ell}},
\end{equation} 
with $\ww_1 = \tta_1$ and $\ww_\ii$ inductively defined by $\ww_\ii = \ww_{\ii-1}^{\mm_\ii} \pdots \ww_2^{\mm_3} \ww_1^{\mm_2} \tta_\ii$. Then $\MM$ is of $O$-type.
\end{prop}

\begin{proof}
First, the presentation~\eqref{E:Fam4} is triangular (and symmetric). Now put $\gg_\ii = \clp{\ww_\ii}$ in~$\MM$. For $\ii \ge 2$, we obtain 
\begin{gather*}
\gg_{\ii-1}^{\mm_\ii+1} 
= \gg_{\ii-1}^{\mm_\ii} \cdot \gg_{\ii-1} 
= \gg_{\ii-1}^{\mm_\ii} \gg_{\ii-2}^{\mm_{\ii-1}} \pdots \gg_1^{\mm_2} \tta_{\ii-1},\\
\gg_{\ii}^{\nn_\ii+1} 
= \gg_{\ii} \cdot \gg_{\ii}^{\nn_\ii}
= \gg_{\ii-1}^{\mm_\ii} \pdots \gg_1^{\mm_2} \tta_{\ii}  \cdot \gg_{\ii}^{\nn_\ii}.
\end{gather*}
The relations of~\eqref{E:Fam4} are valid in~$\MM$, so, in particular, we have $\tta_{\ii-1} = \tta_{\ii}  \cdot \gg_{\ii}^{\nn_\ii}$. It follows that $\gg_{\ii-1}^{\mm_\ii+1} = \gg_{\ii}^{\nn_\ii+1}$ holds for $\ii = 2 \wdots \ell$. 

Put $\Delta = \tta_1^\mm$ where $\mm$ is such that $\mm (\nn_2+1) \pdots (\nn_\ii+1)$ is a multiple of~$(\mm_2+\nobreak1) \pdots (\mm_\ii+1)$ for each~$\ii$---which, for instance, is the case if $\mm$ is a common multiple of $\mm_2+1 \wdots \mm_\ell+1$---say $\mm (\nn_2+1) \pdots (\nn_\ii+1) = \ee_\ii (\mm_2+\nobreak1) \pdots (\mm_\ii+1)$. Then, for every~$\ii$, we have $\Delta = \gg_\ii^{\ee_\ii}$, and, therefore, $\Delta$ commutes with every~$\gg_\ii$. Then an induction on~$\ii$
 shows that $\Delta$ commutes with every~$\tta_\ii$: for $\ii = 1$, the result is obvious as $\Delta$ is a power of~$\tta_1$, for $\ii > 1$, we have $\tta_{\ii-1} = \tta_\ii \gg_\ii$, whence, using the induction hypothesis,
$$\Delta \tta_\ii \gg_\ii 
= \Delta \tta_{\ii-1}
= \tta_{\ii-1} \Delta
= \tta_{\ii-1} \gg_\ii \Delta
= \tta_\ii \Delta \gg_\ii,$$
and $\Delta \tta_\ii  = \tta_\ii \Delta$ by right-cancelling~$\gg_\ii$, which is legitimate as $\MM$, which admits a left-triangular presentation, must be right-cancellative. As $\tta_1 \wdots \tta_\ell$ generate~$\MM$, the element~$\Delta$ is central in~$\MM$, and, by Proposition~\ref{P:QuasiO}, $\MM$ is of $O$-type.
\end{proof}

The corresponding groups turn out to be amalgamated products of torus knot groups.

\begin{coro}
\label{C:Fam4}
For $\ell \ge 2$ and $\mm_2, \nn_2, \mm_3, \nn_3 \wdots \mm_\ell, \nn_\ell \ge 1$, let $\GG$ be a group 
\begin{equation}
\label{E:Fam4group}
\Pres{\ttx_1, \ttx_2 \wdots \ttx_{\ell}}{\ttx_1^{\mm_2+1} = \ttx_2^{\nn_2+1}, \ttx_2^{\mm_3+1} = \ttx_3^{\nn_3+1} \wdots \ttx_{\ell-1}^{\mm_\ell+1} = \ttx_{\ell}^{\nn_\ell+1}}.
\end{equation}
Then $\GG$ is orderable, the subsemigroup of~$\GG$ generated by~$\ttx_1$, $\ttx_1^{-\mm_2}\ttx_2 \wdots \ttx_1^{-\mm_2} \pdots \ttx_{\ii-1}^{-\mm_\ii} \ttx_\ii$ is the positive cone of a left-invariant ordering on~$\GG$ which is isolated in~$\LO\GG$.
\end{coro}

\begin{proof}
We just have to show that $\GG$ admits the presentation~\eqref{E:Fam4} with respect to the specified generators. Now, in~$\GG$, put $\aa_1 = \ttx_1$ and, inductively, $\aa_\ii = \ttx_1^{-\mm_2} \pdots \ttx_{\ii-1}^{-\mm_\ii} \ttx_\ii$. An immediate induction gives $\ttx_\ii = \cl{\ww_\ii}$ for each~$\ii$, where $\cl{\ww_\ii}$ means the evaluation of the word~$\ww_\ii$ when the letter~$\tta_\ii$ is given the value~$\aa_\ii$. Then (as above) we obtain
$$\ttx_{\ii_1}^{\mm_\ii} \pdots \ttx_1^{\mm_2} \aa_{\ii-1} = \ttx_{\ii-1}^{\mm_\ii+1} 
= \ttx_\ii^{\nn_\ii+1} 
= \ttx_{\ii_1}^{\mm_\ii} \pdots \ttx_1^{\mm_2} \aa_\ii \ttx_\ii^{\nn_\ii},$$
whence $\aa_{\ii-1} = \aa_\ii \cl{\ww_\ii}^{\nn_\ii}$. This shows that the relations of~\eqref{E:Fam4} are satisfied by~$\aa_1 \wdots \aa_\ell$ in~$\GG$, yielding an homomorphism~$\Phi$ of the group~$\GG'$ presented by~\eqref{E:Fam4} to~$\GG$ that maps~$\tta_\ii$ to~$\aa_\ii$.

Conversely, in~$\GG'$, define $\xx_1$ to be~$\tta_1$ and $\xx_\ii$ to be~$\cl{\ww_\ii}$ for $\ii \ge 2$. Then the same computation as above in the monoid~$\MM$ shows that $\xx_{\ii-1}^{\mm_\ii+1} = \xx_\ii^{\nn_\ii+1}$ holds for every~$\ii$ in~$\GG'$, leading to an homomorphism of~$\GG$ to~$\GG'$ that is the inverse of~$\Phi$. So $\GG$ admits \eqref{E:Fam4} as a presentation.
\end{proof} 

\begin{exam}
Assume $\mm_2 = \nn_2 = \pdots = \mm_\ell = \nn_\ell = \pp$. Then $\GG$ admits the presentation $\Pres{\ttx_1 \wdots \ttx_\ell}{\ttx_1^{\pp+1} = \ttx_2^{\pp+1} = \pdots = \ttx_\ell^{\pp+1}}$, and the result applies with $\Delta = \ttx_1^{\pp+1}$. The positive cone of the associated isolated ordering is defined by the presentation~\eqref{E:Fam4}, whose relations, in the current case, take the form (as usual, we write $\tta, \ttb, ...$ for $\tta_1, \tta_2, ...$)
$$\tta = \ttb (\tta^\pp\ttb)^\pp,
\quad \ttb = \ttc ((\tta^\pp\ttb)^\pp \tta^\pp\ttc)^\pp,
\quad \ttc = \ttd (((\tta^\pp\ttb)^\pp \tta^\pp\ttc)^\pp (\tta^\pp\ttb)^\pp \tta^\pp\ttd)^\pp, \ {\it etc.}
$$
See Table~\ref{T:Recap} (row 2) for other particular cases.
\end{exam}

We now turn to a family of a different type, where the verification of the $O$-type relies not on the existence of a quasi-central element as in the previous examples, but on the existence of a dominating element---hence an infinitary condition.

\begin{prop}
\label{P:Fam5}
For $\pp, \qq, \rr, \ss \ge 0$, let $\MM$ be the monoid defined by   
\begin{equation}
\label{E:Fam5}
\Pres{\tta, \ttb, \ttc}{\tta = \ttb(\tta^\pp \ttb)^\qq, \ttb = \ttc (\tta^{\rr}\ttc)^\ss},
\end{equation}
For $\rr \ge \pp$ or $\qq = 0$, the element~$\tta$ dominates $\tta, \ttb, \ttc$ in~$\MM$, and $\MM$ is of $O$-type; for $\rr < \pp$ with $\qq \ge 1$, it is not.
\end{prop}

\begin{proof}
As a preliminary remark, we note that the case $\qq = 0$ is trivial, as \eqref{E:Fam5} then reduces to $\tta = \ttb$ and $\MM$ is isomorphic to $\Mon{\tta, \ttb}{\tta = \ttb(\tta^{\rr}\ttb)^\ss}$, in which, by Proposition~\ref{P:Fam1}, $\tta^{\rr + 1}$ is central, hence it right-dominates $\tta, \ttb, \ttc$ by Lemma~\ref{L:QuasiDomin}, and $\MM$ is of $O$-type. So, from now on, we assume $\qq \ge 1$. 

Assume first $\pp < \rr$. Then the relation $\tta = \ttc(\tta^\rr\ttc)^\ss (\tta^\pp\ttb)^\qq$ belongs to the completion~$\RRh$ of the presentation, and it can be written as $\tta = (\ttc\tta^\rr)^{\ss + 1} \tta ...$. Hence, by Lemma~\ref{L:Excluded} with $\uu = \ttc\tta^\rr$, the monoid~$\MM$ cannot be of $O$-type.

From now on, we assume $\rr \ge \pp$. As a preliminary remark, we compute
\begin{equation}
\label{E:Fam5b}
\tta^{\pp+1} \cdot \ttb
= \tta\tta^\pp \ttb
= (\ttb (\tta^\pp \ttb)^\qq) \tta^\pp \ttb
= \ttb (\tta^\pp \ttb)^{\qq+1}
= \ttb \tta^\pp (\ttb (\tta^\pp \ttb)^\qq)
= \ttb \tta^\pp \tta = \ttb \cdot \tta^{\pp+1},
\end{equation}
which shows that $\tta^{\pp+1}$ commutes with~$\tta$ and~$\ttb$. Moreover, as $\ttb\tta^\pp \dive \tta$ holds, applying Lemma~\ref{L:QuasiDomin} in the submonoid of~$\MM$ generated by~$\tta$ and~$\ttb$ shows that $\tta$ right-dominates~$\ttb$. 

Next, let us write $\rr = \mm + {\pp'}$ with $\pp + 1 \,\vert\, \mm$ and $0\le {\pp'} \le \pp$. Then $\tta^\mm$ is a power of~$\tta^{\pp+1}$, hence it commutes with~$\tta$ and~$\ttb$. From here, we separate two cases. 

Assume first that $\pp+1$ does not divide $\rr + 1$, \ie, ${\pp'} < \pp$ holds. Put $\ttc' = \ttc\tta^\pp$, and let $\MM'$ be the submonoid of~$\MM$ generated by~$\tta$, $\ttb$, and $\ttc'$. Let $\Delta = \tta^\mm$. We shall prove that $\Delta$ is right-quasi-central in~$\MM'$. First, we can write $\tta \cdot \tta^{\mm-1} = \Delta$ and $\Delta \cdot \tta = \tta\Delta$ both in~$\MM$ and in~$\MM'$, so $\tta \dive \Delta \dive \tta\Delta$ holds in~$\MM'$ as it does in~$\MM$. Next, we have $\ttb \cdot (\tta^\pp\ttb)^\qq\tta^{\mm-1} = \Delta$ and $\Delta \cdot \ttb = \ttb\Delta$ in~$\MM$ and~$\MM'$, so $\ttb \dive \Delta \dive \ttb\Delta$ holds in~$\MM'$. It remains to treat~$\ttc'$. One the one hand, starting from $\tta = \ttc(\tta^{\rr}\ttc)^\ss (\tta^\pp\ttb)^\qq$, which we re-arrange into $\tta = \ttc' \cdot (\tta^{\rr-\pp}\ttc')^\ss (\ttb\tta^\pp)^{\qq-1}\ttb$ by moving brackets, we see that $\ttc' \dive \tta$ holds in~$\MM'$, whence \textit{a fortiori} $\ttc' \dive \Delta$. On the other hand, in~$\MM$, we find
\begin{align}
\label{E:Fam5c}
\ttc' \cdot \Delta
&=\ttc \tta^\rr \cdot \tta \cdot \tta^{\pp - {\pp'} - 1}
&&\mbox{by $\rr = \mm + {\pp'}$ with ${\pp'} < \pp$}\\\notag
&=\ttc\tta^\rr \cdot \ttb (\tta^\pp\ttb)^\qq \cdot \tta^{\pp - {\pp'} - 1}
&&\mbox{by \eqref{E:Fam5} for $\tta$}\\\notag
&=\ttc\tta^\rr \cdot \ttc (\tta^{\rr}\ttc)^\ss  \cdot (\tta^\pp\ttb)^\qq \cdot \tta^{\pp - {\pp'} - 1}
&&\mbox{by \eqref{E:Fam5} for $\ttb$}\\\notag
&=\ttc (\tta^\rr\ttc)^\ss \cdot \tta^{\rr}\ttc  \cdot \tta^\pp \cdot (\ttb \tta^\pp)^{\qq-1}\ttb\tta^{\pp - {\pp'} - 1}
&&\mbox{by moving brackets}\\\notag
&=\ttb \cdot \tta^{\rr}\cdot \ttc\tta^\pp \cdot (\ttb \tta^\pp)^{\qq-1}\ttb\tta^{\pp - {\pp'} - 1}
&&\mbox{by \eqref{E:Fam5} for $\ttb$}\\\notag
&=\ttb \cdot \Delta \cdot \tta^{\pp'} \cdot \ttc' \cdot (\ttb \tta^\pp)^{\qq-1}\ttb\tta^{\pp - {\pp'} - 1}
&&\mbox{by $\rr = \mm + {\pp'}$}\\\notag
&=\Delta \cdot \ttb \tta^{{\pp'}}\cdot \ttc' \cdot (\ttb \tta^\pp)^{\qq-1}\ttb\tta^{\pp - {\pp'} - 1}
&&\mbox{by \eqref{E:Fam5b}},
\end{align}
which shows that $\Delta \dive \ttc' \Delta$ holds in~$\MM'$. Being a submonoid of a cancellative monoid, the monoid~$\MM'$ is cancellative, so, by Lemma~\ref{L:QuasiDomin}, $\Delta$, \ie, $\tta^\mm$, is right-quasi-central in~$\MM'$. Moreover, in~$\MM$, we have $\tta = \ttc\tta^\pp \cdot (\tta^{\rr-\pp} \ttc\tta^\pp)^\ss (\ttb\tta^\pp)^{\qq-1}\ttb$, which we can re-arrange as $\tta =\nobreak \ttc\tta^\pp (\tta^{\rr-\pp} \ttc\tta^\pp)^\ss (\ttb\tta^\pp)^{\qq-1}\ttb$, \ie, $\tta = \ttc' (\tta^{\rr-\pp} \ttc')^\ss (\ttb\tta^\pp)^{\qq-1}\ttb$, a relation that makes sense in~$\MM'$. So, in~$\MM'$, we have $\ttc' \tta^{\rr-\pp} \dive \tta$. Applying Lemma~\ref{L:QuasiDomin} in~$\MM'$ is not sufficient, but, repeating its proof, we obtain, writing~$\phi'$ for the endomorphism of~$\MM'$ associated with~$\Delta$,
$$\ttc'\tta^{\rr-\pp} \Delta{}^\kk 
= \Delta{}^\kk \phi'{}^\kk(\ttc' \tta^{\rr-\pp})
\dive \Delta{}^\kk \phi'{}^\kk(\tta)
= \Delta{}^\kk \tta,$$
which implies in~$\MM$ the relation $\ttc \tta^{\kk\mm + \rr} \dive \tta^{\kk\mm +1}$ for every~$\kk$. As $\mm-1 \le \rr$ is true by definition, Lemma~\ref{L:DominBis} implies that $\tta$ right-dominates~$\ttc$ in~$\MM$. We saw above that $\tta$ right-dominates~$\ttb$, hence $\tta$ right-dominates all of $\tta, \ttb, \ttc$. By Proposition~\ref{P:DominO}, $\MM$ is of right-$O$-type, hence of $O$-type by symmetry. 

Assume now that $\pp+1$ divides $\rr+1$, \ie, We have ${\pp'} = \pp$ holds. Then, starting from $\tta = \ttb (\tta^\pp)^\qq$, we obtain $\tta^{\qq\mm + 1} 
= \ttb (\tta^\pp\ttb)^\qq (\tta^\mm)^\qq
= \ttb (\tta^\pp\tta^\mm\ttb)^\qq
= \ttb (\tta^\rr\ttb)^\qq$, 
whence, substituting $\ttb$ with $\ttc(\tta^\rr\ttc)^\ss$, 
$$\tta^{\qq\mm + 1} = \ttc(\tta^\rr\ttc)^\qq (\tta^\rr\ttc(\tta^\rr\ttc)^\ss 
= \ttc (\tta^\rr\ttc)^{(\ss+1)\qq + \ss}
= (\ttc \tta^\rr)^{(\ss+1)\qq + \ss}\ttc,$$
and, now putting $\Delta = \tta^{\qq\mm + \rr + 1}$ (not the same value as above),
$$\ttc \cdot \Delta 
= \ttc\tta^\rr \cdot \tta^{\qq\mm + 1} 
= \ttc\tta^\rr \cdot  (\ttc\tta^\rr)^{(\ss+1)\qq + \ss}\ttc
= (\ttc\tta^\rr)^{(\ss+1)\qq + \ss}\ttc \cdot  \tta^\rr\ttc
= \tta^{\qq\mm + 1}  \cdot \tta^\rr \ttc
= \Delta \cdot \ttc,$$
so $\Delta$ commutes with~$\ttc$. As it is a power of~$\tta^{\pp+1}$, it commutes with~$\ttb$, and with~$\tta$, so $\Delta$ is central in~$\MM$, and $\MM$ is of right-$O$-type, hence of $O$-type by symmetry. 

To complete the study of~$\MM$, we shall check that $\tta$ is dominating in all cases. We already proved that $\tta$ right-dominates~$\tta$ and~$\ttb$, and the point is to show that $\tta$ right-dominates~$\ttc$ in the case $\pp + 1 \,\vert\, \rr+1$. Above we proved that $\tta^{\qq\mm + \rr+1}$ is central in this case. Now, we only have $\ttc\tta^\rr \dive \tta$ so, unless in the degenerated case $\qq = 0$, we are not in position for applying Lemma~\ref{L:QuasiDomin} directly. Instead we shall prove $\ttc\tta^{\rr + \kk\mm} \dive \tta^{1 + \kk\mm}$ for $\kk = 0 \wdots \qq$. As $\mm \le \rr$ holds (we have $\rr = \mm + \pp$ by definition), we deduce $\ttc\tta^\nn \dive \tta^{\nn+1}$ for $0 \le \nn \le \qq\mm + \rr$ using the argument of Lemma~\ref{L:QuasiDomin}. Now, as $\tta^{\qq\mm + \rr + 1}$ is central, the same relations then repeat periodically, and $\ttc\tta^\nn \dive \tta^{\nn+1}$ holds for every~$\nn$. What we shall do is to prove
\begin{equation}
\label{E:Fam5Domin}
\tta^{1 + \kk\mm} = \ttc\tta^{\rr + \kk\mm} \cdot (\ttb \tta^\pp)^\kk \ttc(\tta^\rr\ttc)^{\ss-1} (\tta^\pp\ttb)^{\qq-\kk}
\end{equation}
using induction on~$\kk$. For $\kk = 0$, we have $\tta = \ttb (\tta^\pp\ttb)^\qq = \ttc\tta^\rr \tta(\tta^\rr\ttc)^{\ss-1} (\tta^\pp\ttb)^\qq$, which is~\eqref{E:Fam5Domin}. Assume $\qq \ge \kk > 0$. We find
\begin{align*}
\tta^{1 + \kk\mm} 
&= \tta^{1 + (\kk-1)\mm} \cdot \tta^\mm\\
&= \ttc\tta^{\rr + (\kk-1)\mm} \cdot (\ttb \tta^\pp)^{\kk-1} \ttc(\tta^\rr\ttc)^{\ss-1} (\tta^\pp\ttb)^{\qq-\kk+1} \cdot \tta^\mm
&&\mbox{by induction hyp.}\\
&= \ttc\tta^{\rr + (\kk-1)\mm} \cdot (\ttb \tta^\pp)^{\kk-1} \ttc(\tta^\rr\ttc)^{\ss-1} \tta^\mm(\tta^\pp\ttb)^{\qq-\kk+1}
&&\mbox{by $\tta^\mm \, \ttb = \ttb \, \tta^\mm$}\\
&= \ttc\tta^{\rr + (\kk-1)\mm} \cdot (\ttb \tta^\pp)^{\kk-1} \ttc(\tta^\rr\ttc)^{\ss-1} \tta^\rr \ttb(\tta^\pp\ttb)^{\qq-\kk}
&&\mbox{by $\rr = \mm + \pp$}\\
&= \ttc\tta^{\rr + (\kk-1)\mm} \cdot (\ttb \tta^\pp)^{\kk-1} \ttc(\tta^\rr\ttc)^{\ss-1} \tta^\rr \ttb(\tta^\pp\ttb)^{\qq-\kk}
&&\mbox{by $\rr = \mm + \pp$}\\
&= \ttc\tta^{\rr + (\kk-1)\mm} \cdot (\ttb \tta^\pp)^{\kk-1} \ttc(\tta^\rr\ttc)^{\ss-1} \tta^\rr\ttc (\tta^\rr\ttc)^{\ss-1}(\tta^\pp\ttb)^{\qq-\kk}
&&\mbox{by \eqref{E:Fam5} for~$\ttb$}\\
&= \ttc\tta^{\rr + (\kk-1)\mm} \cdot (\ttb \tta^\pp)^{\kk-1} \ttc(\tta^\rr\ttc)^\ss \tta^\rr \ttc (\tta^\rr\ttc)^{\ss-1}(\tta^\pp\ttb)^{\qq-\kk}
&&\mbox{by moving brackets}\\
&= \ttc\tta^{\rr + (\kk-1)\mm} \cdot (\ttb \tta^\pp)^{\kk-1} \ttb \tta^\rr \ttc (\tta^\rr\ttc)^{\ss-1}(\tta^\pp\ttb)^{\qq-\kk}
&&\mbox{by \eqref{E:Fam5} for~$\ttb$}\\
&= \ttc\tta^{\rr + (\kk-1)\mm} \cdot (\ttb \tta^\pp)^{\kk-1} \ttb \tta^\mm\tta^\pp \ttc (\tta^\rr\ttc)^{\ss-1}(\tta^\pp\ttb)^{\qq-\kk}
&&\mbox{by $\rr = \mm + \pp$}\\
&= \ttc\tta^{\rr + (\kk-1)\mm} \cdot \tta^\mm(\ttb \tta^\pp)^{\kk-1} \ttb \tta^\pp \ttc (\tta^\rr\ttc)^{\ss-1}(\tta^\pp\ttb)^{\qq-\kk}
&&\mbox{by by $\tta^\mm \, \ttb = \ttb \, \tta^\mm$}\\
&= \ttc\tta^{\rr + \kk\mm} \cdot (\ttb \tta^\pp)^\kk \ttc (\tta^\rr\ttc)^{\ss-1}(\tta^\pp\ttb)^{\qq-\kk}
&&\mbox{which is \eqref{E:Fam5Domin}.}\\
\end{align*}
So the argument is complete and, even in the case $\pp + 1 \,\vert\, \rr+1$, the element~$\tta$ dominates each of~$\tta$, $\ttb$, and $\ttc$. 
\end{proof}

Note that the monoid of Proposition~\ref{P:Fam5} is generated by~$\tta$ and~$\ttc$ alone and admits the corresponding (less readable) presentation $\Pres{\tta, \ttc}{\tta = \ttc(\tta^\rr\ttc)^\ss (\tta^\pp \ttc (\tta^\rr\ttc)^\ss)^\qq}$. 
The $4$-parameter family of Proposition~\ref{P:Fam5} contains in particular all presentations $\Pres{\tta, \ttb}{\tta = \ttb\tta^\rr\ttb\tta^\pp\ttb\tta^\rr\ttb}$, and all presentations $\Pres{\tta, \ttb}{\tta =\nobreak (\ttb (\tta^\rr \ttb)^\ss)^\qq}$. In terms of orderable groups, we deduce

\begin{coro}
\label{C:Fam5}
For $\pp, \qq, \ss \ge 0$ and $\rr \ge \pp$, let $\GG$ be the group $\Gr{\ttx, \tty}{\ttx^{\pp+1} = (\tty (\ttx^{\rr-\pp} \tty)^\ss)^{\qq+1}}$. Then $\GG$ is orderable, the subsemigroup of~$\GG$ generated by~$\ttx$ and $\tty\ttx^{-\pp}$ is the positive cone of a left-invariant ordering on~$\GG$ which is isolated in~$\LO\GG$.
\end{coro}

\begin{proof}
Once more, it suffices to show that $\GG$ admits the presentation~\eqref{E:Fam5} in terms of the mentioned elements. Put $\aa = \ttx$, $\cc = \tty\ttx^{-\pp}$, and $\bb = \cc(\aa^{\rr}\cc)^\ss$ in~$\GG$. The second relation of~\eqref{E:Fam5} is automatically satisfied; on the other hand, we find
$$\aa^{\pp+1} 
= (\tty (\ttx^{\rr-\pp} \tty)^\ss)^{\qq+1}
= (\cc\aa^\pp (\aa^{\rr-\pp} \cc \aa^\pp)^\ss)^{\qq+1} 
= (\cc (\aa^{\rr}\cc)^\ss \aa^\pp)^{\qq+1}
= (\bb \aa^\pp)^{\qq+1}
= \bb (\aa^\pp \bb)^\qq \aa^\pp,$$
whence the first relation of~\eqref{E:Fam5} by right-cancelling~$\aa^\pp$, yielding an homomorphism of the group~$\GG'$ defined by~\eqref{E:Fam5} to~$\GG$. Conversely, put $\xx = \tta$ and $\yy = \ttc\tta^\pp$ in~$\GG'$. Then $\xx$ and~$\yy$ generate~$\GG'$ and satisfy $\xx^{\pp+1} = (\yy (\xx^{\rr-\pp} \yy)^\ss)^{\qq+1}$, whence an homomorphism of~$\GG$ to~$\GG'$, and, finally, an isomorphism. So $\GG$ admits the expected presentation.
\end{proof}

We finish with a family of still another type, namely one where the third generator is used to split an initially not triangular relation into two triangular relations, by observing that a two-generator relation $\tta \uu = \ttb \vv \ttb$ defines the same group as the two triangular relations $\tta = \ttb \vv \ttc$, $\ttb = \ttc \uu$. The example we choose is interesting as it corresponds to groups that are not torus knot groups.

\begin{prop}
\label{P:Fam6}
For $\pp, \qq, \rr \ge 0$, let $\MM$ be the monoid defined by   
\begin{equation}
\label{E:Fam6}
\Pres{\tta, \ttb, \ttc}{\tta = \ttb\tta^{\pp+2}(\ttb\tta^\pp \ttb \tta^{\pp+2})^\qq \ttc, \ttb = \ttc (\ttb \tta^{\pp+2})^\rr \ttb\tta},
\end{equation}
Then, for $\rr =0$ and for $\rr = 1$, the monoid~$\MM$ is of $O$-type.
\end{prop}

\begin{proof}[Proof (sketch)]
The presentation~\eqref{E:Fam6} is triangular. Put $\Delta = (\tta^{\pp+2}\ttb)^{2\qq + \rr +3}$. Then $\Delta$ is central in~$\MM$ and $\tta \dive \Delta$ holds. Hence, by Proposition~\ref{P:QuasiO}, $\MM$ is of right-$O$-type, hence of $O$-type as the opposite relations also are triangular (with the ordering of $\tta$ and $\ttb$ exchanged).
\end{proof}

\begin{coro}
\label{C:Fam6}
For $\qq, \rr \ge 0$, let $\GG$ be the group $\Gr{\ttx, \tty}{\ttx^{\rr+1} = (\tty\ttx^2 \tty)^{\qq+1}}$. Then, for $\rr = 0$ and~$1$, the group $\GG$ is orderable and, for every~$\pp$, the subsemigroup of~$\GG$ generated by~$\ttx\inv$, $\ttx\tty^{\pp+2}$, and $\ttx \tty \ttx^{-(\rr+1)}$ is the positive cone of a left-invariant ordering on~$\GG$ which is isolated in~$\LO\GG$.
\end{coro}

\begin{proof}
According to the principle stated before Proposition~\ref{P:Fam6}, the group defined by~\eqref{E:Fam6} admits the presentation $\Gr{\tta, \ttb}{\tta (\ttb\tta^{\pp+2})^\rr \ttb \tta = \ttb\tta^{\pp+2} (\ttb\tta^\pp \ttb \tta^{\pp+2})^\qq \ttb}$. Introducing $\ttx = \ttb\tta^{\pp+2}$, this becomes $\Gr{\tta, \ttb, \ttx}{\tta \ttx^\rr \ttb \tta = \ttx (\ttx \tta^{-2} \ttx)^\qq \ttb, \ttx = \ttb\tta^{\pp+2}}$, whence $\Gr{\tta, \ttx}{\tta \ttx^{\rr+1} \tta = \ttx (\ttx \tta^{-2} \ttx)^\qq \ttx}$. Putting $\tty = \tta\inv$, this becomes $\Gr{\ttx, \tty}{\ttx^{\rr+1} = \tty \ttx (\ttx \tty^2 \ttx)^\qq \ttx \tty}$, whence $\Gr{\ttx, \tty}{\ttx^{\rr+1} = (\tty\ttx^2 \tty)^{\qq+1}}$ by arranging the brackets. The result then follows from the converse relations $\tta = \ttx\inv$, $\ttb = \ttx\tty^{\pp+2}$, and $\ttc = \ttx \tty \ttx^{-(\rr+1)}$.
\end{proof}

For $\rr \ge 2$, it seems that the monoid~$\MM$ of Proposition~\ref{P:Fam6} is still of $O$-type: experiments suggest that, for $\pp = 0$ and $\rr = 2$, the right-ceiling of~$\MM$ is the periodic word $\linfty{(\ttb^2\tta^3)}$ for every~$\qq$ but, as no central or quasi-central element seems to exist, this remains a conjecture.

A few further families are displayed in Table~\ref{T:RecapBis}. We skip the verifications.

\begin{table}[h]
\smaller
\begin{tabular}{cclc}
\hline
\hline

\VR(4,0)
1:
&$\Pres{\tta, \ttb}{\tta = \ttb (\tta\ttb^\pp)^\qq \tta\ttb}$ 
&$\Delta = (\tta\ttb^{\pp-1})^2$ central
\hfill\llap{$\Gr{\ttx, \tty}{\ttx^{\qq+2}{=}\tty^2}$}\\

\VR(3.5,0)
2:
&$\Pres{\tta, \ttb}{\tta = \ttb\tta^{\rr}\ttb\tta^{\pp}\ttb\tta^{\rr}\ttb}$
&$\Delta = (\tta^{\rr}\ttb)^2$ right-quasi-central\\

&\hfill\llap{with $\pp+1 \,\vert \, \rr$\quad}
&\hspace{5mm}with $\phi(\tta) = \tta^{\pp} (\ttb \tta^{\rr} \ttb)^2$, $\phi(\ttb) = \ttb$\\

&&\hfill\llap{$\Gr{\ttx, \tty}{\ttx^\pp = (\tty \ttx^{\rr - \pp} \tty)^2}$}\\

\VR(3.5,0)
3:
&$\Pres{\tta, \ttb, \ttc}{\tta = \ttb\tta^\pp\ttb, \ttb = \ttc \ttb\tta^\rr  \ttc}$
&$\Delta = \tta^{\pp(\rr-\pp)+1}$ central\\

&\hfill\llap{with $\pp+1 \,\vert\, \rr+1$\quad}
&\hfill\llap{$\Gr{\ttx, \tty, \ttz}{\ttx^{\pp+1} = \tty^2, \tty = \ttz \tty\ttx^{\rr-\pp}  \ttz}$}\\

\VR(3.5,0)
4:&
$\Pres{\tta, \ttb, \ttc}{\tta = \ttb\tta^\pp\ttb, \ttb = \ttc \ttb\tta^\rr  \ttc}$
&$\Delta = \tta^\rr\ttb^2$ quasi-central\\

&\hfill\llap{with $\pp+1 \,\vert\, \rr$\quad}
&\hspace{5mm}with $\phi(\tta) = \tta^\pp \ttb \tta^{\pp-1} \ttb^3$, $\phi(\ttb) = \ttb$, $\phi(\ttc) = \ttc$;\\

&&\hfill\llap{$\Gr{\ttx, \tty, \ttz}{\ttx^{\pp+1} = \tty^2, \tty = \ttz \tty\ttx^{\rr-\pp} \ttz}$}\\

\VR(3.5,0)
5:
&$\Pres{\tta, \ttb, \ttc}{\tta = \ttb(\tta \ttb)^\pp, 
\ttb = \ttc \ttb (\tta^\rr \ttb)^\pp \ttc}$
&$\Delta = \tta^{\pp(\pp+1)(\rr-1)+2}$ central\\

\VR(0,3)
&\hfill\llap{with $\rr$ odd\quad}
&\hfill\llap{$\Gr{\ttx, \tty, \ttz}{\ttx^2 = (\ttx\tty)^{\pp+1}, \ttz\tty\ttz = \tty(\ttx^\rr\tty)^\qq}$}\\

\VR(3.5,0)
6:
&$(\tta, \ttb, \ttc \,; \tta = \ttb \tta^{\pp+1}(\ttb\tta^\pp\ttb\tta^{\pp+1})^\qq\ttc, ...$\hspace{1cm}\null
&$\Delta = (\tta^{\pp+1} \ttb)^{\rr+3}$ central\\

\VR(0,3)
&\hfill\llap{$\ttb = \ttc(\ttb\tta^{\pp+1})^\rr\ttb\tta)$\quad}
&\hfill\llap{$\Gr{\ttx, \tty}{\tty^{\rr+3} = \ttx^{\qq+2}}$}\\

\hline
\hline
\end{tabular}
\caption{\sf\smaller \VR(4,0)More examples of monoids of $O$-type, with the justification and an alternative presentation of the associated group; in~1, two-generator symmetric presentations with central elements that are not a power of~$\tta$; in~2 (a special case of Proposition~\ref{P:Fam4}), two-generators symmetric presentations with quasi-central elements that are not central (for $\pp \ge 1$); in~6, an example similar to Proposition~\ref{P:Fam6} (which corresponds to replacing $\pp+1$ with $\pp + 2$).}
\label{T:RecapBis}
\end{table} 


\section{Limits of the approach}
\label{S:Limits}

So far, we did not discuss the range of our approach, namely the question of whether many monoids of (right)-$O$-type admit (right)-triangular presentations. Owing to the positive results of Sections~\ref{S:Exp} and~\ref{S:Fam}, which provide a number of such monoids of $O$-type, it would even be conceivable that all monoids of $O$-type could admit a triangular presentation. In this section, we show that this is not the case, and give a simple criterion discarding a number of such monoids, in particular the $\nn$-strand Dubrovina--Dubrovin braid monoids for $\nn \ge 4$.

So our starting point is

\begin{ques}
\label{Q:Pres}
Assume that $\MM$ is a monoid of right-$O$-type and $\SS$ is a generating subfamily of~$\MM$. Does $\MM$ admit a right-triangular presentation based on~$\SS$?
\end{ques} 

What is significant in a right-triangular presentation is not the fact that it consists of triangular relations, but the condition that there is at most one letter~$\NN(\ss)$ and one relation $\NN(\ss) =\nobreak \ss \CC(\ss)$ for every~$\ss$: every positive presentation can be trivially transformed into a presentation of the same monoid consisting of triangular relations by introducing, for every relation~$\uu = \vv$, a new letter~$\ss$ and replacing $\uu = \vv$ with the triangular relations $\ss = \uu$, $\ss = \vv$. 

The following result, which is a special case of a result of~\cite{Dfx} for monoids in which any two elements admit a least common right-multiple, may appear relevant for Question~\ref{Q:Pres}.

\begin{fact}
\label{F:Noeth}
Assume that $\MM$ is a monoid of right-$O$-type that satisfies Condition~\eqref{E:Noeth}, and $\SS$ is any generating subfamily of~$\MM$. For all~$\ss, \ss'$ in~$\SS$ with $\ss \dive \ss'$, choose an $\SS$-word~$\ww$ such that $\ss\ww$ represents~$\ss'$. Let $\RR$ be the family of all relations~$\ss\ww = \ss'$ so obtained. Then $\Pres\SS\RR$ is a presentation of~$\MM$.
\end{fact}

\begin{proof}[Proof (sketch)]
We wish to prove for all $\SS$-words~$\uu, \vv$ that $\clp\uu = \clp\vv$ is equivalent to $\uu \eqpR \vv$. By construction, $\RR$ consists of relations that are valid in~$\MM$, hence $\uu \eqpR \vv$ always implies $\clp\uu = \clp\vv$, and the problem is the converse implication. Standard arguments show that \eqref{E:Noeth} is equivalent to the existence of a map~$\wit$ from~$\MM$ to the ordinals such that $\ss\not=1$ implies $\wit(\ss\gg) > \wit(\gg)$. Then one proves that $\clp\uu = \clp\vv$ with $\wit(\clp\uu) = \alpha$ implies $\uu \eqpR \vv$ using induction on~$\alpha$. For $\alpha = 0$, we have $\wit(\clp\uu) = \wit(\clp\vv) = 0$, hence $\clp\uu$ is minimum with respect to proper right-divisibility in~$\MM$, implying $\clp\uu = \clp\vv = 1$, whence $\uu = \vv = \ew$.
Assume now $\alpha > 0$. Then $\uu$ and $\vv$ cannot be empty. Write $\uu = \ss \uu_0$, $\vv = \ss' \vv_0$ with $\ss, \ss'$ in~$\SS$. Then, by definition, we have (*) $\wit(\clp{\uu_0}) < \wit(\clp\uu)$ and $\wit(\clp{\vv_0}) < \wit(\clp\vv)$.
Assume first $\ss' = \ss$. By assumption, we have $\clp\uu = \clp\vv$, \ie, $\ss \clp{\uu_0} = \ss \clp{\vv_0}$. As $\MM$ is left-cancellative, we deduce $\clp{\uu_0} = \clp{\vv_0}$. By~(*) and the induction hypothesis, this implies $\uu_0 \eqpR \vv_0$, whence {\it a fortiori} $\uu = \ss\uu_0 \eqpR \ss\vv_0 = \vv$. Finally, assume $\ss' \not= \ss$. In~$\MM$, the elements~$\ss$ and~$\ss'$ are comparable for~$\dive$, say for instance $\ss \dive \ss'$. Then, by construction, there exists in~$\RR$ one relation $\ss \ww = \ss'$ such that $\ss \clp\ww = \ss'$ holds in~$\MM$. We deduce $\ss \clp{\uu_0} = \ss' \clp{\vv_0} = \ss \clp{\ww} \clp{\vv_0} = \ss \clp{\ww\vv_0}$, whence $\clp{\uu_0} = \clp{\ww \vv_0}$ since $\MM$ is left-cancellative. By~(*) and the induction hypothesis, this implies $\uu_0 \eqpR \ww\vv_0$, whence $\uu = \ss\uu_0 \eqpR \ss\ww\vv_0 \eqpR \ss'\vv_0 = \vv$. So the induction is complete.
\end{proof}

The above positive result is misleading. The range of Fact~\ref{F:Noeth} is nonempty since it applies at least to the monoid~$(\NNNN, +)$, but, as already mentioned, the Noetherianity condition~\eqref{E:Noeth} fails in almost all monoids that admit triangular presentations, and the following example shows that, when \eqref{E:Noeth} fails, we cannot hope for a result similar to Fact~\ref{F:Noeth}. 

\begin{exam}
\label{X:Noeth}
Let $\MM$ be the Klein bottle monoid $\Mon{\tta, \ttb}{\tta = \ttb\tta\ttb}$. Then $\MM$ is of right-$O$-type, and it is generated by~$\tta$ and~$\ttb$. Now, in~$\MM$, we have $\tta = \ttb^2\tta\ttb^2$, so, if Fact~\ref{F:Noeth} were valid here, $\Pres{\tta, \ttb}{\tta = \ttb^2\tta\ttb^2}$ would be an alternative presentation of~$\MM$. This is not the case: by Lemma~\ref{L:Excluded} applied with $\uu = \ttb$ and $\vv = \ew$, monoid~$\Mon{\tta, \ttb}{\tta = \ttb^2\tta\ttb^2}$ is not of right-$O$-type and, therefore, it is not isomorphic to~$\MM$.
\end{exam}

Actually, we shall establish a rather general negative answer to Question~\ref{Q:Pres} in the case of generating families with at least three elements.

\begin{defi}
Assume that $\MM$ is a monoid and $\SS$ is included in~$\MM$. An element~$\ss$ of~$\SS$ is called \emph{preponderant in~$\SS$} if $\gg \dive \hh \ss$ holds for all~$\gg, \hh$ in the submonoid generated by~$\SS {\setminus} \{\ss\}$.
\end{defi}

\begin{prop}
\label{P:NoPres}
Assume that $\MM$ is a monoid of right-$O$-type and $\SS$ is a generating subfamily of~$\MM$ that contains a preponderant element and has at least three elements. Then $\MM$ admits no right-triangular presentation based on~$\SS$.
\end{prop}

\begin{proof}
We assume that $\MM$ admits a right-triangular presentation~$\Pres\SS\RR$ and shall derive a contradiction by exhibiting two elements of~$\MM$ that cannot admit a common right-multiple. 

As $\MM$ is of right-$O$-type, owing to Lemma~\ref{L:Enum}, we can enumerate~$\SS$ as $\{\tta_\ii \mid \ii \in \II\}$ so that all relations in~$\RR$ have the form $\tta_{\ii-1} = \tta_\ii \CC(\tta_\ii)$. Assume that $\tta_\ii$ is preponderant in~$\SS$. Then $\ii$ must be minimal in~$\II$ as $\tta_\jj \div \tta_\ii$ holds for every~$\jj \not= \ii$. So we may assume $\II = \{1, 2, ...\}$ (finite or infinite), and that $\tta_1$ is preponderant in~$\SS$.

As $\tta_1$ is preponderant in~$\SS$, it may occur in no word~$\CC(\tta_\ii)$ with $\ii \ge 3$ for, otherwise, writing $\CC(\tta_\ii) = \uu \tta_1 \vv$ with no $\tta_1$ in~$\uu$, applying the definition of preponderance with $\gg = \tta_{\ii-1}^2$ and $\hh = \clp{\tta_\ii \uu}$ would lead to the contradiction
$$\tta_{\ii-1} \div \tta_{\ii-1}^2 \dive \clp{\tta_\ii \uu \tta_1} \dive \clp{\tta_\ii \CC(\tta_\ii)} = \tta_{\ii-1}.$$
On the other hand, $\tta_1$ must occur in~$\CC(\tta_2)$ for, otherwise, we would obtain similarly the contradiction $\clp{\tta_2 \CC(\tta_2)} \div \clp{\tta_2 \CC(\tta_2) \tta_2} \dive \tta_1 = \clp{\tta_2 \CC(\tta_2)}$. Write $\tta_2\CC(\tta_2) = \uu_0 \tta_1 \vv_0$ with no~$\tta_1$ in~$\uu_0$. 

\begin{claim}
Assume that $\ww$ is an $\SS$-word that is $\eqpR$-equivalent to a word beginning with~$\tta_1$. Then $\ww$ contains at least one letter~$\tta_1$ and, if $\uu$ is the initial fragment of~$\ww$ that goes up to the first letter~$\tta_1$, there exists~$\rr \ge 0$ satisfying $\uu \eqpR \uu_0^\rr$.
\end{claim}

We prove the claim using induction on the combinatorial distance~$\nn$ of~$\ww$ to a word beginning with~$\tta_1$, \ie, on the length of an $\RR$-derivation from~$\ww$ to such a word. For $\nn = 0$, \ie, if $\ww$ begins with~$\tta_1$, the word~$\uu$ is empty, and we have $\uu = \ew = \nobreak \uu_0^0$. Assume $\nn > 0$. Let $\ww'$ be a word obtained from~$\ww$ by applying one relation of~$\RR$ that lies at distance $\nn-1$ from a word beginning with~$\tta_1$. By induction hypothesis, $\ww'$ contains at least one letter~$\tta_1$, and we have $\ww' = \uu' \tta_1 \vv'$ with no~$\tta_1$ in~$\uu'$ and $\uu' \eqpR \uu_0^{\rr'}$ for some~$\rr'$. We consider the various ways $\ww$ can be obtained from~$\ww'$. First, if one relation of~$\RR$ is applied inside~$\vv'$, we have $\ww = \uu' \tta_1 \vv$ with $\vv \eqpR \vv'$ and the result is clear with $\uu = \uu'$ and $\rr = \rr'$. Next, assume that the distinguished letter~$\tta_1$ is involved. By hypothesis, $\NN(\tta_1)$ is not defined, so there is no relation $\ss = \tta_1 \CC(\tta_1)$ in~$\RR$. On the other hand, $\uu'$ contains no~$\tta_1$ and, therefore, $\tta_1$ occurs in no relation $\ss = ...$ for $\ss$ occurring in~$\uu'$. So the only ways $\tta_1$ may be involved is either $\tta_1$ being replaced with~$\tta_2 \CC(\tta_2)$, or $\tta_2 \CC(\tta_2)$ (which contains at least one~$\tta_1$) being replaced with~$\tta_1$. In the first case, we obtain $\uu = \uu' \uu_0 \tta_1 \vv_0 \vv'$, which shows that $\ww$ contains a letter~$\tta_1$ and gives $\uu = \uu' \uu_0$, whence $\uu \eqpR \uu_0^{\rr' + 1}$, the expected result with $\rr = \rr'+1$. In the second case, there must exist decompositions $\uu' = \uu \uu_0$ and $\vv' = \vv_0 \vv$ so that we have $\ww' = \uu \uu_0 \tta_1 \vv_0 \vv$ and $\ww = \uu \tta_1 \vv$. Again $\ww$ contains~$\tta_1$, and we find now $\uu\uu_0 \eqpR \uu' \eqpR \uu_0^{\rr'}$, whence $\uu \eqpR \uu_0^{\rr-1}$ because, by assumption, $\MM$ is right-cancellative. This is again the expected result, this time with $\rr = \rr'-1$. Finally, it remains the case when one relation of~$\RR$ is applied inside~$\uu'$. In this case, we obtain $\ww = \uu \tta_1 \vv'$ with $\uu \eqpR \uu'$, whence $\uu \eqpR \uu' \eqpR \uu_0^{\rr'}$, and the result is clear with $\rr = \rr'$. So the proof of the claim is complete.

We shall now easily obtain a contradiction. Indeed, by construction, the word~$\uu_0$ begins with the letter~$\tta_2$, so $\tta_2 \dive \clp{\uu_0}$ holds. By assumption, $\tta_2$ and~$\tta_3$ are distinct, so $\CC(\tta_3)$ is nonempty, and we obtain $1 \div \tta_3 \div \tta_3 \clp{\CC(\tta_3)} = \tta_2 \dive \clp{\uu_0}$, so that $\tta_3 = \clp{\uu_0^\rr}$ fails for every~$\rr$. Then the above claim implies that no $\SS$-word beginning with~$\tta_3\tta_1$ may be $\eqpR$-equivalent to an $\SS$-word beginning with~$\tta_1$. In other words, the elements~$\tta_1$ and~$\tta_3 \tta_1$ cannot admit a common right-multiple in~$\MM$, contrary to the assumption that $\MM$ is of right-$O$-type.
\end{proof}

Proposition~\ref{P:NoPres} prevents a number of monoids of right-$O$-type from admitting a right-triangular presentation. 

\begin{coro}
Assume that $\MM$ is a monoid of right-$O$-type that is generated by~$\tta, \ttb, \ttc$ with $\tta \mult \ttb \mult \ttc$ and $\ttb, \ttc$ satisfying some relation $\ttb = \ttc \vv$ with no~$\tta$ in~$\vv$. Then, unless $\MM$ is generated by~$\ttb$ and~$\ttc$, there is no way to complete $\ttb = \ttc\vv$ with a relation $\tta = \ttb \uu$ so as to obtain a presentation of~$\MM$.
\end{coro}

\begin{proof}
For a contradiction, assume that $(\tta, \ttb, \ttc ; \tta = \ttb \uu, \ttb = \ttc \vv)$ is a presentation of~$\MM$. If there is no~$\tta$ in~$\uu$, the assumption that $\tta = \ttb \uu$ is valid in~$\MM$ implies that $\tta$ belongs to the submonoid generated by~$\ttb$ and~$\ttc$, so $\MM$ must be generated by~$\ttb$ and~$\ttc$.

Assume now that there is at least one~$\tta$ in~$\uu$. As $\tta$ does not occur in $\ttb = \ttc \vv$, a word containing~$\tta$ cannot be equivalent to a word not containing~$\tta$. This implies that $\tta$ is preponderant in~$\{\tta, \ttb, \ttc\}$. Indeed, assume that $\gg, \hh$ belong to the submonoid of~$\MM$ generated by~$\ttb$ and~$\ttc$. By the above remark, $\hh \tta \gg' = \gg$ is impossible, hence so is $\hh \tta \dive \gg$. As, by assumption, $\MM$ is of right-$O$-type, we deduce $\gg \dive \hh \tta$. Then Proposition~\ref{P:NoPres} gives the result.
\end{proof}

So, for instance, no right-triangular presentation made of $\ttb = \ttc \ttb \ttc$ (Klein bottle relation) or $\ttb = \ttc \ttb^2 \ttc$ (Dubrovina--Dubrovin braid relation) plus a relation of the form $\tta = \ttb ...$ may define a monoid of right-$O$-type. In the case of braids, we obtain the following general result.

\begin{coro}
\label{C:DD}
Let $\BDD\nn$ be the submonoid of the braid group~$\BB_\nn$ generated by $\ss_1 =\nobreak \sig1 \pdots \sig{\nn-1}$, $\ss_2 = (\sig2 \pdots \sig{\nn-1})\inv$, $\ss_3 = \sig3 \pdots \sig{\nn-1}$, ..., $\ss_{\nn-1} = \sigg{\nn-1}{(-1)^\nn}$. Then $\BDD\nn$ is a monoid of $O$-type that admits no right-triangular presentation based on $\{\ss_1 \wdots \ss_{\nn-1}\}$ for $\nn \ge 4$.
\end{coro}

\begin{proof}
That $\BDD\nn$ is of $O$-type was established by Dubrovina--Dubrovin in~\cite{DuD}. Now, as a braid that admits an expression containing at least one~$\sig1$ and no~$\siginv1$ cannot admit an expression with no~$\sigg1{\pm1}$~\cite{Dhr}, the generator~$\ss_1$ is preponderant in~$\{\ss_1 \wdots \ss_{\nn-1}\}$. Proposition~\ref{P:NoPres} implies that $\BDD\nn$ admits no triangular presentation based on $\{\ss_1 \wdots \ss_{\nn-1}\}$ for $\nn \ge\nobreak 4$.
\end{proof}

One can indeed convert the standard presentation of the braid group~$B_\nn$ into a presentation in terms of the generators~$\ss_1 \wdots \ss_{\nn-1}$ of Corollary~\ref{C:DD}. For instance, writing $\tta, \ttb,...$ for $\ss_1, \ss_2, ...$, one can check that $\BDD4$ admits the presentation
\begin{equation}
\label{E:Braid}
\Pres{\tta, \ttb, \ttc}{\tta = \ttb^2\tta^2\ttb\tta\ttb\tta^2\ttb^2, \ttb = \ttc\ttb^2\ttc, \tta\ttb\ttc = \ttc\tta\ttb},
\end{equation}
a triangular presentation augmented with a third, additional relation. But the triangular presentation made of the first two relations in~\eqref{E:Braid} is not a presentation of~$\BDD4$, nor of any monoid of $O$-type either.


\section{Further questions}
\label{S:Ques}

Although triangular presentations may seem to be extremely particular, it turns out that a number of monoids of $O$-type with such presentations exist---much more than was expected first. So the main question is to further explore the range of the approach and understand which ordered groups are eligible. In particular, we started here from a monoid viewpoint and did not address the question of starting from an (ordered) group and possibly finding a relevant monoid with a triangular presentation.
We now mention a few more specific problems.


\subsection*{Two-generator monoids of $O$-type}

Whereas describing all monoids of $O$-type is certainly out of reach, the particular case of two-generator monoids seems more accessible, and several natural questions arise.

First, we saw in Section~\ref{S:Limits} that some monoids of $O$-type admit no triangular presentation, but the argument of Proposition~\ref{P:NoPres} requires the existence of at least three generators.

\begin{ques}
\label{Q:NoPres}
Does every two-generator monoid of right-$O$-type admit a right-triangular presentation?
\end{ques} 

If $\MM$ is a monoid of right-$O$-type generated by~$\tta \mult \ttb$, some triangular relation $\tta = \ttb \ww$ is satisfied in~$\MM$. However, the word~$\ww$ is not unique and Question~\ref{Q:NoPres} asks in particular for a preferred choice of~$\ww$.
Here is a negative observation about what could be a natural approach. Let $(\ww_\nn)_{\nn \ge 1}$ be the shortlex-enumeration of $\{\tta, \ttb\}^*$. Using induction on~$\nn$, we can construct a triangular presentation $\Pres{\letter{\ww_1} \wdots \letter{\ww_\nn}}{\RR_\nn}$ that specifies the $\dive$-adjacent elements in $\{\clp{\ww_\ii} \mid \ii \le \nn\}$ in~$\MM$: assuming that the evaluation~$\clp{\ww_\nn}$ of~$\ww_\nn$ lies between~$\clp{\ww_\ii}$ and~$\clp{\ww_\jj}$ in $\{\clp{\ww_1} \wdots \clp{\ww_{\nn}}\}$, we obtain~$\RR_\nn$ from~$\RR_{\nn-1}$ by removing the former relation $\letter{\ww_\jj} = \letter{\ww_\ii} \cdot ...$ and adding two new relations $\letter{\ww_\nn} = \letter{\ww_\ii} \cdot \uu$, $\letter{\ww_\jj} = \letter{\ww_\nn} \cdot \vv$ with $\uu, \vv$ (say) shortlex-minimal. We could expect that, at least if a central element exists, the process converges and leads to a triangular presentation of~$\MM$. Let $\MM = \Mon{\tta, \ttb}{\tta = \ttb \tta \ttb}$. Then $\RR_3$ specifies $\letter{\ew} < \letter\ttb < \letter\tta$ and contains the relation $\letter{\tta} = \letter{\ttb} \cdot \tta\ttb$, whereas $\RR_4$ corresponds to inserting $\letter{\ttb^2}$ between~$\letter{\ttb}$ and~$\letter{\tta}$, so we remove $\letter{\tta} = \letter{\ttb} \cdot \tta\ttb$ and add $\letter{\ttb^2} = \letter{\ttb} \cdot \ttb$ and $\letter{\tta} = \letter{\ttb^2} \cdot \tta\ttb^2$. Now we observed in Example~\ref{X:Noeth} that $\letter{\tta} = \letter{\ttb} \cdot \tta\ttb$ is not a consequence of the latter two relations, so the process cannot converge to a presentation of~$\MM$, and Question~\ref{Q:NoPres} remains open.

Next, all identified monoids of \hbox{$O$-type} with two generators share some properties. 

\begin{ques}
\label{Q:Palindrome}
Is every two-generator triangular presentation defining a monoid of $O$-type necessarily palindromic, \ie,  the relation is invariant under reversing the order of letters? Is the right-ceiling necessarily equal to~$\linfty\ss$, where $\ss$ is the top generator?
\end{ques}

At the moment, we have no counter-example. 

More generally, we can wonder whether a complete description of all monoids of $O$-type that admit a two-generator triangular presentation could be possible. When stated in terms of two-generator presentation, Proposition~\ref{P:Fam5} may appear promising, but a complete solution still seems out of reach. The general form of the relation in a two-generator triangular presentation is $\tta = \ttb \tta^{\ee_1} \ttb \tta^{\ee_2} \ttb ...$. Proposition~\ref{P:Fam1} corresponds to the case when all exponents~$\ee_\ii$ are equal, whereas Proposition~\ref{P:Fam5} corresponds to the case when two exponents~$\pp, \rr$ occur with a periodic distribution
$\tta = \ttb (\tta^\rr \ttb)^\ss (\tta^\pp \ttb)
(\tta^\rr \ttb)^\ss (\tta^\pp \ttb)...$\,. Inductively extending Proposition~\ref{P:Fam5} seems doable---the next case would be that of three relations $\tta = \ttb (\tta^\pp \ttb)^\qq$, $\ttb = \ttc (\tta^\rr \ttc)^\ss$, $\ttc = \ttd (\tta^\tt \ttd)^\uu$ involving three exponents~$\pp, \rr, \tt$---even finding methods that cover all cases with two exponents seems problematic. We give below an example of such a monoid that is indeed of $O$-type but enters no identified family.

\begin{exam}
\label{X:Classif}
Let $\MM$ be the monoid defined by $\Pres{\tta, \ttb}{\tta = \ttb\tta\ttb\tta\ttb^2 \tta\ttb^2 \tta\ttb\tta\ttb}$, which is also defined by $\Pres{\tta, \ttb, \ttc}{\tta = \ttb\ttc \tta\ttc\ttb, \ttb = \ttc\tta\ttc\tta\ttc}$. Working with the second presentation and putting $\delta = \ttb^2$, one can establish the formulas 
$$\tta \delta^\nn \cdot \ttc\tta\ttc\ttb\ttc (\ttb\ttc)^{2\nn} = \ttb\delta^\nn \cdot \ttb = \ttc\delta^\nn \cdot \tta\ttc\tta\ttc\ttb = \delta^{\nn+1}.$$
It follows that $\delta$ dominates $\tta, \ttb$, and $\ttc$ in~$\MM$. Hence, by Proposition~\ref{P:DominO}, $\MM$ is of $O$-type. Experiments suggest that the right-ceiling is $\linfty\tta$, so that, by Lemma~\ref{L:CeilingDomin}, $\tta$ should dominate~$\tta, \ttb, \ttc$ as well, but we have no proof.
\end{exam}


\subsection*{Periodicity of the right-ceiling}

Many of the monoids mentioned in Sections~\ref{S:Exp} and~\ref{S:Fam} admit a central or a quasi-central element that is a power of the top generator, and we could wonder whether this is always the case. It is not.

\begin{exam}
\label{X:NoQuasi}
Consider the monoid~$\MM$ of Proposition~\ref{P:Fam5} with $\rr \ge 2$ and $\pp + 1 \!\not\vert\ \rr + 1$. We have seen in the proof that the relation $\ttc \tta^{\kk\mm+\rr} \dive \tta^{\kk\mm+1}$ holds for every~$\kk$. As $\mm \le \rr$ holds by definition, we deduce $\ttc \tta^{(\kk+1)\mm} \dive \tta^{\kk\mm+1}$ for each~$\kk$, which in turns implies $\ttc \tta^\nn \div \tta^\nn$ for every~$\nn$. So $\tta^\nn \dive \ttc \tta^\nn$ is always impossible, and $\tta^\nn$ is not quasi-central for any~$\nn$.
\end{exam}

This however says nothing about quasi-central elements that are not a power of the top generator, and we can raise

\begin{ques}
\label{Q:Quasi}
Does every monoid of $O$-type that admits a triangular presentation contain a quasi-central element?
\end{ques}

A positive answer seems unlikely but, on the other hand, it is uneasy to \textit{a priori} discard the existence of exotic quasi-central elements. In particular, looking at the right-ceiling is irrelevant in general: a quasi-central element~$\Delta$ dominates all generators but, if $\Delta$ admits no expression that is a right-top word, Lemma~\ref{L:CeilingDomin} cannot be used. For instance, for the presentation $\Pres{\tta\ttb}{\tta = \ttb\tta\ttb^3\tta\ttb}$ (row~5 in Table~\ref{T:Exp2}), the right-ceiling $\linfty\tta$ does not discard the existence of the central element~$(\tta\ttb)^3$. Owing to the explicit formulas in the proof of Proposition~\ref{P:Fam5}, we think it might be possible to prove that no nontrivial quasi-central element exists in the case of $\Pres{\tta, \ttb, \ttc}{\tta = \ttb \tta^2 \ttb, \ttb = \ttc \tta^4 \ttc}$.

The situation is also open for dominating elements. As for the top generator to necessarily dominate the other generators, the answer is negative. Indeed, according to Lemma~\ref{L:CeilingDomin}, the top generator~$\ss$ dominates the other generators if and only if the right-ceiling is~$\linfty\ss$. So, to discard a positive answer, it suffices to exhibit one presentation where the right-ceiling has a different form. The following example does, and, in addition, it shows that not only the top generator, but even any power of it need not dominate the other generators.

\begin{exam}
\label{X:NotDomin} 
Let $\MM$ be defined by the presentation
\begin{equation}
\label{E:DD}
\Mon{\tta, \ttb, \ttc}{\tta = \ttb \tta \ttc, \ttb = \ttc \ttb \tta}
\end{equation}
(case $\pp = \qq = \rr = 0$ in row~6 of Table~\ref{T:RecapBis}). One easily checks that $\ttb^2 \tta^2$ is right-quasi-central in~$\MM$, with $\phi(\tta) = \tta \ttb \tta^2 \ttc \tta \ttc^3$, $\phi(\ttb) = \ttb \tta^2 \ttc^2$, and $\phi(\ttc) = \ttc$, and that $\ttb^2\tta^2$ dominates $\{\tta, \ttb, \ttc\}^4$. By Lemma~\ref{L:CeilingDomin}, it follows that the right-ceiling is the periodic word $\linfty(\ttb^2\tta^2)$ and, therefore, $\tta$ cannot dominate~$\ttb$ and~$\ttc$ in~$\MM$. Moreover, we find $\ttb\tta^2 = \tta^3 \cdot \ttb\tta^2\ttc^3$, and $\ttb (\tta^\nn) = (\tta^\nn)^2 \cdot \ttb\tta^2\ttc^{2\nn}\tta^{\nn-2}$ for $\nn \ge 2$. This shows that $\tta^\nn$ does not dominate~$\ttb$ for any~$\nn$. 
\end{exam}

Now, as in the case of quasi-central elements, this says nothing about dominating elements that are not powers of the top generator and leaves the following natural questions open:

\begin{ques}
\label{Q:Periodic}
Does every monoid of $O$-type that admits a triangular presentation based on a set~$\SS$ contain an element that dominates~$\SS$?
Is the right-ceiling necessarily periodic in a monoid of $O$-type that admits a finite triangular presentation?
\end{ques} 

By Lemma~\ref{L:CeilingDomin}, a positive answer to the second question means the existence of an element of~$\SS^\nn$ that dominates all of~$\SS^\nn$ for some~$\nn \ge 1$, hence \textit{a fortiori}~$\SS$, so it implies a positive answer to the first question. Owing to the examples known so far, we conjecture a positive answer to both questions, but we have no clue toward a proof. 

We add two more related examples. The first one shows that a power of the top generator may dominate the other generators although the top generator does not.

\begin{exam}
\label{X:NonDomin}
Let $\MM$ be defined by $\Pres{\tta, \ttb, \ttc}{\tta = \ttb\ttc\ttb, \ttb = \ttc\ttb\tta\ttb\ttc}$. Then~$\MM$ is generated by~$\ttb$ and~$\ttc$, with the presentation $\Pres{\ttb, \ttc}{\ttb = \ttc\ttb^2\ttc\ttb^2\ttc}$. So, by Proposition~\ref{P:Fam1}, $\ttb^3$ is central and $\MM$ is of $O$-type. Now $\tta^3 = \ttb^3$ happens to hold and, therefore, $\tta^3$ dominates $\ttb$ and~$\ttc$. On the other hand, we find $\ttb\tta = \tta^2 \cdot \ttb\ttc^2\ttb$, whence $\ttb\tta \not\dive \tta^2$, and $\tta$ does not dominate~$\ttb$. 
\end{exam} 

The second one shows that the period of the right-ceiling can be arbitrarily large. 

\begin{exam}
\label{X:Periodic}
Let $\MM_\nn$ be defined by the cycling presentation
$$\Pres{\tta_1 \wdots \tta_\nn}{\tta_1 = \tta_2 \pdots \tta_\nn, \tta_2 = \tta_3 \pdots \tta_\nn \tta_1 \wdots \tta_{\nn-1} = \tta_\nn \tta_1 \pdots \tta_{\nn-1}}.$$
Then $\tta_1^2$ is central in~$\MM_\nn$, and the right-ceiling is $\linfty(\tta_{\nn-1} \pdots \tta_1)$, hence it has period $\nn-1$. 
\end{exam}

A similar behaviour can be found with three generators in the family of row~6 in Table~\ref{T:RecapBis}.


\subsection*{Complexity of reversing}

In the context of a presentation that is complete for right-reversing, the existence of common right-multiples implies the termination of every right-reversing. However, the argument gives no complexity bound, at least no polynomial bound.

\begin{exam}
\label{X:Expon}
Consider the presentation $\Pres{\tta, \ttb}{\tta = \ttb\tta\ttb^{\rr+1}}$ of the Baumslag-Solitar group $\BS(\rr+1,-1)$. For every~$\nn$, the signed word $\tta^{-\nn}\ttb\tta^\nn$ reverses to the word~$\ttb^{(\rr+1)^\nn}$, whose length is exponential in~$\nn$. As every reversing step adds at most $\rr$~letters, the number of steps needed to reverse the length~$2\nn+1$ word $\tta^{-\nn}\ttb\tta^\nn$ must be exponential  in~$\nn$.
 \end{exam}
 
An exponential complexity may occur whenever there exists a right-quasi-central element such that the associated endomorphism duplicates some letter. Now, in all examples, the monoid is not of left-$O$-type: this is the case in Example~\ref{X:Expon} or, for instance, for $\Pres{\tta, \ttb, \ttc}{\tta = \ttb\tta\ttc\ttb, \ttb = \ttc\tta\ttc}$, where~$\tta^2$ is right-quasi-central with $\phi(\ttc) = (\ttc\ttb)^2$ but the opposite presentation leads to the non-terminating reversing $\uu \Rev{12} \vv\inv \uu \vv$ for $\uu = \ttb\inv \ttc^2\tta\ttb$ and $\vv = \ttc^2$. By contrast, such behaviours could not be found for monoids of $O$-type.

\begin{ques}
\label{Q:Complexity}
If a triangular presentation defines a monoid of $O$-type, does the associated reversing necessarily have a polynomial (quadratic?) complexity?
\end{ques}

Note that the existence of a quasi-central element that is not central need not imply an exponential complexity. For instance, for the presentation $\Pres{\tta, \ttb}{\tta = \ttb\tta^2\ttb\tta\ttb\tta^2\ttb}$ (row~2 of Table~\ref{T:RecapBis}) with $\Delta = (\tta^2\ttb)^2$ quasi-central, we have $\phi(\tta) = \tta (\ttb\tta^2\ttb)^2$, and the shortest expression of~$\phi(\tta)$ is longer by 8~letters than that of~$\tta$. However, $\phi(\tta^2) = \tta^2$ holds  in the monoid, and reversing $\Delta^{-\nn} \tta \Delta^\nn$ leads to a word of length linear in~$\nn$ 
in a quadratic number of steps.
The monoid of Example~\ref{X:Classif} is a good test-case for Question~\ref{Q:Complexity}. It turns out that reversing the length~$2\nn$ word $\tta^{-(\nn-1)}\ttb\inv \tta^\nn$ leads to a word of length~$(2\nn)^2$ in a number of steps that is cubic in~$\nn$: this is compatible with a positive answer to Question~\ref{Q:Complexity}, but discards a uniform quadratic upper bound.


\subsection*{Connection with rewrite systems}

Triangular presentations are simple in many respects, and they could be eligible for alternative approaches, in particular rewrite systems.

\begin{exam}
\label{X:Rewrite}
Consider the presentation $\Pres{\tta, \ttb}{\tta = \ttb \tta \ttb}$ of the Klein bottle group. To obtain Noetherianity, we orientate the relation as $\ttb\tta\ttb \rightarrow \tta$. Then there is a  critical pair as $\ttb \tta \ttb \tta \ttb$ rewrites into~$\tta^2 \ttb$ and~$\ttb\tta^2$. Adding $\ttb\tta^2 \rightarrow \tta^2 \ttb$ yields a system that is locally confluent and Noetherian (the rules diminish the length or replace a word with a word of the same length and lexicographically smaller), hence confluent. Then every word rewrites in finitely many steps into a unique, well defined terminal word, providing a unique normal form for the elements of the monoid.
\end{exam}

The existence of a normal form as above gives an easy solution to the word problem of the monoid, and can in turn be extended to the group. However recognizing divisibility properties is not clear.

\begin{ques}
\label{Q:Rewrite}
Can one use the rewrite system approach to investigate the existence of common multiples in the associated monoid?
\end{ques}

We have no answer. Let us also observe that using (as above) the lexicographical ordering to solve critical pairs need not lead to a Noetherian system in general.

\begin{exam}
Consider the triangular presentation $\Pres{\tta, \ttb, \ttc}{\tta = \ttb\ttc, \ttb = \ttc \ttb^2}$. Starting with the two rules $\ttb \ttc \rightarrow \tta$, $\ttc \ttb^2 \rightarrow \ttb$, a critical pair comes from rewriting $\ttb \ttc \ttb^2$ into~$\tta \ttb^2$ and~$\ttb^2$. Using the lexicographical ordering would lead to adding the rule $\ttb^2  \rightarrow \tta \ttb^2$, which clearly contradicts Noetherianity.
\end{exam}


\subsection*{Isolated points in the case of a direct limit}

Theorems~1 and~\ref{T:Main} are valid in the case of an infinite presentation, thus leading to orderable groups with an explicit positive cone. But the argument showing that the involved ordering is isolated in its space of orderings is valid only when the presentation is finite. However, as observed by C.\,Rivas~\cite{Riv}, a non-finitely generated monoid may give rise to an isolated ordering, so it makes sense to raise

\begin{ques}
\label{Q:Infinite}
If $(\SS, \RR)$ is an infinite triangular presentation defining a monoid of $O$-type, may the associated ordering be isolated in the space~$\LO{\Gr\SS\RR}$?
\end{ques}

In the direction of a positive answer, it would be natural to address Question~\ref{Q:Infinite} in the context of a direct limit of finitely generated monoids. The properties of subword reversing make this situation easy to analyze.

\begin{prop}
\label{P:Limit}
Assume that $\Pres\SS\RR$ is an infinite triangular presentation \begin{equation}
\label{E:PresLimit}
\Pres{\tta_1, \tta_2, ...}{\tta_1 = \tta_2 \ww_2 \tta_2 , \tta_2 = \tta_3 \ww_3 \tta_3, ...}
\end{equation}
with $\ww_\ii$ in~$\{\tta_1 \wdots \tta_\ii\}^*$ and, putting $\SS_\nn = \{\tta_1 \wdots \tta_\nn\}$ and $\RR_\nn = \{\tta_{\ii-1} = \tta_\ii \ww_\ii \tta_\ii \mid \ii \le \nn\}$, that the monoid $\Mon{\SS_\nn}{\RR_\nn}$ is of $O$-type for every~$\nn$ (or, at least, for unboundedly many~$\nn$). Then $\Mon\SS\RR$ is a direct limit of the monoids~$\Mon{\SS_\nn}{\RR_\nn}$, it is of $O$-type, and $\Mon\SS\RR {\setminus} \{1\}$ is the positive cone of a left-invariant ordering on the group~$\Gr\SS\RR$.
\end{prop}

\begin{proof}
Assume $\nn < \mm$. Owing to the assumption about~$\ww_\ii$, the presentations~$\Pres{\SS_\nn}{\RR_\nn}$ and $\Pres{\SS_\mm}{\RR_\mm}$ are well defined and, by definition, they are right-triangular, so that, by Proposition~\ref{P:Complete}, $\Pres{\SS_\nn}{\widehat{\RR_\nn}}$ and $\Pres{\SS_\mm}{\widehat{\RR_\mm}}$ are complete for right-reversing. Now assume that $\uu, \vv$ are $\SS_\nn$-words. Then $\uu$ and~$\vv$ represent the same element in~$\Mon{\SS_\nn}{\RR_\nn}$ (\resp in $\Mon{\SS_\mm}{\RR_\mm}$) if and only if $\uu\inv \vv$ is $\widehat{\RR_\nn}$-reversible (\resp $\widehat{\RR_\mm}$-reversible) to the empty word. By definition of reversing, the relations in~$\RR_\mm\setminus \RR_\nn$ are never involved in the reversing of~$\uu\inv \vv$, so the latter two relations are both equivalent to $\uu\inv \vv$ being $\RR$-reversible to~$\ew$. It follows that the identity on $\SS_\nn$ induces an embedding of~$\Mon{\SS_\nn}{\RR_\nn}$ into~$\Mon{\SS_\mm}{\RR_\mm}$. So $\Mon{\SS_\nn}{\RR_\nn}$ identifies with the submonoid of $\Mon{\SS_\mm}{\RR_\mm}$  generated by $\SS_\nn$ and $\Mon\SS\RR$ is then the direct limit, here the union, of all monoids~$\Mon{\SS_\nn}{\RR_\nn}$.

It follows that $\Mon\SS\RR$ is of $O$-type. Indeed, a direct limit of monoids of right-$O$-type is of right-$O$-type: any two elements of the limit belong to some monoid of the considered direct system, hence are comparable with respect to left-divisibility in that monoid, and therefore in the limit. 
\end{proof}

The interest of Proposition~\ref{P:Limit} is to provide \emph{local} conditions for recognizing a monoid of $O$-type: in order to show that the monoid~$\Mon\SS\RR$ is, say, of right-$O$-type, it is sufficient to consider the finite type monoids~$\Mon{\SS_\nn}{\RR_\nn}$. A typical example is the group
\begin{equation}
\label{E:Infinite}
\Gr{\ttx_1, \ttx_2...}{\ttx_1^2 = \ttx_2^\qq, \ttx_2^2 = \ttx_3^\qq, ...}
\end{equation}
considered in Proposition~\ref{P:Fam4}. For $\qq = 2$, the element~$\Delta = \ttx_1^2$ is central in~$\Mon\SS\RR$, and Theorem~1 implies that $\Mon\SS\RR$ is of $O$-type. By contrast, for $\qq $ odd, the element $\Delta_\nn = \ttx_1^{2^{\nn-2}}$ is central in~$\Mon{\SS_\nn}{\RR_\nn}$, but not in  in~$\Mon{\SS_{\nn+1}}{\RR_{\nn+1}}$. It follows that $\Mon{\SS_\nn}{\RR_\nn}$ is of $O$-type for every~$\nn$, and $\Mon\SS\RR$ is of $O$-type by Proposition~\ref{P:Limit}, but, in this case, no power of~$\ttx_1$ is central in~$\Mon\SS\RR$. The multi-toric groups~\eqref{E:Infinite} are natural test-cases for Question~\ref{Q:Infinite}.


\subsection*{The specific case of braids}

The braid group~$B_3$ is eligible for our current approach, as the submonoid~$\BDD3$ generated by~$\sig1\sig2$ and~$\siginv2$ turns out to be a monoid of $O$-type with the triangular presentation $\Pres{\tta, \ttb}{\tta = \ttb\tta^2\ttb}$. By contrast, we saw in Section~\ref{S:Limits} that, for $\nn \ge 4$, the submonoid~$\BDD\nn$ of~$B_\nn$ generated by the Dubrovina-Dubrovin generators~$\ss_1 \wdots \ss_{\nn-1}$ admits no triangular presentation based on~$\{\ss_1 \wdots \ss_{\nn-1}\}$. This however does not discard the possibility that $\BDD\nn$ admits a triangular presentation based on other generators.

\begin{ques}
\label{Q:Braids}
Does the monoid~$\BDD\nn$ admit a (finite) triangular presentation for $\nn \ge 4$?
\end{ques}

Natural candidates could be the Birman--Ko-Lee band generators~\cite{BKL}. For $1 \le \ii < \jj \le\nobreak \nn$, put $\aa_{\ii, \jj} =  \sig\ii \pdots \sig{\jj-2} \sig{\jj-1} \siginv{\jj-2} \pdots \siginv\ii$, whence in particular $\sig\ii = \aa_{\ii, \ii+1}$. Then the family $(\aa_{\ii, \jj})_{1 \le \ii < \jj \le \nn}$ generates~$B_\nn$, and the submonoid of~$B_\nn$ generated by the~$\aa_{\ii,\jj}$'s, which is known as the dual braid monoid, has many nice properties. Now there exists a simple connection between the monoid~$\BDD\nn$ and the elements~$\aa_{\ii, \jj}$. 

\begin{prop}
\label{P:DD}
Put $\bb_{\ii, \jj} = \aa_{\ii, \jj}^{\parity{\ii+1}}$. Then, for every~$\nn$, the monoid~$\BDD\nn$ is generated by the elements~$\bb_{\ii, \jj}$.
\end{prop}

\begin{proof}
We recall from~\cite{Dhr} that a braid is called $\sig\ii$-positive (\resp negative) if it admits a decomposition in terms of the generators~$\sig\kk$ that contains no letter~$\sigg\kk{\pm1}$ with $\kk < \ii$, and contains at least one letter~$\sig\ii$ and no letter $\siginv\ii$ (\resp at least one letter~$\siginv\ii$ and no letter~$\sig\ii$). Then an element of~$\BB_\nn$ belongs to~$\BDD\nn$ if and only if it is either $\sig\ii$-positive for some odd~$\ii$ or $\sig\ii$-negative for some even~$\ii$. The braid relations imply $\aa_{\ii, \jj} = \sig{\jj-1} \pdots \sig{\ii+1} \sig\ii \siginv{\ii + 1} \pdots \siginv{\jj-1}$ for $\ii < \jj$, hence $\aa_{\ii, \jj}$ is $\sig\ii$-positive, and $\bb_{\ii, \jj}$ is $\sig\ii$-positive for odd~$\ii$ and $\sig\ii$-negative for even~$\ii$. Therefore, $\bb_{\ii, \jj}$ belongs to~$\BDD\nn$ for all~$\ii, \jj$.
Conversely, in~$B_\nn$, we have
$$\sig\ii \pdots \sig{\nn-1} = (\sig\ii \pdots \sig{\nn-2} \sig{\nn-1} \siginv{\nn-2} \pdots \siginv\ii)(\sig\ii \pdots \sig{\nn-3} \sig{\nn-2} \siginv{\nn-3} \pdots \siginv\ii) \pdots (\sig\ii \sig{\ii+1} \siginv\ii) (\sig\ii),$$
whence $\ss_\ii = (\sig\ii \pdots \sig{\nn-1})^{\parity{\ii+1}} = 
\begin{cases}
\bb_{\ii, \nn} \bb_{\ii, \nn-1} \pdots \bb_{\ii, \ii+1}
&\mbox{for odd~$\ii$},\\
\bb_{\ii, \ii+1} \pdots \bb_{\ii, \nn-1} \bb_{\ii, \nn}
&\mbox{for even~$\ii$}.
\end{cases}$ Hence $\ss_\ii$ belongs to the submonoid of~$B_\nn$ generated by the~$\bb_{\ii, \jj}$'s and, finally, $\BDD\nn$ coincides with the latter. 
\end{proof}

Proposition~\ref{P:DD} makes it natural to wonder whether $\BDD\nn$ admit a triangular presentation based on the elements~$\bb_{\ii, \jj}$, or on connected elements.
The answer is positive for $\nn = 3$. Indeed, starting from the standard presentation of~$B_3$ in terms of the~$\aa_{\ii, \jj}$'s, here $\sig1\sig2 = \sig2 \aa_{1,3} = \aa_{1,3}\sig1$, one deduces that $\BDD3$ admits the presentation $\Mon{\bb_{1,2}, \bb_{2,3}, \bb_{1,3}}{\bb_{1, 2} = \bb_{1,3} \bb_{1,2} \bb_{2,3}, \bb_{1,3} = \bb_{2,3} \bb_{1,3} \bb_{1,2}}$, which is~\eqref{E:DD} with $\tta, \ttb, \ttc$ interpreted as $\sig1$, $\aa_{1,3}$, and~$\siginv2$. The quasi-central element is then $\sig1\sigg22\sig1$ and the associated endomorphism is defined by $\phi(\sig1) = \sigg22\sig1\sigg2{-2}$ and $\phi(\sig2) = \sig2$.

For $\nn \ge 4$, the question remains open. It might be natural to replace, for even~$\ii$, the generators~$\aa_{\ii, \jj}\inv$ with the symmetric versions $\siginv\ii \pdots \siginv\jj\sig{\jj-1}\pdots \sig\ii$, \ie, to put the negative factors first. Then relations similar to those for~$\BDD3$ arise, typically $\tta = \ttb \tta\ttc$, $\ttb = \ttc\ttb\tta$, $\ttc = \ttd \ttc \tte$, $\ttd = \tte \ttd \ttc$ for~$\BDD4$, but the associated monoid is not~$\BDD4$ because $\tta\tte = \tte\tta$ is missing.

We conclude with an amusing application. For small values of~$\nn$, the connection of Proposition~\ref{P:DD} between the Dubrovina--Dubrovin ordering and the Birman--Ko--Lee generators of braid groups implies the existence of a braid ordering on~$\BB_\nn$ that is isolated (contrary to the Dehornoy ordering) and, at the same time, includes the positive braid monoid~$\BP\nn$ (contrary to the Dubrovina--Dubrovina ordering).

\begin{coro}
The subsemigroup generated by $\sig1$, $\sig2$, and $\sig1\siginv2\siginv1$ is the positive cone of an isolated left-invariant ordering on the braid group~$\BB_3$. The subsemigroup generated by $\sig1$, $\sig2$, $\sig3$, $\sig1\sig2\siginv1$, $\sig2\siginv3\siginv2$, and $\sig1\sig2\siginv3\siginv2\siginv1$ is the positive cone of an isolated left-invariant ordering on the braid group~$\BB_4$.
\end{coro}

\begin{proof}
Conjugation by~$\sig1 \pdots \sig\nn$ defines an order~$\nn$ automorphism~$\phi_\nn$ of~$B_\nn$ that rotates the Birman--Ko--Lee generators~\cite{Dhr}. For $\nn = 3$, one has $\phi_3 : \sig1 \mapsto \sig2 \mapsto \aa_{1,3} \mapsto\nobreak \sig1$. By Propos\-ition~\ref{P:DD}, $\BDD3$ is generated by~$\sig1$, $\aa_{1,3}$, $\siginv2$. Hence the monoid~$\phi_3(\BDD3)$ is generated by~$\sig2$, $\sig1$, $\aa_{1,3}\inv$, and it is (when $1$ is removed) the positive cone of a left-invariant ordering on~$B_3$. 

Similarly, for $\nn = 4$, we have $\phi_4 : \sig1 \mapsto \sig2 \mapsto \sig3 \mapsto \aa_{1,4} \mapsto \sig1$ and $\aa_{1,3} \leftrightarrow \aa_{2,4}$. By Proposition~\ref{P:DD}, the monoid~$\BDD4$ is generated by~$\sig1$, $\aa_{1,3}$, $\aa_{1,4}$, $\siginv2$, $\aa_{2,4}\inv$, $\sig3$. Hence $\phi_4^2(\BDD4)$ is generated by~$\sig3$, $\aa_{1,3}$, $\sig2$, $\sig3$, $\aa_{1,4}\inv$, $\aa_{2,4}\inv$, $\sig1$, and it is (with $1$ removed) the positive cone of a left-invariant ordering.
\end{proof}

By the above remarks, the monoid~$\phi_3(\BDD3)$ admits the presentation~\eqref{E:DD} with $\tta = \sig2$, $\ttb = \sig1$, $\ttc = \aa_{1,3}\inv = \sig1 \siginv2 \siginv1$. The construction does not extend to $\nn \ge 5$, because the negative entries~$\siginv2$ and~$\siginv4$ cannot be eliminated simultaneously.



\begin{thebibliography}{99}

\def\Reff#1; #2; #3; #4; #5; #6; #7\par{%
\bibitem{#1} #2, \emph{#3}, #4 {\bf #5} (#6) #7}

\def\Ref#1; #2; #3; #4\par{%
\bibitem{#1} #2, \emph{#3}, #4}

\Ref Adj; S.I. Adyan; Defining relations and
algorithmic problems for groups and semigroups; Proc.
Steklov Inst. Math., {\bf 85} (1966).

\Reff BKL; J.\,Birman, K.H.\,Ko \& S.J.\,Lee; A new approach
to the word problem in the braid groups; Advances in
Math.; 139-2; 1998; 322-353.
   
\Ref ClP; A.H.\,Clifford \& G.B.\,Preston; The algebraic Theory of Semigroups, vol.~1; Amer. Math. Soc. Surveys {\bf 7}, (1961).

\Reff DDHPV; M.\,Dabkovska, M.\,Dabkowski, V.\, Harizanov, J.\,Przytycki, and M.\,Veve; Compactness of the space of left orders; J. Knot Theory Ramifications; 16; 2007; 267--256.

\Reff Dfa; P.\,Dehornoy; Deux propri\'et\'es des groupes de tresses; C. R. Acad. Sci. Paris; 315; 1992; 633--638.

\Reff Dff; P.\,Dehornoy; Groups with a complemented presentation; J. Pure Appl. Algebra; 116; 1997; 115--137.

\Reff Dgp; P.\,Dehornoy; Complete positive group presentations; J.  Algebra; 268; 2003; 156--197.

\Reff Dia; P.\,Dehornoy; The subword reversing method; Intern. J. Alg. and Comput.; 21; 2011; 71--118

\Ref Garside; P.\,Dehornoy, with F.\,Digne, E.\,Godelle, D.\,Krammer, J.\,Michel; Garside Theory; in progress; http://www.math.unicaen.fr/$\sim$garside/Garside.pdf.

\Ref Dhr; P.\,Dehornoy, with I.\,Dynnikov, D.\,Rolfsen, B.\,Wiest; Ordering Braids; Mathematical Surveys and Monographs vol. 148, Amer. Math. Soc. (2008).

\Reff Dfx; P.\,Dehornoy \& L. Paris; Gaussian
groups and Garside groups, two generalizations of Artin
groups; Proc. London Math. Soc.; 79-3; 1999; 569--604.

\Reff DuD; T.\,Dubrovina and N.\,Dubrovin; On braid groups; Sb. Math.; 192; 2001; 693-703.

\Ref Ito; T.\,Ito; Dehornoy-like left-orderings and isolated left-orderings; arXiv:1103.4669.

\Ref Ito2; T.\,Ito; Construction of isolated left-orderings via partially central cyclic amalgamation; arXiv:1107.0545.

\Ref KKM; A.I.\,Kokorin, V.M.\,Kopyutov, and N.Ya.\,Medvedev; Right-Ordered Groups; Plenum Publishing Corporation (1996).

\Reff Nav; A.\,Navas; On the dynamics of (left) orderable groups; Ann. Inst. Fourier; 60; 2010; 1685--1740.

\Reff Nav2; A.\,Navas; A remarkable family of left-ordered groups: central extensions of Hecke groups; J.
Algebra; 328; 2011; 31-Ð42.

\Ref Pic; M.\,Picantin; Petits groupes gaussiens; Th\`ese
de doctorat, Universit\'e de Caen (2000).

\Reff Rem;  J.H.\,Remmers; On the geometry of positive presentations; Advances in Math.; 36; 1980; 283--296.

\Ref Riv; C.\,Rivas; Left-orderings on free products of groups; J. Algebra; 350; 2012; 318--329.

\Reff Sik; A.S.\,Sikora; Ceilingology on the spaces of orderings of groups; Bull. London Math. Soc.; 36; 2004; 519--526.

\Reff Tar; V.\,Tararin; On the theory of right orderable groups;  Matem. Zametki; 54; 1993; 96--98. Translation to english in Math. Notes 54 (1994), 833--834.

\end{thebibliography}
\end{document}